\documentclass[11pt]{article}
\usepackage[margin=1in]{geometry}
\usepackage[utf8]{inputenc}

\usepackage{amsthm,amsmath,amsfonts,amssymb}
\usepackage{braket}            % Dirac notation
\usepackage{hyperref}
\usepackage{tikz}
\usetikzlibrary{chains, positioning, shapes.symbols,shapes,arrows}
\usepackage[ruled,vlined]{algorithm2e}

\usepackage[square, numbers, comma, sort&compress]{natbib}

\newtheorem{theorem}{Theorem}[section]
\newtheorem{remark}[theorem]{Remark}

\newtheorem{assumption}{Assumption}
\newtheorem{definition}[theorem]{Definition}

\newtheorem{proposition}[theorem]{Proposition}
\newtheorem{lemma}[theorem]{Lemma}

\newcommand{\tr}{\mathrm{t}\mathrm{r}} %trace Operator

\title{Graphon Quantum Filtering Systems}

\author{
  Hamed Amini\thanks{University of Florida. Email: \texttt{aminil@ufl.edu}} \and
  Nina H. Amini\thanks{CNRS, L2S, CentraleSupélec, Université Paris-Saclay. Email: \texttt{nina.amini@centralesupelec.fr}} \and
  Sofiane Chalal\thanks{L2S, CentraleSupélec, Université Paris-Saclay. Email: \texttt{sofiane.chalal@centralesupelec.fr}} \and
  Gaoyue Guo\thanks{MICS and CNRS FR-3487, CentraleSupélec, Université Paris-Saclay. Email: \texttt{gaoyue.guo@centralesupelec.fr}}
}

\date{} % leave empty or provide a date

\begin{document}

\maketitle

\begin{abstract}
We consider a non-exchangeable system of interacting quantum particles with mean-field type interactions, subject to continuous measurement on dense graphs. In the mean-field limit, we derive a graphon-based quantum filtering system, establish its well-posedness, and prove propagation of chaos for multi-class bosonic systems with blockwise interactions. We then discuss applications to quantum state preparation and quantum graphon games.



\end{abstract}

\setcounter{tocdepth}{1}
\tableofcontents

\section{Introduction}\label{sec:intro}

{
The study of large systems of interacting quantum particles is fundamental for understanding collective behavior in microscopic and mesoscopic regimes. As the number of particles increases, directly solving the many-body Schrödinger equation becomes infeasible due to the exponential growth of the underlying Hilbert space. Mean-field limits provide an effective alternative by replacing full interacting dynamics with equations that capture averaged interactions \cite{spohn80kinetic,erdHos2010derivation,lewin14derivation}.
Establishing rigorous connections between exact many-body dynamics and their mean-field approximations is a central problem in mathematical physics. Beyond their conceptual importance, mean-field models increasingly support emerging quantum technologies, guiding the control of large-scale quantum simulators, informing the development of efficient quantum algorithms, and providing theoretical foundations for scalable quantum architectures.

Most existing mean-field results focus on closed quantum systems and deterministic macroscopic observables \cite{tindall22,pearce1981mean}. In contrast, many physical platforms are open, continuously monitored, and subject to measurement backaction~\cite{wiseman09quantum,handel05}, making the development of a mean-field theory at the level of quantum trajectories both conceptually important and technologically motivated. A further source of complexity comes from the interaction structure. While most of the mean-field literature assumes homogeneous or exchangeable interactions, many real systems exhibit heterogeneous coupling patterns \cite{tindall22}. From a mathematical standpoint, this setting is considerably more challenging: the mean-field limit for non-exchangeable quantum systems has not yet been rigorously established, and the limiting description leads to a continuous family of coupled operator-valued stochastic differential equations whose well-posedness requires new probabilistic tools.

In this paper, we develop a formal framework for systems of quantum particles subjected to continuous measurement and their mean-field graph limits. Such models naturally give rise to dynamic quantum games and, by extension, quantum mean-field games (QMFGs) \cite{kolokoltsov22dynamic,kolokoltsov22qmfg}, but their relevance extends far beyond game-theoretic applications.

In standard closed quantum mechanics, the dichotomy between the state evolution (Schrödinger equation) and the measurement postulate (Born rule) prevents the observation of significant dynamics, a phenomena known as the Zeno effect \cite{misra77zeno}. To overcome this, one must consider open quantum systems subjected to indirect measurement, where the natural way to incorporate continuous measurements is provided by the quantum filtering framework \cite{belavkin03eventum,boutenhandel07,boutenhandel09}. The a posteriori state evolution is then described by a a stochastic modification of the Schrödinger equation, known as the Belavkin or quantum filtering equation. Depending on the nature of the measurement scheme, two types of quantum filtering equations arise:

\begin{itemize}
    \item Diffusive equation (homodyne detection): 
    {{\small} \begin{align*}
    \mathrm{d}\rho_t &= -\mathrm{i}[\tilde{H},\rho_t]\mathrm{d}t + \left(L\rho_t L^{\dagger} - \tfrac{1}{2}\{L^{\dagger}L,\rho_t\}\right)\mathrm{d}t + \sqrt{\eta}\left(L\rho_t + \rho_tL^{\dagger} - \tr\big((L + L^{\dagger})\rho_t\big)\rho_t\right)\mathrm{d}W_t,
    \end{align*}}
    where $(W_t)_{t\geq0}$ is a Wiener process (see Section~\ref{sec:QF} for the meaning of the other terms),  with the observation process
            {{\small} \begin{align*}
            Y_t &= W_t + \sqrt{\eta}\int_{0}^{t}\tr\big((L+L^{\dagger})\rho_{s}\big)\mathrm{d}s.
            \end{align*}}
    
    \item Jump equation (photon counting): 
    {{\small} \begin{align*}
    \mathrm{d}\rho_{t} = -\mathrm{i}[\tilde{H},\rho_t]\mathrm{d}t + \left(L\rho_t L^{\dagger} - \frac{1}{2}\{L^{\dagger}L,\rho_t\}\right)\mathrm{d}t + \left(\frac{L\rho_tL^{\dagger}}{\tr(L\rho_tL^{\dagger})} - \rho_t\right)\big(\mathrm{d}{N}_t - {\eta}\,\tr(L\rho_tL^{\dagger})\mathrm{d}t\big),
    \end{align*}}
    where $({N}_t)_{t \geq 0}$ is a counting process whose stochastic intensity is given by   $\int_{0}^{t}\tr(L\rho_sL^{\dagger})\mathrm{d}s$ and the observation process is $Y_t = N_t.$
        \end{itemize}

The quantum filtering equations can be derived using quantum probability and quantum stochastic calculus~\cite{parthasarathy12,boutenhandel07}. However, they may also be obtained within standard quantum mechanics as suitable continuous-time limits of repeated measurement schemes~\cite{pellegrini10markov,pellegrini08diffusive,pellegrini10jump}, or through the theory of quantum instruments~\cite{BB91,bar95const}. In this paper we focus only on the diffusive case. %and we will specify the significance of the terms later.

In his seminal series of papers, Kolokoltsov \cite{kolokoltsov21law,kolokoltsov22qmfg} developed the quantum mean-field games using the $N$-body quantum filtering equation as a starting point:
{{\small} \begin{align*}
\mathrm{d}\boldsymbol{\rho}_t^{N} &= -\mathrm{i}{\Big[\sum_{l=1 }^{N}\mathbf{\tilde{H}}_{l} + \frac{1}{N}\sum_{l=1}^{N}\sum_{l' > l}\mathbf{A}_{ll'},\boldsymbol{\rho}_{t}^{N}\Big]}\mathrm{d}t
+ \sum_{l=1}^{N}\Big({\bf L}_l\boldsymbol{\rho}_{t}^{N}{\bf L}^{\dagger}_{l} - \frac{1}{2}\big\{{\bf L}^{\dagger}_{l}{\bf L}_{l},\boldsymbol{\rho}_{t}^{N}\big\}\Big)\mathrm{d}t\nonumber\\
&\quad +\sum_{l = 1}^{N}\Big({\bf L}_l\boldsymbol{\rho}_{t}^{N} + \boldsymbol{\rho}_{t}^{N}{\bf L}^{\dagger}_l - \tr\big(({\bf L}_l + {\bf L}_l^{\dagger})\boldsymbol{\rho}_{t}^{N}\big)\boldsymbol{\rho}_{t}^{N}\Big)\mathrm{d}W_t^{l},
\end{align*}}
where $\boldsymbol{\rho}_{t}^{N}$ denotes the state of the $N$-particle system evolving on a high-dimensional tensor-product Hilbert space. In contrast with the classical setting, due to entanglement between quantum particles, we cannot separate the individual dynamics of the particles, and the concept of an empirical measure becomes meaningless in the quantum setting. Therefore, one must adopt a different methodology to establish propagation of chaos results. Kolokoltsov adapted the Pickl method \cite{pickl11simple} to the stochastic setting to derive the corresponding limiting equation
{{\small}
\begin{align*}
\mathrm{d}\gamma_t&=-\mathrm{i}[\tilde{H}+A^{\mathbb{E}_{\mathbb{P}}[\gamma_t]},\gamma_t]\mathrm{d}t+(L\gamma_tL^{\dagger}-\tfrac12\{L^{\dagger}L,\gamma_t\})\mathrm{d}t+\big(L\gamma_t+\gamma_tL^{\dagger}-\mathrm{tr}\!\big((L+L^{\dagger})\gamma_t\big)\gamma_t\big)\mathrm{d}W_t.
\end{align*}}
This leads to a new McKean–Vlasov type equation formulated on an infinite-dimensional complex-valued Hilbert space. The analysis in \cite{kolokoltsov21law,kolokoltsov22qmfg} addresses only homogeneous interactions. The present work develops the corresponding theory for heterogeneous interactions, allowing the coupling strengths between particles to vary.
%\medskip
}

\paragraph*{Primary contributions.}
{We first develop a mathematical framework for continuously observed, interacting quantum systems with heterogeneous coupling and establish their mean-field graph limits. Using graphon theory, we derive the graphon-based quantum filtering system
{{\small}
\begin{align*}
\mathrm{d}\gamma_{t,u}
&= -\mathrm{i}\Bigl[\tilde{H}
  + \int_{0}^{1}w(u,v)A^{\mathbb{E}_{\mathbb{P}}[\gamma_{t,v}]}\mathrm{d}v,\gamma_{t,u}\Bigr]\mathrm{d}t
  + \Bigl(L\gamma_{t,u}L^{\dagger}
    - \tfrac{1}{2}\{L^{\dagger}L,\gamma_{t,u}\}\Bigr)\mathrm{d}t\\
&\quad + \sqrt{\eta}\Bigl(L\gamma_{t,u}
    + \gamma_{t,u}L^{\dagger}
    - \mathrm{tr}\bigl((L+L^{\dagger})\gamma_{t,u}\bigr)\gamma_{t,u}\Bigr)\mathrm{d}W^u_t,
\quad u \in I := [0,1],
\end{align*}}
where $\mathbf{W}=\{W^u : u\in I \}$ is a collection of essentially pairwise independent Wiener processes, and $w$ is a graphon function {{\small} $w : I\times I \to I$}, indicating the strength of the pairwise interaction. Establishing the well-posedness of this system requires addressing some obstacles that do not appear in previous cases dealing with Belavkin diffusive equation.}

{The first difficulty lies in identifying an appropriate probability space on which to formulate our system of stochastic differential equations. Since the system is driven by a continuum family of independent Wiener processes, serious measurability issues arise. In fact, it is well known in the literature that a continuum family of independent Wiener processes cannot be constructed on a standard probability space while preserving joint measurability (see, for instance, \cite{judd85law} or \cite[Section 3.7]{carmonadelarue18I}). To overcome this difficulty, we specify a probability space that ensures the joint measurability of $\mathbf{W}$ via a Fubini extension \cite{sun06exact} which extended the usual product space.}

{The second difficulty is to identify  an appropriate valued space. Indeed, although each $\gamma_{t,u}$ is a trace-class operator, the space of trace-class operators $\mathcal{B}_1(\mathbb{H})$ forms a Banach space where Itô theory is ill-defined. Extensions of stochastic calculus via UMD theory~\cite{UMD15} (the framework for Banach spaces with unconditional martingale differences) do not apply here, as the space of trace-class operators does not belong to this category. Consequently, following the approach of~\cite{kolokoltsov25quantumstoeq}, we work within the larger space of Hilbert--Schmidt operators $\mathcal{B}_2(\mathbb{H})$. Equipped with the appropriate inner product, this forms the Liouville space $\mathcal{L}_2(\mathbb{H})$, which is a Hilbert space~\cite{lonigro24liouville}. This allows us to invoke established results for Hilbert space-valued SDEs. The proof then proceeds as follows. Since the trace operation is unbounded on the Liouville space, standard Lipschitz arguments cannot be applied directly. To overcome this, we parametrize the law dependence by a deterministic function and linearize the system to obtain a strong solution via standard arguments for Hilbert space-valued SDEs. We then apply a version  of Girsanov transformation on the Fubini extension space to construct a weak solution for parametrized non-linear system.  
For the delicate issue of positivity, we extend to the infinite-dimensional Hilbert space the argument of \cite{mirrahimiHandel07} which consists in dealing with the complete observed evolution and using conditional expectation to project into the imperfect observed equation. We then prove pathwise uniqueness and invoke the Hilbert space formulation of the Yamada–Watanabe theorem to obtain a strong solution. Finally, to complete the well-posedness result, we establish stability result with respect to the underlying graphon function.} 

{A second major objective of the paper is to establish the trajectory-level mean-field limit for non-exchangeable open quantum systems. To the best of our knowledge, such limits have not been studied in the literature. We consider the block-wise graphon case under perfect measurement efficiency $(\eta=1)$, where the limiting system can be described by a representative trajectory for each block. We extend Pickl’s method by introducing a weighted functional that quantifies the total number of “bad’’ particles across all blocks. This provides, for the first time, a propagation-of-chaos type result for heterogeneous open quantum systems, providing an initial step toward a general theory of mean-field limits in non-exchangeable quantum settings.

Finally, we discuss applications. We first examine quantum state preparation through continuous measurement and feedback, showing that graphon quantum filtering significantly reduces the complexity of feedback-based stabilization in many-body systems. In particular, each particle’s state converges to an eigenstate of the measurement operator, and appropriate feedback laws stabilize the population toward any desired eigenstate. We then introduce the concept of a graphon quantum dynamic game, demonstrating how the limiting model extends classical mean-field game ideas. The paper concludes with several perspectives and directions for future research.}

\paragraph*{Related literature.}
The study of mean-field limits and the associated notion of propagation of chaos for exchangeable interacting particle systems dates back to the seminal work of Kac \cite{kac56}. Since then, a vast literature has developed; see \cite{chaintron22I, chaintron22II} or \cite{mischler13kac} for comprehensive reviews.
A growing area of interest is the mean-field limit for non-exchangeable interacting particle systems. Using graphon theory, many works have established mean-field limits for dense graphs (see \cite{coppini22note, ayi24large, bet24weakly, bayraktar23graphon}) and for sparse graphs (see \cite{crucianelli24interacting, jabin25mean}). Within the graphon framework, extensions of classical mean-field games have also been proposed, see \cite{amini23graphon, caines19graphon, lacker22label, aurell22sto, parise23graphon}.
In the quantum setting, the notion of quantum propagation of chaos for exchangeable systems dates back to Spohn \cite{spohn80kinetic}, and substantial literature has followed, see, e.g., \cite{alicki83nonlinear, bardos00weak, lewin14derivation, pickl11simple, porat24pickl, kolokoltsov21law}.
{From a physics perspective and in the static case, \cite{tindall22, searle24thermo} examined graphon-type interaction structures in spin closed systems, but without addressing the question of the justification of mean-field limit. The present work can be viewed as a first step in that direction.}

%\medskip
\paragraph*{Organization.} 
The paper is organized as follows.
In Section~\ref{sec:pre}, we recall the main notions from quantum filtering theory, graphon theory, and the Fubini extension.
Section~\ref{sec:particles} presents the formalism of non-exchangeable observed interacting quantum particles.
In Section~\ref{sec:MF}, we introduce the limiting graphon quantum filtering system in a heuristic manner and then study its well-posedness.
Section~\ref{sec:Chaos} is devoted to the rigorous derivation of this system for blockwise interactions under perfect measurement efficiency.
Section~\ref{sec:Application} discusses applications to quantum state preparation and graphon-based quantum dynamic games.
Section~\ref{sec:concludes} concludes the paper with some perspectives for future research.
Appendix~\ref{sec:appA} discusses the principle of indistinguishability in quantum mechanics, while Appendix~\ref{sec:appB} contains technical results concerning the linear and nonlinear equations.

%\medskip
\paragraph*{Some notations and facts.}

The sets of natural, integer, real, and complex numbers are denoted by $\mathbb{N}, \mathbb{Z}, \mathbb{R}, \mathbb{C}$, respectively.
For any $z \in \mathbb{C}$, $\overline{z}$ denotes its complex conjugate, and the roman symbol $\mathrm{i}$ is used for the pure imaginary unit.
We write $\mathbb{I}_A$ for the indicator of a set $A$.
For a probability measure $\nu$, $\mathbb{E}_{\nu}$ denotes expectation under $\nu$. For $k \in \mathbb{N}$, we use the shorthand $[k] := \{1,\dots,k\}$ and denote by $\mathfrak{S}_k$ the set of all permutations of $[k]$.

Let $T > 0$ be a fixed finite time horizon. For a Banach space $E$, let $\mathfrak{B}(E)$ denote its Borel $\sigma$-field. We define $\mathcal{C}_{E} := \mathcal{C}([0,T],E)$ as the space of continuous functions from $[0,T]$ to $E$, endowed with the topology of uniform convergence. 

For a measure space $(J,\Sigma,\mathrm{r})$, define
$$
L^2_{E}(J,\mathrm{r})
:=\Big\{ X:J\to E \mid X \;\text{is }\Sigma\text{--}\mathfrak{B}(E)\text{-measurable and }
\int_J \|X(s)\|_E^2\mathrm{r}(\mathrm{d}s)<\infty \Big\}.
$$
As usual, we identify elements of $L^2_{E}(J,\mathrm{r})$
that are equal $\mathrm{r}$-almost everywhere (a.e.).

We fix the notation $\mathbb{H}$ for a complex separable Hilbert space endowed with the scalar product $\langle \cdot,\cdot\rangle : \mathbb{H}\times \mathbb{H} \to \mathbb{C}$, with associated norm $\|\psi \| := \sqrt{\langle \psi,\psi\rangle}$. We use both the notation $\psi$ and the “ket” $\ket{\psi}$ for a vector in $\mathbb{H}$, and the “bra” $\bra{\psi}$ for an element of the dual space $\mathbb{H}^{\star} \simeq \mathbb{H}$. Thus, $\braket{\psi|\psi'} := \braket{\psi,\psi'}$ is the inner product between ${\psi}$ and ${\psi'}$.
We define $\mathcal{B}(\mathbb{H})$ as the space of bounded operators with norm 
$$ \|O \|_{\text{op}} := \sup_{\|\psi\|\leq 1} \|O\psi\| < \infty, $$
and we often omit the subscript “op”  writing simply $\|O\|$. We denote by $O^{\dagger}$ the adjoint of $O \in \mathcal{B}(\mathbb{H})$, and we write $\langle O\rangle_{\psi} := \bra{\psi}O\ket{\psi}$.
We write $\mathbf{1}$ for the identity operator on the underlying Hilbert space in use (e.g., $\mathbb{H}$, $\mathbb{H}^{\otimes N}$, etc.). 

The commutator and anti-commutator are defined as follows: for $O_1, O_2 \in \mathcal{B}(\mathbb{H})$,
{{\small} \begin{align*}
    [O_1,O_2] := O_1O_2 - O_2O_1 \; \text{ and } \; \{O_1,O_2\} := O_1O_2 + O_2O_1.
\end{align*}}
For any operator $O \in \mathcal{B}(\mathbb{H})$ and $l \in \{1,\dots,N\}$, we denote by
$$
\mathbf{O}_{l} := \mathbf{1}\otimes\dots\otimes{O}\otimes\dots\otimes\mathbf{1},
$$
the operator in $\mathcal{B}(\mathbb{H}^{\otimes N})$ that acts as $O$ on the $l$-th tensor factor and as the identity on all the others (i.e., an operator on the Hilbert space $\mathbb{H}^{\otimes N}$).

For any operator $B \in \mathcal{B}(\mathbb{H}\otimes\mathbb{H})$ and $l,l' \in \{1,\dots,N \}$, we denote by $\mathbf{B}_{ll'}$ the operator on $\mathcal{B}(\mathbb{H}^{\otimes N})$ that acts only on $\mathbb{H}_{l}$ and $\mathbb{H}_{l'}$.
We denote by $\mathcal{B}_{1}(\mathbb{H})$ the space of trace-class operators and by $\mathcal{B}_{2}(\mathbb{H})$ the space of Hilbert–Schmidt operators: 
{{\small}\begin{align*}
    O \in \mathcal{B}_1(\mathbb{H}) \Longleftrightarrow \tr(|O|) < \infty,\
    O \in \mathcal{B}_2(\mathbb{H}) \Longleftrightarrow \tr(OO^{\dagger}) < \infty.
\end{align*}
}

The Liouville space $\mathcal{L}_{2}(\mathbb{H})$ is the Hilbert space of all Hilbert–Schmidt operators on $\mathbb{H}$ endowed with the scalar product  
{{\small} \begin{align*}
 \langle O_1, O_2 \rangle_{2} := \sum_{n \in \mathbb{N}} \langle O_1\psi_n,O_2\psi_n\rangle,
\end{align*}}
for all $O_1,O_{2} \in \mathcal{B}_{2}(\mathbb{H})$, where $\{\psi_n\}$ is an orthonormal basis of $\mathbb{H}$.

In quantum mechanics, a state is described by an operator in the space of density operators $\mathcal{S}(\mathbb{H})$, defined as
{{\small} \begin{align*}
    \mathcal{S}(\mathbb{H}) := \Big\{ \rho \in \mathcal{B}_{1}(\mathbb{H})\; \Big| \; \rho \geq 0, \; \tr(\rho) = 1, \; \rho= \rho^{\dagger}  \Big\}.
\end{align*}}
An important special case of a quantum state is a pure state (or wavefunction state), which represents a state of maximal knowledge. A state $\rho \in \mathcal{S}(\mathbb{H})$ is said to be pure if it can be written as a rank-one projector, i.e., there exists a vector $\psi \in \mathbb{H}$ such that $\rho = \ket{\psi}\bra{\psi}$.

%\medskip
Throughout the paper, we use $C$ to denote various constants and $C(p)$ to emphasize dependence on some parameter $p$. Their values may change from line to line.

\section{Preliminaries}\label{sec:pre}

{In this section we recall main notions and results from quantum filtering theory, graphon theory, and the Fubini extension.}

\subsection{Quantum filtering}\label{sec:QF}
The theory of quantum filtering, with its fundamental principle of non-demolition measurement, was pioneered by Belavkin \cite{belavkin83theory,belavkin87non,belavkin92quantum}. 
It provides a rigorous framework for the real-time observation and control of quantum dynamical systems without destroying their state. 
For a comprehensive introduction to the theory, we refer the reader to \cite{boutenhandel07,boutenhandel09,gough18introduction}. 

In the case of homodyne measurement, the evolution of the state is described by a diffusive stochastic differential equation, known as the Belavkin equation:
{{\small} \begin{align}\label{belavkinhomodyne} \mathrm{d}\rho_t &= -\mathrm{i}[\tilde{H},\rho_t]\mathrm{d}t + \left(L\rho_t L^{\dagger} - \tfrac{1}{2}\{L^{\dagger}L,\rho_t\}\right)\mathrm{d}t + \sqrt{\eta}\left(L\rho_t + \rho_tL^{\dagger} - \tr\big((L + L^{\dagger})\rho_t\big)\rho_t\right)\mathrm{d}W_t,\nonumber\\ \rho_{0} &\in \mathcal{S}(\mathbb{H}), \end{align}}
where $(W_t)_{t \geq 0}$ is a standard Wiener process on a probability space $(\Omega,\mathcal{F},\mathbb{P})$. 
Here $\tilde{H}=\tilde{H}^{\dagger}\in\mathcal{B}(\mathbb{H})$ is the Hamiltonian of the system, $L\in\mathcal{B}(\mathbb{H})$ is the measurement operator. 
The scalar $\eta \in (0,1]$ denotes the measurement efficiency. 
The process $(\rho_t)_{t \geq 0}$ takes values in the space of density operators $\mathcal{S}(\mathbb{H})$. 
The corresponding measurement record associated with $L$ is given by
{{\small} \begin{align*}
    Y_t &= W_t + \sqrt{\eta}\int_{0}^{t}\mathrm{tr}\big((L+L^{\dagger})\rho_{s}\big)\mathrm{d}s.
\end{align*}}
In the filtering approach, the process $(\rho_{t})_{t \geq 0}$ represents the best estimate of the system’s quantum state given the measurement record $(Y_s)_{0 \leq s \leq t}$.
%\medskip
\begin{remark}The Belavkin equation can also be formulated in a more general setting involving the simultaneous continuous observation of several physical observables, corresponding to a family of measurement operators $(L_{1},\dots,L_{M})$ and an $M$-dimensional Wiener process ${W}=(W^{1},\dots,W^{M})$. In this case, the equation takes the form
{{\small} \begin{align*}
    \mathrm{d}\rho_t =& -\mathrm{i}[\tilde{H},\rho_t]\mathrm{d}t + \sum_{l=1}^{M}\left(L_l\rho_t L_l^{\dagger} - \frac{1}{2}\{L^{\dagger}_lL_l,\rho_t\}\right)\mathrm{d}t \\ 
    &+ \sqrt{\eta}\sum_{l=1}^{M}\left(L_l\rho_t + \rho_tL_l^{\dagger} - \tr \big((L_l + L^{\dagger}_l)\rho_t\big)\rho_t\right)\mathrm{d}W^{l}_t.\nonumber    
\end{align*}}
\end{remark}

The main interest of these models lies in the fact that they allow for the incorporation of quantum feedback control (Markovian feedback in the terminology of mean-field games). Indeed, as in stochastic control with partial information, the separation theorem \cite{bensoussan92stochastic} makes it possible to decouple the state estimation problem from the optimal control problem. The theory of quantum filtering then reduces the analysis to the controlled version of \eqref{belavkinhomodyne}, where the Hamiltonian part is modified by introducing a real control parameter $\alpha$ in the set of admissible controls $\mathcal{U}$; so that the equation becomes (see, e.g., \cite{handel05,doherty00quantum}):  
{{\small}\begin{align}\label{belavkinhomodyne}
\mathrm{d}\rho_t&=-\mathrm{i}[\tilde{H}+\alpha_t\hat{H},\rho_t]\mathrm{d}t+(L\rho_t L^{\dagger}-\tfrac12\{L^{\dagger}L,\rho_t\})\mathrm{d}t
+\sqrt{\eta}(L\rho_t+\rho_tL^{\dagger}-\mathrm{tr}\!\big((L+L^{\dagger})\rho_t\big)\rho_t)\mathrm{d}W_t,
\end{align}}
where $\hat{H} = \hat{H}^{\dagger} \in \mathcal{B}(\mathbb{H})$ is the control Hamiltonian.

%\medskip

We can now formulate the optimal control problem given $\alpha$ by fixing an objective function
{{\small} \begin{align*} 
\mathcal{J}\Big((\alpha_s)_{0 \leq s \leq T},\rho_0\Big) := \mathbb{E}_{\mathbb{P}}\Big[\int_{0}^{T}C(\alpha_s,\rho_s)\mathrm{d}s + F(\rho_T) \Big],
\end{align*}}
where $C$ and $F$, representing respectively the running and terminal costs, are assumed to be bounded and differentiable. The objective is to find $\alpha^{\star}$ that minimizes $\mathcal{J}$. 
The associated Hamilton--Jacobi--Bellman equation has been studied in  \cite{gough05hamilton,belavkin09qdpp,belavkin09cybernetics}.

%\medskip

The following proposition addresses an important point concerning the wave-function representation of the state dynamics.  

\begin{proposition}\label{puritypreservation}
If the measurement efficiency is perfect, i.e., $\eta = 1$, and the initial state is pure $(\rho_0 = \ket{\psi_0}\bra{\psi_0})$, then purity is preserved and the dynamics \eqref{belavkinhomodyne} admit the following pure-state representation: 
{{\small}
\begin{align}\label{stoschrodinger}
\mathrm{d}\psi_t &= -\Big(\mathrm{i}\Big(\tilde{H} -  \langle \tfrac{L + L^{\dagger}}{2}\rangle_{\psi_t}\tfrac{L - L^{\dagger}}{2\mathrm{i}}\Big) + \tfrac{1}{2}\Big((L - \langle \tfrac{L + L^{\dagger}}{2}\rangle_{\psi_t})^{\dagger}(L - \langle \tfrac{L + L^{\dagger}}{2}\rangle_{\psi_t})\Big) \Big)\psi_t\mathrm{d}t\nonumber\\
&\qquad\quad + \Big( L -\langle \tfrac{L + L^{\dagger}}{2}\rangle_{\psi_t} \Big)\psi_t\mathrm{d}W_t,  
\end{align}
}
or, equivalently,
{{\small}
\begin{align*}
\mathrm{d}\psi_t &= -\Big(\mathrm{i}\tilde{H} + \tfrac{1}{2}L^{\dagger}L - \tfrac{\langle L + L^{\dagger}\rangle_{\psi_t}}{2}L + \tfrac{\langle L + L^{\dagger}\rangle_{\psi_t}^{2}}{8} \Big)\psi_t\mathrm{d}t 
+ \Big( L -\langle \tfrac{L + L^{\dagger}}{2}\rangle_{\psi_t} \Big)\psi_t\mathrm{d}W_t. 
\end{align*}
}

\end{proposition}
\noindent This pure-state formulation is commonly referred to as the stochastic Schrödinger equation.

{
\begin{proof}
    Let's denote the drift (resp. diffusive) part of the equation \eqref{stoschrodinger} by $\mathrm{F}_t$ (resp. $\mathrm{G}_t$). 
    Then by Itô's rule,
{{\small} \begin{align*}
\mathrm{d}(\ket{\psi_t}\bra{\psi_t})
&= (\mathrm{d}\ket{\psi_t})\bra{\psi_t} 
   + \ket{\psi_t}\mathrm{d}\bra{\psi_t}
   + (\mathrm{d}\ket{\psi_t})(\mathrm{d}\bra{\psi_t})\\
&= \Big(\mathrm{F}_t \ket{\psi_t}\mathrm{d}t + \mathrm{G}_t \ket{\psi_t}\mathrm{d}W_t\Big)\bra{\psi_t} 
+ \ket{\psi_t}\Big(\bra{\psi_t} \mathrm{F}_t^{\dagger} \mathrm{d}t 
   + \bra{\psi_t} \mathrm{G}_t^{\dagger} \mathrm{d}W_t\Big) \\
&\quad + \Big(\mathrm{F}_t \ket{\psi_t}\mathrm{d}t 
   + \mathrm{G}_t \ket{\psi_t}\mathrm{d}W_t\Big)
   \Big(\bra{\psi_t} \mathrm{F}_t^{\dagger}\mathrm{d}t + \bra{\psi_t} \mathrm{G}_t^\dagger \mathrm{d}W_t\Big)\\
&= \Big( \mathrm{F}_t \ket{\psi_t}\bra{\psi_t} 
   + \ket{\psi_t}\bra{\psi_t} \mathrm{F}_t^{\dagger} 
   + \mathrm{G}_t \ket{\psi_t}\bra{\psi_t} \mathrm{G}_t^{\dagger} \Big)\mathrm{d}t \\
&\quad + \Big( \mathrm{G}_t\ket{\psi_t}\bra{\psi_t} 
   + \ket{\psi_t}\bra{\psi_t} \mathrm{G}_t^{\dagger} \Big)\mathrm{d}W_t,
\end{align*}}
and,
{{\small} 
\begin{align*}
\mathrm{F}_t\ket{\psi_t}\bra{\psi_t} + \ket{\psi_t}\bra{\psi_t}\mathrm{F}_t^{\dagger} 
   &= -\mathrm{i}[\tilde{H},\ket{\psi_t}\bra{\psi_t}] 
   - \tfrac{1}{2}\{L^{\dagger} L,\ket{\psi_t}\bra{\psi_t}\} \\
   &\quad + \tfrac{1}{2}\langle L+L^{\dagger}\rangle_{\psi_t}\big(L\ket{\psi_t}\bra{\psi_t} 
   + \ket{\psi_t}\bra{\psi_t}L^{\dagger}\big) 
   - \tfrac{1}{4}\langle L+L^{\dagger}\rangle_{\psi_t}^2\ket{\psi_t}\bra{\psi_t},
\end{align*}}
where the commutator with $\tilde{H}$ comes from the Hamiltonian part of 
$\mathrm{F}_t$. Since $\tilde{H}$ is self-adjoint, we have
{{\small} \begin{align*}
(-\mathrm{i}\tilde{H})\ket{\psi_t}\bra{\psi_t} 
+ \ket{\psi_t}\bra{\psi_t}(\mathrm{i}\tilde{H})
= -\mathrm{i}[\tilde{H},\ket{\psi_t}\bra{\psi_t}],
\end{align*}}
and,
{{\small} 
\begin{align*}
\mathrm{G}_t\ket{\psi_t}\bra{\psi_t}\mathrm{G}_t^{\dagger} 
   &= L\ket{\psi_t}\bra{\psi_t}L^{\dagger} 
   - \tfrac{1}{2}\langle L+L^{\dagger}\rangle_{\psi_t}\big(L\ket{\psi_t}\bra{\psi_t} 
   + \ket{\psi_t}\bra{\psi_t}L^{\dagger}\big) \\
   &\quad + \tfrac{1}{4}\langle L+L^{\dagger}\rangle_{\psi_t}^2 \ket{\psi_t}\bra{\psi_t}.
\end{align*}}
Adding the two contributions gives
{{\small} \begin{align*}
\mathrm{F}_t\ket{\psi_t}\bra{\psi_t} 
   \!+\! \ket{\psi_t}\bra{\psi_t}\mathrm{F}_t^{\dagger} 
   \!+\! \mathrm{G}_t\ket{\psi_t}\bra{\psi_t}\mathrm{G}_t^{\dagger} 
\!&=\! -\mathrm{i}[\tilde{H},\ket{\psi_t}\bra{\psi_t}]
   \!-\! \tfrac{1}{2}\{L^{\dagger}L,\ket{\psi_t}\bra{\psi_t}\}
   \!+\! L\ket{\psi_t}\bra{\psi_t}L^{\dagger},
\end{align*}}
and for the diffusion term we have
{{\small} \begin{align*}
\mathrm{G}_t\ket{\psi_t}\bra{\psi_t} 
   + \ket{\psi_t}\bra{\psi_t} \mathrm{G}_t^{\dagger} 
   &= L\ket{\psi_t}\bra{\psi_t} + \ket{\psi_t}\bra{\psi_t}L^{\dagger} 
   - \langle L+L^{\dagger}\rangle_{\psi_t}\ket{\psi_t}\bra{\psi_t}.
\end{align*}}
Therefore, denoting $\rho_t = \ket{\psi_t}\bra{\psi_t}$, the process satisfies
{{\small} \begin{align*}
\mathrm{d}\rho_t
&= -\mathrm{i}[\tilde{H},\rho_t]\,\mathrm{d}t 
   + \Big(L\rho_t L^{\dagger} - \tfrac{1}{2}\{L^{\dagger}L,\rho_t\}\Big)\mathrm{d}t + \Big(L\rho_t + \rho_t L^{\dagger} - \langle L+L^{\dagger}\rangle_{\psi_t}\rho_t\Big)\mathrm{d}W_t,
\end{align*}}
where
{{\small} \begin{align*}
\langle L+L^{\dagger}\rangle_{\psi_t} = \tr\!\left((L+L^{\dagger})\rho_t\right).
\end{align*}}
This completes the proof.
\end{proof}}

\subsection{Graphon theory}
We give a brief summary of the notion of graphon theory relevant to our work, see \cite{lovasz12large} for further details. 

Let $I := [0,1]$. The set $I$ is an index set labeling a continuum of particles. The Lebesgue space over $I$ is denoted by $(I,\mathfrak{B}_I,\lambda_I)$ where $\lambda_{I}(\mathrm{d}u)  :=\mathrm{d}u $ is the Lebesgue measure. 
Denote by $\mathcal{W}$ the space of all bounded symmetric measurable functions $w : I^2 \rightarrow \mathbb{R}$. The elements of $\mathcal{W}$ will be called kernels. 

Let $\mathcal{W}_0$ be the subset of $\mathcal{W}$ such that $0 \leq w \leq 1$, where elements of $\mathcal{W}_0$ will be called graphons. 
{Intuitively, graphons generalize graphs to a continuum of vertices: $w(u,v)$ represents the weight of the edge between $u$ and $v$. These objects are ideal for studying convergent sequences of dense graphs.}

For a graphon $w$ we can define the operator $ \mathcal{T}_w : L^{\infty}_{}(I) \to L^1_{}(I)$ as follows: for any $ \phi \in L^{\infty}_{}(I)$,
$$ \mathcal{T}_{w}(\phi)(u) := \int_{I}^{}w(u,v)\phi(v)\mathrm{d}v.$$

The operator $\mathcal{T}_{w}$ can be viewed as an operator from $\mathcal{T}_{w} : L^{2}_{}(I) \to L^{2}_{}(I)$ then it is a Hilbert-Schmidt operator and the corresponding  theory for such operator can be applied.

We define some norm associated with kernel element $w \in \mathcal{W}$. The cut norm is defined by
$$ \|w\|_{\square} := \sup_{\mathrm{B}_1,\mathrm{B}_2 \in \mathfrak{B}_I}\Big|\int_{\mathrm{B}_{1}\times\mathrm{B}_{2}}w(u,v)\mathrm{d}u\mathrm{d}v\Big|,$$
where this optima is attained and it is equal to 
$$\sup_{\phi_{1},\phi_{2} : I \to I}\Big|\int_{I^2}\phi_{1}(u)\phi_{2}(v)w(u,v)\mathrm{d}u\mathrm{d}v\Big|.$$
The operator norm of $\mathcal{T}_{w}$ associated to kernel $w \in \mathcal{W}$ is given by
\begin{align*}
\|\mathcal{T}_{w}\|_{\text{op}} := \sup_{\|\phi\|_{} \leq 1}\int_{I}\Big|\int_{I}w(u,v)\phi(v)\mathrm{d}v\Big|\mathrm{d}u.
\end{align*}
The $L^1$-norm is defined as 
\begin{align*}
    \|w\|_{L^1} &:= {\int_{I}^{}\int_{I}^{}|w(u,v)|^{}\mathrm{d}v\mathrm{d}u}.
\end{align*}

%\end{definition}

According to \cite[lemma 8.11]{lovasz12large}, for every kernel $w \in \mathcal{W}$, the operator norm and the cut norm are equivalent, 
\begin{align}\label{lovasz}
\|w\|_{\square} \leq \|\mathcal{T}_{w}\|_{\text{op}} \leq 4\|w\|_{\square}.
\end{align}

%\begin{lemma}\cite[lemma 8.11]{lovasz12large}
%\label{lovasz}
%For every kernel $w \in \mathcal{W}$, we have equivalence between operator norm and cut norm, 
%$$\|w\|_{\square} \leq \|\mathcal{T}_{w}\|_{\text{op}} %\leq 4\|w\|_{\square}$$
%\end{lemma}

One can establish a natural relationship between the adjacency matrix of a graph and graphon function through the notion of a step kernel.

\begin{definition}[Step Kernel]
Given a finite graph $G_N = (V_{N},E_{N})$, where $V_N = \{\tfrac{1}{N},\dots,\tfrac{N}{N}\}$ with adjacency matrix $(\xi_{pq}^N)_{p, q \in V_N}$, the associated step kernel 
 $w^{G_N} : [0,1]^2 \rightarrow \mathbb{R}$ is defined by partitioning the unit interval $[0,1]$ into subintervals
  $I_p = \left(p-\frac{1}{N}, p\right]$ for $p \in V_N$, such that
{{\small} \begin{align*}
    w^{G_N}(u,v) = \xi_{pq}^N, \quad \text{for } (u,v) \in I_p \times I_q.
\end{align*}}
\end{definition}

This definition of a step kernel allows us to study the limiting behavior of a sequence of graphs using the cut metric. A sequence of dense graphs $(G_N)_{N \geq 0}$ is said to converge to a graphon $w$ if their associated step kernels converge to $w$ in the cut norm.

\subsection{Fubini extension theory}

To resolve the measurability difficulties that arises when constructing a continuum family of pairwise independents stochastic processes, the notion of Fubini extension was introduced in \cite{sun06exact}. Indeed, one may ask whether a pairwise independent family $ \Theta := \Big\{ \theta_{u} : u \in I \Big\} $ can  be constructed on a probability space $ (\Omega,\mathcal{F},\mathbb{Q}) $ such that the process $ \Theta \colon I\times\Omega \ni (u,\omega)\mapsto \Theta(u,\omega) = \theta_{u}(\omega) $ satisfies joint measurability property. To achieve this, such processes must be constructed on extensions of the standard product space $ (I\times\Omega,\mathfrak{B}_I\otimes\mathcal{F},\lambda_{\mathcal{I}}\otimes\mathbb{Q})$.

We follow a similar approach to \cite{aurell22sto,coppini25nonlinear,crucianelli24interacting,amini25brownian} to define stochastic processes on the appropriate Fubini extension space, see e.g., \cite[Section 3.7]{carmonadelarue18I} for an introduction to the formalism of Fubini extensions. {We begin by recalling an  important notion due to Sun \cite{sun98almost} which shows that all the notions of independence are, in fact, almost identical to their pairwise counterparts in an ideal setting of a continuum of random variables.}

\begin{definition}[Essential pairwise independence]
    Let $(I', \mathcal{I}', \lambda')$ and $(\Omega', \mathcal{F}', \mathbb{Q}')$ be two probability spaces, and let $E$ be a Banach space . A random variable $\Theta : I' \times \Omega' \to E$, is said to be \emph{essentially pairwise independent} (e.p.i.) if for $\lambda'$-almost every $u \in I'$, the random variables $\Theta(u,.)$ and $\Theta(v,.)$ are independent for $\lambda'$-almost every $v \in I'$. 
\end{definition}

We now turn to the notion of a Fubini extension, which provides the appropriate probability space for hosting a jointly measurable e.p.i. family. This extension enlarges the classical product space in such a way that both independence and Fubini’s theorem are preserved, and it serves as the foundational framework for all the results developed in this paper.

\begin{definition}[Fubini extension] Let $(I',\mathcal{I}',\lambda')$ and
$(\Omega',\mathcal{F}',\mathbb{Q}')$ be two probability spaces.
A \emph{triple} $(I'\times\Omega' ,\mathcal{V}',\mathcal{Q}')$ is called a
Fubini extension of the product space $(I'\times\Omega',\mathcal{I}'\otimes\mathcal{F}',\lambda'\otimes\mathbb{Q}')$ if for any Banach space $E$-valued $\mathcal{Q}'$-integrable function $g$ on $(I'\times \Omega',\mathcal{V}')$ :
\begin{itemize}
\item the functions $g_u : \omega \mapsto g(u,\omega)$ and $g_{\omega} : u \mapsto g(u,\omega)$ are integrable on $(\Omega', \mathcal{F}',\mathbb{Q}') $ for $\lambda'-$a.e. $u \in I'$, and on $(I',\mathcal{I}',\lambda')$ for $\mathbb{Q}'$-a.e. $\omega \in \Omega'$, respectively. 
\item the functions $u \mapsto \int_{\Omega'}g_u(\omega)\mathbb{Q}'(\mathrm{d}\omega)$ and $\omega \mapsto \int_{I'}g_{\omega}(u)\lambda'(\mathrm{d}u)$   are integrable, respectively, on $(I',\mathcal{I}',\lambda')$ and  $(\Omega', \mathcal{F}',\mathbb{Q}') $, with 
{{\small} \begin{align*}\int_{I'\times\Omega'}g(u,\omega)\mathcal{Q}'(\mathrm{d}u,\mathrm{d}\omega)  &= \int_{I'}\Bigl(\int_{\Omega'}g_u(\omega)\mathbb{Q}'(\mathrm{d}\omega)\Bigr)\lambda'(\mathrm{d}u)           =\int_{\Omega'}\Bigl(\int_{I'} g_{\omega}(u)\lambda'(\mathrm{d}u)\Bigr)\mathbb{Q}'(\mathrm{d}\omega).
\end{align*}}
\end{itemize}
\end{definition}

Although it is common in the literature to denote the Fubini extension $(I'\times \Omega', \mathcal{V}', \mathcal{Q}')$ by $(I' \times \Omega', \mathcal{I}' \boxtimes \mathcal{F}', \lambda' \boxtimes \mathbb{Q}')$, we  avoid this notation here for clarity.

The following result from \cite{sun09individual} guarantees the existence of Fubini extension space carrying a collection of e.p.i. jointly measurable random variables. 

\begin{theorem}{\cite[Theorem $1$]{sun09individual}}
There exists a probability space $(I,\mathcal{I},\lambda)$ extending $(I,\mathfrak{B}_{I},\lambda_{I})$, a probability space $(\Omega,\mathcal{F},\mathbb{Q})$, and a Fubini extension 
$ (I \times \Omega, \mathcal{V}, \mathcal{Q})
$ of the usual product space $
(I\times\Omega, \mathcal{I}\otimes\mathcal{F}, \lambda\otimes\mathbb{Q})$
such that for any measurable mapping 
$$
\varphi : I \to \mathrm{P}(S),
$$
where $S$ is a Polish space and $\mathrm{P}(S)$ denotes the set of Borel probability measures on $S$, there exists a $\mathcal{V}$-measurable process,
$$
\Theta : I\times\Omega \to S,
$$
such that the random variables $\Theta(u,.)$ are e.p.i. and 
$$
\mathbb{Q}\circ (\Theta({u},.)^{-1})=\varphi(u),
\quad\text{for all } u\in I.
$$
\end{theorem}

Throughout this paper, we work on the space $(I,\mathcal{I},\lambda)$ for the index space.

%\medskip

Now we state an important result on the Fubini extension, the Exact Law of Large Numbers (ELLN), which allows us to replace the stochastic mean by a deterministic mean.

\begin{proposition}{\cite[Corollary 2.10]{sun06exact}}\label{ELLN}
Let $\Theta $ be an integrable random variable on $(I\times \Omega,\mathcal{V},\mathcal{Q})$. If $\Theta$ is essentially pairwise independent, then  the sample mean is the same as the mean of $\Theta$, 
\begin{align*}
\int_{I}\Theta(u,\tilde{\omega})\lambda(\mathrm{d}u) = \int_{I \times \Omega} \Theta(u,\omega)\mathcal{Q}(\mathrm{d}u,\mathrm{d}\omega), \;\; \forall \tilde{\omega} \in \Omega, \; \mathbb{Q}\text{-a.s} 
    \end{align*}
\end{proposition}

We next state a version of Girsanov’s theorem, adapted to a family of e.p.i.\ Wiener processes in the Fubini extension setting, as detailed in our companion paper~\cite{amini25brownian}.

\begin{proposition}{\cite[Theorem 4]{amini25brownian} }\label{girsanovlemma}
Let $(I \times \Omega, \mathcal{V}, \mathcal{Q})
$ be  a Fubini extension probability space. Consider a collection of e.p.i. Wiener processes $\mathbf{B} = \big\{B^{u} : u \in I \big\}$, and  a real valued process  $(\Theta_t)_{t \geq 0} = \Big\{(\theta_{t,u}^{})_{t \geq 0} : u \in I\Big\}$, 
such that for all $(t,u) \in [0,T]\times I$, $\theta_{t,u} $ is $\sigma(B_s^{u}; s\leq t)$-measurable.

Define the process $(\mathcal{E}_{t})_{t \geq 0} = \Big\{(\mathcal{E}_{t,u})_{t \geq 0} : u \in I\Big\}$ by 
{{\small} \begin{align*} \mathcal{E}_{t} := \exp\Big\{\int_{0}^{t}\Theta_{s}\mathrm{d}\mathbf{B}_s^{} - \frac{1}{2}\int_{0}^{t}\Theta_{s}^{2}\mathrm{d}s \Big\}.
\end{align*}}
If for all $u \in I$, the process $(\mathcal{E}_{t,u})_{t \geq 0}$ is a $\mathbb{Q}$-martingale, then  under the probability measure $\tilde{\mathcal{P}}$ defined by the density  $\mathcal{E}_T$ with respect to $\mathcal{Q}$,
the process $\tilde{\mathbf{B}} = \Big\{\tilde{{B}}^{u} : u \in I\Big\}$ is a Wiener process on the probability space $(I\times \Omega, \mathcal{V}, \tilde{\mathcal{P}}),$ where 
{{\small} \begin{align*}
\tilde{\mathbf{B}}_t &:= \mathbf{B}_t^{} - \int_{0}^{t} \Theta_{s}\mathrm{d}s.
\end{align*}}
\end{proposition}

\section{Non-exchangeable many-body observed quantum particles}\label{sec:particles}
We begin the analysis by fixing a configuration space $(\mathfrak{X},\Omega_{\mathfrak{X}},\mu)$, from which we construct the one-particle Hilbert space
$\mathbb{H} := L^2_{\mathbb{C}}(\mathfrak{X},\mu)$.

It is not a restriction to proceed in this way, because a  complex separable Hilbert space has the following isomorphism $\mathbb{H} \simeq l^2_{\mathbb{C}}(\mathbb{N}) \simeq L^{2}_{\mathbb{C}}(\mathbb{N},\delta),$ where $\delta$ is a counting measure. For example, if we consider the position of a quantum particle in space, then $\mathfrak
{X} = \mathbb{R}^d$ with $\mu$ being the Lebesgue measure. If instead we consider discrete energy levels for a quantum particle, we may set $\mathfrak{X} = \mathbb{N}$ or $\mathfrak{X} = \{1,\dots,d\}$, with $\mu $ as the counting measure. 

Consider $N$ quantum particles distributed over a graph $ G_N = (V_N, E_N) $, which is completely specified by its adjacency matrix $ (\xi_{pq}^{N})_{p, q \in V_N}$ indicating the intensity of interaction between two particles. The coefficients $\xi_{pq}^{N}$, issued from the step kernel function $w^{G_N}$, can be constructed in two ways:
\begin{itemize}
    \item $\xi_{pq}^{N} = {w}^{G_N}(p,q),$
    \item $\xi_{pq}^{N} = \xi_{qp}^{N}= \text{Bernoulli}({w}^{G_N}(p,q))$ independently for $p,q \in V_N.$
\end{itemize}

The associated Hilbert space is defined by the tensor product of the Hilbert space of each particle and is given by
\begin{align*}\mathbb{H}^{N} := \bigotimes_{ q \in V_N}\mathbb{H} &= L^2_{\mathbb{C}}(\mathfrak{X},\mu)\otimes\dots\otimes L^2_{\mathbb{C}}(\mathfrak{X},\mu)  \equiv L^2_{\mathbb{C}}(\mathfrak{X}^{N},\mu^{\otimes N}),
\end{align*}
so that the state of the whole system is represented by a density operator on $\mathbb{H}^{N}$, i.e., 
\begin{align*} \mathcal{S}(\mathbb{H}^{N}) &:= \Big\{ \boldsymbol{\rho} \in \mathcal{B}_1(\mathbb{H}^{N})\; \Big| \; \boldsymbol{\rho} \geq 0, \; \tr(\boldsymbol{\rho}) = 1,\; \boldsymbol{\rho} = \boldsymbol{\rho}^{\dagger} \Big\}. \end{align*}
Each particle is governed by a self-adjoint free Hamiltonian $\tilde{H}^{\dagger} = \tilde{H} \in \mathcal{B}(\mathbb{H}) $. The interaction between two particles is represented by an  operator $ A $ on $L^2_{\mathbb{C}}(\mathfrak{X}\times{\mathfrak{X}},\mu^{\otimes 2})$ of Hilbert-Schmidt type, with kernel $a$ satisfying the following proprieties:
\begin{align*}
\|a\|^{2}_{2} &= \int_{\mathfrak{X}^4}|a(x,y;x',y')|^{2}\mu(\mathrm{d}x)\mu(\mathrm{d}y)\mu(\mathrm{d}x')\mu(\mathrm{d}y') < \infty,\\
a(x,y;x',y') &= a(y,x;y',x'), \qquad a(x,y;x',y') = \overline{a(x',y';x,y)},
\end{align*}
\begin{align*} 
A& : L^2_{\mathbb{C}}(\mathfrak{X}\times\mathfrak{X},\mu^{\otimes 2}) \to
L^2_{\mathbb{C}}(\mathfrak{X}\times{\mathfrak{X}},\mu^{\otimes 2}),\\
  Af(x,y) &:= \int_{\mathfrak{X}^2}a(x,y;x',y')f(x',y')\mu(\mathrm{d}x')\mu(\mathrm{d}y'), \; \forall f \in L^2_{\mathbb{C}}(\mathfrak{X}^2,\mu^{\otimes 2}).
\end{align*}

At the same time, each particle is observed through a quantum channel $L \in \mathcal{B}_{}(\mathbb{H})$ with detection efficiency $\eta \in (0,1]$. Then the dynamics of a  posteriori state is given by the $N$-Body Belavkin, with a family $\mathbf{W}^{N} = \{ W^{q} , \; q \in V_N\}$ of pairwise  independent   Wiener processes
{{\small} \begin{align}\label{Nbelavkin}
\mathrm{d}\boldsymbol{\rho}_t^{N} &= -\mathrm{i}{[\mathbf{H}^{N},\boldsymbol{\rho}_{t}^{N}]}\mathrm{d}t
+ \sum_{q \in V_{N}} \Big({\bf L}_q\boldsymbol{\rho}_{t}^{N}{\bf L}^{\dagger}_{q} - \frac{1}{2}\big\{{\bf L}^{\dagger}_{q}{\bf L}_{q},\boldsymbol{\rho}_{t}^{N}\big\}\Big)\mathrm{d}t\nonumber\\
&\quad +\sqrt{\eta}\sum_{q \in V_{N}}\Big({\bf L}_q\boldsymbol{\rho}_{t}^{N} + \boldsymbol{\rho}_{t}^{N}{\bf L}^{\dagger}_q - \tr\big(({\bf L}_q + {\bf L}_q^{\dagger})\boldsymbol{\rho}_{t}^{N}\big)\boldsymbol{\rho}_{t}^{N}\Big)\mathrm{d}W_t^{q},\\
\boldsymbol{\rho}_{0}^{N} &= \bigotimes_{{V_N}}\rho_{0}, \;\; \rho_{0} \in \mathcal{S}(\mathbb{H}),\nonumber
\end{align}}
where
{{\small} \begin{align*}
    \mathbf{H}^{N} &:= \sum_{q \in V_N}\mathbf{\tilde{H}}_{q} + \frac{1}{N}\sum_{q \in V_N}\sum_{p > q} \xi^{N}_{pq}\mathbf{A}_{pq}
\end{align*}}
and the corresponding observation process for particle $q \in V_N$ is given by the following stochastic process 
{\begin{align*}
\mathrm{d}Y_t^{q} &= \mathrm{d}W_t^{q} + \sqrt{\eta}\,\tr\big( ({\bf L}_q + {\bf L}_q^{\dagger})\boldsymbol{\rho}_{t}^{N}\big)\mathrm{d}t, \quad q \in V_N.
\end{align*}}

\section{Mean-field graphon quantum filtering systems}\label{sec:MF}

The idea behind mean-field limit, is that as $N \to \infty $, the interaction weakens in the sense that the influence of each particle, tends to zero. The system's description is then given by a tensoriel product state. The justification for the preservation of mean-field dynamics is provided by the notion of quantum propagation of chaos, which essentially states that the system remains approximately in a product state if it starts in a product state.  This can be summarized as follows: 
\begin{align*}
\boldsymbol{\rho}_{0}^{N} = \rho_{0}^{\otimes N} \xrightarrow[N \to \infty]{\text{Propagation of chaos}}\boldsymbol{\rho}_{t}^{N}  \approx \bigotimes_{u \in I}\gamma_{t,u}^{}, 
\end{align*}
where $\gamma_{t,u}$ is the density operator of the $u \in I$-th particle at time $t$. 

Thus, the dynamics of the system in the mean-field  limit are given by the following
graphon quantum mean-field Belavkin system of equation driven by a familly $\mathbf{W} = \{ W^{u} : u \in I\}$ of e.p.i. Wiener processes on Fubini extended space $(I\times\Omega,\mathcal{V},\mathcal{P})$ of the usual product space  $(I\times\Omega,\mathcal{I}\otimes\mathcal{F},\lambda\otimes\mathbb{P})$: 
{{\small} \begin{align}\label{graphonbelavkin}
\mathrm{d}\gamma_{t,{u}} &= -\mathrm{i}\Big[\tilde{H} + \int_{0}^{1}{w(u,v)}{A^{\mathbb{E}_{\mathbb{P}}[\gamma_{t,{v}}]}}\mathrm{d}v, \gamma_{t,{u}} \Big]\mathrm{d}t + \Big(L\gamma_{t,{u}}L^{\dagger} - \frac{1}{2}\{L^{\dagger}L, \gamma_{t,{u}}\}\Big)\mathrm{d}t\nonumber\\
&\quad\quad\quad\quad + \sqrt{\eta}\Big(L\gamma_{t,{u}} + \gamma_{t,{u}}L^{\dagger} - \tr\big((L+L^{\dagger})\gamma_{t,{u}}\big)\gamma_{t,{u}}\Big)\mathrm{d}W_t^{u},\\
\gamma_{0,u} &=  \rho_{0} \in \mathcal{S}(\mathbb{H}),\nonumber
\end{align}}
where the mean-field operator $A^{\bullet}$ is a linear operator:
{{\small} \begin{align} \label{meanfield}
\quad A^{\bullet} &: L^{2}_{\mathbb{C}}(\mathfrak{X}\times\mathfrak{X},\mu^{\otimes 2}) \to L^{2}_{\mathbb{C}}(\mathfrak{X}\times\mathfrak{X},\mu^{\otimes 2}),\nonumber\\
A^{\rho}(x,x') &= \int_{\mathfrak{X}^{2}}a(x,y;x',y')\overline{\rho(y,y')}\mu(\mathrm{d}y)\mu(\mathrm{d}y'), \; \forall \rho \in  L^{2}_{\mathbb{C}}(\mathfrak{X}\times\mathfrak{X},\mu^{\otimes 2}).
\end{align}}

The associated observation process for the particle $u \in I$ is
{{\small} \begin{align*}\mathrm{d}Y_t^{u} = \mathrm{d}W_t^{u} + \sqrt{\eta}\tr\big((L+L^{\dagger})\gamma_{t,u}\big)\mathrm{d}t.
\end{align*}}

An important estimate between the mean-field operator and the pairwise operator is established in \cite[Lemma 5.1]{kolokoltsov22qmfg}.

\begin{lemma}[\cite{kolokoltsov22qmfg}]\label{mfoperatornorm}
Let $A$ be a measurable complex-valued  Hilbert-Schmidt operator on $\mathfrak{X}^2$ with associated kernel $a: \mathfrak{X}^4\mapsto \mathbb{C}$  and $\rho: \mathfrak{X}^2 \mapsto\mathbb{C}$ be a complex-valued function. Then for $A^{\rho}$ as defined in \eqref{meanfield}, we have  
{{\small} \begin{align*}
\big\|A^{\rho}\big\|_{2} \leq \big\|{\rho}\big\|_{2} \big\|a\big\|_{2}.
\end{align*}}
\end{lemma}

\begin{remark}
When the interaction is homogeneous, i.e., $w(u,v) = 1$ for all $(u,v) \in I\times I$, the stochastic system reduces to the standard mean-field Belavkin equation studied in \cite{kolokoltsov22qmfg,kolokoltsov21law}:
\begin{align*}
\mathrm{d}\gamma_{t,u} &= -\mathrm{i}\Big[\tilde{H} + A^{\mathbb{E}_{\mathbb{P}}[\gamma_{t,u}]}, \gamma_{t,u} \Big]\mathrm{d}t + \Big(L\gamma_{t,u}L^{\dagger} - \frac{1}{2}\{L^{\dagger}L, \gamma_{t,u}\}\Big)\mathrm{d}t\\
&\quad\quad\quad\quad + \sqrt{\eta}\Big(L\gamma_{t,u} + \gamma_{t,u}L^{\dagger} - \tr\big((L+L^{\dagger})\gamma_{t,u}\big)\gamma_{t,u}\Big)\mathrm{d}W_t^{u}, \quad \quad \quad u \in [0,1].
\end{align*}
\end{remark}

By averaging over trajectories, that is, $m_{t,u} = \mathbb{E}_{\mathbb{P}}[\gamma_{t,u}],$ we obtain a new nonlinear Lindblad system of equations given by
{{\small} \begin{align*}
\mathrm{d}m_{t,u} &= -\mathrm{i}\Big[\tilde{H} + \int_{0}^{1}{w}(u,v)A^{m_{t,v}}\mathrm{d}v, m_{t,u} \Big]\mathrm{d}t + \Big(Lm_{t,u}L^{\dagger} - \frac{1}{2}\{L^{\dagger}L, m_{t,u}\}\Big)\mathrm{d}t, \; \forall u \in I.
\end{align*}}
This can be recovered equally in the absence of measurement, i.e., when $\eta \to 0$. In this case, the observation process $(Y_t^u)_{t \geq 0}$ reduces to pure noise, namely $(Y_t^u)_{t \geq 0} = (W_t^u)_{t \geq 0}$. 

Note also that the processes $(\gamma_{t,u})_{t \geq 0}$ in \eqref{graphonbelavkin} are independent but not identically
 distributed non-linear diffusions.

\begin{remark}[Formulation  in the Fubini extended space]\label{rem:formulation}

Before stating the first theorem, it is useful to point out that the graphon SDEs \eqref{graphonbelavkin}
can be formulated as a single SDE defined on the Fubini extension space, driven by a standard Wiener process (see \cite{amini25brownian} for details and  \cite{coppini25nonlinear} for an $L^2$--formulation approach). 

Let 
{{\small} \begin{align*}
\mathcal{H} := L^{2}_{\mathcal{L}_2(\mathbb{H})}(I\times\Omega,\mathcal{P}), 
\qquad  
\mathcal{H}_I := L^{2}_{\mathcal{L}_2(\mathbb{H})}(I,\lambda).
\end{align*}}
The collection $\{\gamma_{t,u} : u\in I\}$ can equivalently be represented by
$$
\Gamma_t \in \mathcal{H}, 
\qquad \Gamma_t(u,\omega) := \gamma_{t,u}(\omega).
$$
Define
$$
\mathrm{M}_t := \mathbb{E}_{\mathbb{P}}[\Gamma_t] \in \mathcal{H}_I,
$$
so that $\mathrm{M}_t(u) = \mathbb{E}_{\mathbb{P}}[\gamma_{t,u}]$ for each $u\in I$.

Given the Hilbert--Schmidt interaction operator $A^{\bullet}$, the linear operator of interaction $\mathbb{W}_{A}$ is defined by
{{\small} \begin{align*}
\mathbb{W}_A : &\; \mathcal{H}_I \longrightarrow \mathcal{H}_I\\
\big(\mathbb{W}_A[\mathrm{G}]\big)(u) &:= \int_I w(u,v)A^{\mathrm{G}(v)}\mathrm{d}v,
\qquad \forall \mathrm{G} \in \mathcal{H}_I, \; u\in I.
\end{align*}}
%\medskip

With these notations, the system \eqref{graphonbelavkin} can be written as the Hilbert-space valued stochastic equation
\begin{align}\label{compactbelavkin}
\mathrm{d}\Gamma_t = \mathcal{D}\big[\Gamma_t,\mathbb{W}_A(\mathrm{M}_t)\big]\mathrm{d}t
+ \mathcal{K}[\Gamma_t]\mathrm{d}\mathbf{W}_t,
\end{align}
where:
\begin{itemize}
  \item $\mathcal{D}[.,.]$ collects the drift terms:
{{\small} \begin{align*}
    \mathcal{D}[\Gamma,\Gamma'] 
    = -\mathrm{i}[\tilde{H}+\Gamma',\Gamma] 
      + L\Gamma L^{\dagger} - \tfrac{1}{2}\{L^{\dagger}L,\Gamma\},
\end{align*}}
  \item $\mathcal{K}[.]$ denotes the diffusion part:
{{\small} \begin{align*}
    \mathcal{K}[\Gamma]
    = \sqrt{\eta}\,\big(L\Gamma + \Gamma L^{\dagger}-\tr((L+L^{\dagger})\Gamma)\Gamma\big).
\end{align*}}
\end{itemize}
\end{remark}

The following theorem establishes the well-posedness of the system \eqref{graphonbelavkin}.

\begin{theorem}[Well-posedness]
Let $T > 0$, $\tilde{H} \in \mathcal{B}(\mathbb{H})$ be self-adjoint, $L \in \mathcal{B}_{2}(\mathbb{H})$ and let $A^{\bullet}$ be defined as in \eqref{meanfield}.
Under the Fubini extension probability space $(I\times\Omega ,\mathcal{V}, \mathcal{P})$ with a family of essentially pairwise independents Wiener processes  
$\mathbf{W} = \big\{W^{u} : u \in I \big\}$, the system \eqref{graphonbelavkin} has a unique strong solution and for every $(t,u) \in [0,T]\times I$, it holds that $\gamma_{t,u} \in \mathcal{S}(\mathbb{H}).$

Moreover, let $(\tilde{\gamma}_t)_{t \geq 0}$ and $(\hat{\gamma}_t)_{t \geq 0}$ be two solutions of \eqref{graphonbelavkin} associated to graphon $\tilde{w}$ and $\hat{w}$, respectively, 
 and with the same initial condition. Then there exists  some constant $C := C(T,\eta, \|\tilde{H}\|,\|a\|_{2},\|L\|_{2})$ such that 
$$ \sup_{t \in [0,T]}\mathbb{E}_{\mathbb{P}}\Big[\int_{I}\Big\|\tilde{\gamma}_{t,u} -\hat{\gamma}_{t,u}\Big\|^{2}_{2}\mathrm{d}u\Big] \leq C\|\tilde{w} - \hat{w}\|_{L^1}^{}. $$
\end{theorem}

\begin{proof} The system \eqref{graphonbelavkin} can be seen as a type of McKean-Vlasov interacting stochastic differential equation. The proof is structured in five steps: 

\begin{enumerate}
\item \textit{Parametrization:}
Set $\mathcal{C}_{I} := \mathcal{C}\big([0,T]\times I,\mathcal{S}(\mathbb{H})^{}\big)$, and for each $\xi := (\xi^{u})_{u \in I} \in \mathcal{C}_{I}$ equip it with the norm 
$$ \|\xi\|_{\star,T} := \sup_{t \in [0,T]}\sup_{u \in I}\|\xi^{u}_{t}\|_{2}.$$
We formally replace the dependency on law by  parameterization and define the  mapping $\Xi:\mathcal{C}_{I}\to \mathcal{C}_{I}$ by $\Xi(\xi) := (\Xi_{u}(\xi))_{u \in I},$
where $\Xi_{u}(\xi):=(\mathbb{E}_{\mathbb{P}}[\gamma_{t,u}^{\xi}])_{0 \leq t \leq T}$ with
{{\small} \begin{align}
\mathrm{d}\gamma^{\xi}_{t,u} &= -\mathrm{i}\left[ \tilde{H} + \int_{0}^{1} w(u,v) A^{\xi_{t}^{v}} \mathrm{d}v, \gamma^{\xi}_{t,u} \right] \mathrm{d}t + \left( L \gamma^{\xi}_{t,u} L^{\dagger} - \frac{1}{2} \{ L^{\dagger} L, \gamma^{\xi}_{t,u} \} \right) \mathrm{d}t \nonumber \\
&\quad\quad + \sqrt{\eta} \left( L \gamma^{\xi}_{t,u} + \gamma^{\xi}_{t,u} L^{\dagger} - \tr\left( (L + L^{\dagger}) \gamma^{\xi}_{t,u} \right) \gamma^{\xi}_{t,u} \right) \mathrm{d}W_{t}^{u}, \quad u \in I. \label{graphonpara}
\end{align}}

\item \textit{Linearization and positivity:} 
We linearize the parametrized system by expressing its solution as a function of the solutions of a suitable linear SDE. Consider the following  $\xi$-parametrized linear system:
\begin{align} \mathrm{d}\vartheta^{\xi}_{t,{u}} &= -\mathrm{i}\Big[\tilde{H} + \int_{0}^{1}w(u,v)A^{\xi^{v}_{t}}\mathrm{d}v, \vartheta^{\xi}_{t,{u}} \Big]\mathrm{d}t + \Big(L\vartheta^{\xi}_{t,{u}}L^{\dagger} - \frac{1}{2}\{L^{\dagger}L, \vartheta^{\xi}_{t,{u}}\}\Big)\mathrm{d}t\nonumber\\
&\quad\quad\quad + \sqrt{\eta}\Big(L\vartheta^{\xi}_{t,{u}} + \vartheta^{\xi}_{t,{u}}L^{\dagger}\Big)\mathrm{d}Y_t^{u}, \; u \in I,\label{lineargraphonpara} \end{align}
where $\{ Y^{u} : u \in I \} $ is a collection of independent Wiener processes on the following auxiliary filtered probability space $(I\times\Omega, \mathcal{V}, \mathcal{Q}).$

From Proposition \ref{linearwellposdness} in Appendix~\ref{sec:appB}, this system admits a strong solution in $\mathcal{L}_{2}(\mathbb{H})^{}$ with positivity preserved for every $u \in I$.

Define  $\gamma_{t,u}^{\xi} := \frac{\vartheta_{t,u}^{\xi}}{\tr(\vartheta_{t,u}^{\xi})}.$ By using Proposition~\eqref{lineartononlinear} and by Itô formula, this corresponds to a weak solution of Equation \eqref{graphonpara}. Note that $\gamma_{t,u}^{\xi}$ is positive and trace-one preserving as soon as $\tr(\gamma_{0,u}^{\xi})=1$.

\item \textit{Pathwise uniqueness step:}
In order to apply the fixed point argument, we need to define a strong solution for any $\xi \in \mathcal{C}_I$ to \eqref{graphonpara}. Since the Liouville space is a Hilbert space, weak existence and pathwise 
uniqueness imply the existence of a unique strong solution by the Hilbert-space version of the Yamada–Watanabe theorem. Consider two  solutions $\gamma^{\xi,a}, \gamma^{\xi,b}$ defined on the same probability space $(I\times\Omega,\mathcal{V}, \mathcal{P}) $ with the same initial state in $\mathcal{L}_2(\mathbb{H})$, and set $K:=L + L^{\dagger}$. Then,
{{\small} \begin{align*}
\Big\|\gamma_{t,u}^{\xi,a} - \gamma_{t,u}^{\xi,b}\Big\|_{2}^{2} &= \Big\|\int_{0}^{t}-\mathrm{i}\big[\tilde{H} + \int_{0}^{1}w(u,v)A^{\xi_t^{v}}\mathrm{d}v, (\gamma_{s,u}^{\xi,a} - \gamma_{s,u}^{\xi,b})\big]\mathrm{d}s\\ 
&+ \int_{0}^{t}\Big(L(\gamma_{s,u}^{\xi,a} - \gamma_{s,u}^{\xi,b})L^{\dagger} - \frac{1}{2}\{L^{\dagger}L,(\gamma_{s,u}^{\xi,a} - \gamma_{s,u}^{\xi,b})\}\Big)\mathrm{d}s  \\
+ \sqrt{\eta}\int_{0}^{t}\Big(L{}(\gamma_{s,u}^{\xi,a} &- \gamma_{s,u}^{\xi,b}) + (\gamma_{t,u}^{\xi,a} - \gamma_{s,u}^{\xi,b})L^{\dagger}\\
& -\tr\big(K\gamma_{s,u}^{\xi,a}\big)(\gamma_{s,u}^{\xi,a} - \gamma_{s,u}^{\xi,b}) - \tr\big(K(\gamma_{s,u}^{\xi,a} - \gamma_{s,u}^{\xi,b})\big)\gamma_{s,u}^{\xi,b} \Big)\mathrm{d}W_s^{u}\Big\|_{2}^{2}.
\end{align*}}
By the triangular inequality and the fact that $(a+b)^2 \leq 2(a^2+b^2)$,
{{\small} 
\begin{align*}\Big\|\gamma_{t,u}^{\xi,a} - \gamma_{t,u}^{\xi,b}\Big\|_{2}^{2} &\leq 2\Big\|\int_{0}^{t}-\mathrm{i}\big[\tilde{H} + \int_{0}^{1}w(u,v)A^{\xi_t^{v}}\mathrm{d}v, (\gamma_{s,u}^{\xi,a} - \gamma_{s,u}^{\xi,b})\big]\mathrm{d}s\\ 
&+ \int_{0}^{t}\Big(L(\gamma_{s,u}^{\xi,a} - \gamma_{s,u}^{\xi,b})L^{\dagger} - \frac{1}{2}\{L^{\dagger}L,(\gamma_{s,u}^{\xi,a} - \gamma_{s,u}^{\xi,b})\}\Big)\mathrm{d}s   \Big \|^2_{2}\\
+2&\eta\Big\|\int_{0}^{t}\Big(L{}(\gamma_{s,u}^{\xi,a} - \gamma_{s,u}^{\xi,b}) + (\gamma_{s,u}^{\xi,a} - \gamma_{s,u}^{\xi,b})L^{\dagger}\\
& - \tr\big(K\gamma_{s,u}^{\xi,a}\big)(\gamma_{s,u}^{\xi,a} - \gamma_{s,u}^{\xi,b}) - \tr\big(K(\gamma_{s,u}^{\xi,a} - \gamma_{s,u}^{\xi,b})\big)\gamma_{s,u}^{\xi,b} \Big) \mathrm{d}W_s^{u}\Big\|_{2}^{2}.
\end{align*}}
By taking expectations and applying the Cauchy-Schwarz inequality and the Itô-Isometry,
{{\small} 
\begin{align*}\mathbb{E}_{\mathbb{P}}\Big[\Big\|\gamma_{t,u}^{\xi,a} - \gamma_{t,u}^{\xi,b}\Big\|_{2}^{2}\Big] &\leq 4\mathbb{E}_{\mathbb{P}}\Big[t\int_{0}^{t}\Big\|\big[\tilde{H} + \int_{0}^{1}w(u,v)A^{\xi_t^{v}}\mathrm{d}v, (\gamma_{s,u}^{\xi,a} - \gamma_{s,u}^{\xi,b})\big]\Big\|_{2}^{2}\mathrm{d}s\Big]\\ 
+ 4\mathbb{E}_{\mathbb{P}}\Big[t\int_{0}^{t}&\Big\|\Big(L(\gamma_{s,u}^{\xi,a} - \gamma_{s,u}^{\xi,b})L^{\dagger} - \frac{1}{2}\{L^{\dagger}L,(\gamma_{s,u}^{\xi,a} - \gamma_{s,u}^{\xi,b})\}\Big)\Big \|^2_{2}\mathrm{d}s \Big]  \\
+4\eta\mathbb{E}_{\mathbb{P}}\Big[\int_{0}^{t}&\Big\|\Big(L{}(\gamma_{s,u}^{\xi,a} - \gamma_{s,u}^{\xi,b}) + (\gamma_{s,u}^{\xi,a} - \gamma_{s,u}^{\xi,b})L^{\dagger} - \tr\big(K\gamma_{s,u}^{\xi,a}\big)(\gamma_{s,u}^{\xi,a} - \gamma_{s,u}^{\xi,b}) \Big\|_{2}^{2}\mathrm{d}s\Big]\\ 
+ 4\eta\mathbb{E}_{\mathbb{P}}\Big[\int_{0}^{t}&\Big\| \tr\big(K(\gamma_{s,u}^{\xi,a} - \gamma_{s,u}^{\xi,b})\big)\gamma_{s,u}^{\xi,b} \Big\|_{2}^{2}\mathrm{d}s\Big].
\end{align*}}
Note that each contribution is linear in $\gamma_{t,u}^{\xi,a} - \gamma_{t,u}^{\xi,b}$,
and thanks to the Cauchy-Schwarz inequality, we get
{{\small}
\begin{align*}
    \tr^{2}\big(K(\gamma_{t,u}^{\xi,a} - \gamma_{t,u}^{\xi,b})\big) &\leq \|K\|_{2}^{2}\|\gamma_{t,u}^{\xi,a} - \gamma_{t,u}^{\xi,b}\|_{2}^{2}.
\end{align*}
}
Then, we find
{{\small} \begin{align*}
\sup_{u \in I} \mathbb{E}_{\mathbb{P}}\Big[\Big\|\gamma_{t,u}^{\xi,a} - \gamma_{t,u}^{\xi,b}\Big\|_{2}^{2}\Big] &\leq C\int_{0}^{t}\sup_{u \in I}\mathbb{E}_{\mathbb{P}}\Big[\Big\|\gamma_{s,u}^{\xi,a} - \gamma_{s,u}^{\xi,b}\Big\|_{2}^{2}\Big]\mathrm{d}s.
\end{align*}}
By Gr\"onwall lemma, we conclude $\gamma^{\xi,a}=\gamma^{\xi,b}$ almost surely. Thanks to the Hilbert space formulation of Yamada-Watanabe theorem (see \cite[Appendix E]{liurockner15spde}), we obtain the existence of a unique strong solution.

\item \textit{Fixed point step:} We now define the mapping 
{{\small} \begin{align*}\Xi : \mathcal{C}_{I} \to
\mathcal{C}_{I} \; \text{by} \;\; 
\Xi(\xi) := \big((\mathbb{E}_{\mathbb{P}}[\gamma_{t,u}^{\xi}])_{0 \leq t \leq T}\big)_{u \in I},
\end{align*}}
so that the process $\gamma^{m}$ corresponds to the solution of \eqref{graphonbelavkin} if and only if $m = \Xi(m)$.
Therefore, we should prove existence of a unique fixed point, i.e., by proving that $\Xi$ is contraction
with respect to the  norm on $\mathcal{C}_{I}$. 

Let us pick two arbitrary elements $\tilde{\xi} = (\tilde{\xi}^{u})_{u \in I}$ and $\hat{\xi} = (\hat{\xi}^{u})_{u \in I}$ in $\mathcal{C}_{I}$ equipped with the norm $\|\xi\|_{\star,t} := \sup_{r \in [0,t]}\sup_{u \in I}\|\xi^{u}_{r}\|_{2}$.
Set
{{\small} \begin{align*}
\Delta\gamma_{t,u} := \gamma_{t,u}^{\tilde{\xi}} - \gamma_{t,u}^{\hat{\xi}},\; \Delta{\xi^{u}_{t}} := \tilde{\xi}^{u}_{t} - \hat{\xi}^{u}_{t}, \; \Delta{\xi} := \tilde{\xi} - \hat{\xi}, \; \Delta A^{\xi^{v}} := A^{\tilde{\xi}^{v}} - A^{\hat{\xi}^{v}}.
\end{align*}}
Then,
{{\small} \begin{align*}
\Delta\gamma_{t,u} &=-\mathrm{i}\int_{0}^{t}\Big[\tilde{H} ,\Delta\gamma_{s,u}\Big]\mathrm{d}s - \mathrm{i}\int_{0}^{t}\Big[\int_{0}^{1}w(u,v)A^{\xi^{v}}\mathrm{d}v, \Delta\gamma_{s,u}\Big]\mathrm{d}s\\
&\qquad - \mathrm{i}\int_{0}^{t}\Big[\int_{0}^{1}w(u,v)\Delta A^{\xi^{v}}\mathrm{d}v, \gamma_{s,u}^{\hat{\xi}}\Big]\mathrm{d}s\\    &+\int_{0}^{t}\Big(L\Delta\gamma_{s,u}L^{\dagger} - \frac{1}{2}\{L^{\dagger}L, \Delta\gamma_{s,u}\}\Big)\mathrm{d}s + \sqrt{\eta}\int_{0}^{t}\Big(L\Delta\gamma_{s,u} + \Delta\gamma_{s,u}L^{\dagger}\Big)\mathrm{d}W_{s}^{u} \\
&\quad\quad\quad - \sqrt{\eta}\int_{0}^{t}\Big(\tr\big(K\Delta\gamma_{s,u}\big)\gamma_{s,u}^{\tilde{\xi}} + \tr\big(K\gamma_{s}^{\hat{\xi}^{u}}\big)\Delta\gamma_{s,u}\Big)\mathrm{d}W_{s}^{u},
\end{align*}}
and,
{{\small} \begin{align*}
\big\|\Xi_{u}(\tilde{\xi})_t - \Xi_{u}(\hat{\xi})_t&\big\|_{2} \leq \mathbb{E}_{\mathbb{P}}\big[\|\Delta\gamma_{t,u}\|_{2} \big]\\
&\leq \int_{0}^{t}\mathbb{E}_{\mathbb{P}}\Big[ \Big\| [H,\Delta\gamma_{s,u}]\Big\|_{2} + \Big\|L\Delta\gamma_{s,u}L^{\dagger}\Big\|_{2} + \frac{1}{2}\Big\|\{L^{\dagger}L,\Delta\gamma_{s,u}\}\Big\|_{2}\Big]\mathrm{d}s\\
+\int_{0}^{t}\Big(\mathbb{E}_{\mathbb{P}}&\Big[\Big\|[\int_{0}^{1}w(u,v)A^{\xi^{v}}\mathrm{d}v,\Delta\gamma_{s,u}] \Big\|_{2} \Big] + \mathbb{E}_{\mathbb{P}}\Big[\Big\|[\int_{0}^{1}w(u,v)\Delta A^{\xi^{v}}\mathrm{d}v,\gamma_{s,{u}}] \Big\|_{2} \Big]\Big)\mathrm{d}s\\
+ \int_{0}^{t}\Big(\mathbb{E}_{\mathbb{P}}&\Big[\Big\| L\Delta\gamma_{s,u} + \Delta\gamma_{s,u}L^{\dagger}\Big\|_{2}^2\Big]\Big)^{\frac{1}{2}}\mathrm{d}s\\
&+\int_{0}^{t}\Big(\mathbb{E}_{\mathbb{P}}\Big[\Big\|\tr\big(K\Delta\gamma_{s,u}\big)\gamma_{s,u}^{\tilde{\xi}} \Big\|_{2}^{2}\Big]\Big)^{\frac{1}{2}} + \Big(\mathbb{E}_{\mathbb{P}}\Big[\Big\|\tr\big(K\gamma_{s,u}^{\hat{\xi}}\big)\Delta\gamma_{s,u} \Big\|_{2}^{2}\Big]\Big)^{\frac{1}{2}} \mathrm{d}s\\
&\leq C\int_{0}^{t}\Big(\mathbb{E}_{\mathbb{P}}\Big[\Big\|\Delta\gamma_{s,u}\Big\|_{2}\Big] + \int_{0}^{1}w(u,v)\big\|\Delta\xi_s^{v}\big\|_{2}\mathrm{d}v\Big)\mathrm{d}s.
\end{align*}}
In view of the Gr\"onwall’s inequality, one concludes the existence of some constant $C$ such that 
{{\small} \begin{align*}
    \|\Xi_{u}(\tilde{\xi})_{r} - \Xi_{u}(\hat{\xi})_{r}\|_{2} \leq C\int_{0}^{t}\int_{0}^{1}w(u,v)\|\Delta\xi_{s}^{v}\|_{2}\mathrm{d}v\mathrm{d}s
    \leq C\int_{0}^{t}\sup_{u \in I}\|\Delta\xi_{s}^{u}\|_{2}\mathrm{d}s.
\end{align*}}

Replacing $\xi^{a}$ (resp. $\xi^{b}$) by $\Xi(\xi^{a})$ (resp. $\Xi(\xi^{b})$), it follows that 
{{\small} \begin{align*}
    \|\Xi_{u}^{(2)}(\tilde{\xi})_{t} - \Xi_{u}^{(2)}(\hat{\xi})_{t}\|_{2} \leq C^2\int_{0}^{t}\int_{0}^{s}\sup_{u \in I}\|\Delta\xi_{s}^{u}\|_{2}\mathrm{d}r\mathrm{d}s
    \leq \frac{C^{2}t^{2}}{2}\sup_{0 \leq r \leq t}\sup_{u \in I}\|\Delta\xi_{r}^{v}\|_{2}.
\end{align*}}
By induction, for any $k \geq 1$, 
{{\small} \begin{align*}
    \|\Xi_{u}^{(k)}(\tilde{\xi})_{t} - \Xi_{u}^{(k)}(\hat{\xi})_{t}\|_{2} &\leq \frac{C^{k}t^{k}}{k!}\sup_{0 \leq r \leq t}\sup_{u \in I}\|\Delta\xi_{r}^{v}\|_{2},
\end{align*}}
and thus,
{{\small} \begin{align*}
\sup_{0 \leq t \leq T}\sup_{u \in I}\|\Xi_{u}^{(k)}(\tilde{\xi})_{t} - \Xi_{u}^{(k)}(\hat{\xi})_{t} \|_{2} &\leq \frac{C^{k}T^{k}}{k!}\sup_{0 \leq r \leq t}\sup_{u \in I}\|\Delta\xi_{r}^{v}\|_{2},\\
\|\Xi^{(k)}(\tilde{\xi}) - \Xi^{(k)}(\hat{\xi}) \|_{\star,T} &\leq \frac{C^{k}}{k!}\|\Delta\xi\|_{\star,T}.
\end{align*}}
Therefore, for $k$ large enough, $\Xi$ is a contraction.

\item \textit{Stability with respect to graphon:}
Let $(\tilde{\gamma}_t)_{t \geq 0}$ and $(\hat{\gamma}_t)_{t \geq 0}$ be two solutions of Equation  \eqref{graphonbelavkin} associated with graphon $\tilde{w}$ and $\hat{w}$, respectively, 
with the same initial conditions. Then,
{{\small}
\begin{align*}
\quad &\Big\|\tilde{\gamma}_{t,u} - \hat{\gamma}_{t,u}  \Big\|^{2}_{2} = \Big\|\int_{0}^{t}\Big(-\mathrm{i}\big[\tilde{H},\tilde{\gamma}_{s,u} - \hat{\gamma}_{s,u}\big] 
- \mathrm{i}\big[\int_{0}^{1}\tilde{w}(u,v)A^{\mathbb{E}_{\mathbb{P}}[\tilde{\gamma}_{s,v}]}\mathrm{d}v,\tilde{\gamma}_{s,v}\big]\\
&\qquad\qquad\qquad + \mathrm{i}\big[\int_{0}^{1}\hat{w}(u,v)A^{\mathbb{E}_{\mathbb{P}}[\hat{\gamma}_{s,v}]}\mathrm{d}v,\hat{\gamma}_{s,u}\big]\Big)\mathrm{d}s \\
&\qquad + \int_{0}^{t}\Big(L(\tilde{\gamma}_{s,u} - \hat{\gamma}_{s,u})L^{\dagger} - \frac{1}{2}\{L^{\dagger}L,(\tilde{\gamma}_{s,u} - \hat{\gamma}_{s,u})\}\Big)\mathrm{d}s\\
&\qquad + \sqrt{\eta}\int_{0}^{t}\Big( L(\tilde{\gamma}_{s,u} - \hat{\gamma}_{s,u}) + (\tilde{\gamma}_{s,u} - \hat{\gamma}_{s,u})L^{\dagger}- \tr(K\tilde{\gamma}_{s,u})\tilde{\gamma}_{s,u} + \tr(K\hat{\gamma}_{s,u})\hat{\gamma}_{s,u}\Big)\mathrm{d}W_{s}^{u}\Big\|_{2}^{2}\\
\leq\;& 2\Big\|\int_{0}^{t}\big[\tilde{H},\tilde{\gamma}_{s,u} - \hat{\gamma}_{s,u}\big]\mathrm{d}s\Big\|_{2}^{2} 
+ 2\Big\|\int_{0}^{t}\Big[\int_{0}^{1}\tilde{w}(u,v)A^{\mathbb{E}_{\mathbb{P}}[\tilde{\gamma}_{s,v} -\hat{\gamma}_{s,v}]}\mathrm{d}v,\tilde{\gamma}_{s,v}\Big]\mathrm{d}s\Big\|_{2}^{2} \\ 
&+ 2\Big\|\int_{0}^{t}\Big[\int_{0}^{1}\hat{w}(u,v)A^{\mathbb{E}_{\mathbb{P}}[\hat{\gamma}_{s,v}]}\mathrm{d}v,(\hat{\gamma}_{s,u}-\tilde{\gamma}_{s,u})\Big]\mathrm{d}s\Big\|_{2}^{2}\\
&+ 2\Big\|\int_{0}^{t}\Big[\int_{0}^{1}(\tilde{w}(u,v)-\hat{w}(u,v))A^{\mathbb{E}_{\mathbb{P}}[\hat{\gamma}_{s,v}]}\mathrm{d}v,\tilde{\gamma}_{s,u}\Big]\mathrm{d}s\Big\|_{2}^{2}\\
&+ 2\Big\|\int_{0}^{t}\Big(L(\tilde{\gamma}_{s,u} - \hat{\gamma}_{s,u})L^{\dagger} - \frac{1}{2}\{L^{\dagger}L,(\tilde{\gamma}_{s,u} - \hat{\gamma}_{s,u})\}\Big)\mathrm{d}s \Big\|_{2}^{2}\\
&+ 2\eta\Big\|\int_{0}^{t}\Big( L(\tilde{\gamma}_{s,u} - \hat{\gamma}_{s,u}) + (\tilde{\gamma}_{s,u} - \hat{\gamma}_{s,u})L^{\dagger}\Big)\mathrm{d}W_{s}^{u}\Big\|_{2}^{2}\\
&+ 2\eta\Big\|\int_{0}^{t}\Big(\tr\big(K(\tilde{\gamma}_{s,u} - \hat{\gamma}_{s,u})\big)\tilde{\gamma}_{s,u} + \tr\big(K\hat{\gamma}_{s,u}\big)(\tilde{\gamma}_{s,u} -\hat{\gamma}_{s,u})\Big)\mathrm{d}W_{s}^{u}\Big\|_{2}^{2}.
\end{align*}
}

By taking the expectation,
{{\small} 
\begin{align*}
 & \mathbb{E}_{\mathbb{P}}\Big[\Big\|\tilde{\gamma}_{t,u} - \hat{\gamma}_{t,u}  \Big\|^{2}_{2} \Big] \\
&\leq C\,\mathbb{E}_{\mathbb{P}}\Big[\int_{0}^{t}\Big\| \tilde{\gamma}_{s,u} - \hat{\gamma}_{s,u}\Big\|_{2}^{2}\mathrm{d}s \Big] 
+ \mathbb{E}_{\mathbb{P}}\Big[\Big\|\int_{0}^{t}\Big[\int_{0}^{1}(\tilde{w}(u,v)-\hat{w}(u,v))A^{\mathbb{E}_{\mathbb{P}}[\hat{\gamma}_{s,v}]}\mathrm{d}v,\tilde{\gamma}_{s,u}\Big]\mathrm{d}s\Big\|_{2}^{2} \Big],\\
&\leq C\Big(\int_{0}^{t}\mathbb{E}_{\mathbb{P}}\Big[\Big\| \tilde{\gamma}_{s,u} - \hat{\gamma}_{s,u}\Big\|_{2}^{2} \Big]\mathrm{d}s 
+ \int_{0}^{t}\int_{0}^{1}\Big\|(\tilde{w}(u,v)-\hat{w}(u,v))A^{\mathbb{E}_{\mathbb{P}}[\hat{\gamma}_{s,v}]}\Big\|_{2}^{2}\mathrm{d}v\mathrm{d}s\Big),\\
&\stackrel{\text{Lemma~\eqref{mfoperatornorm}}}{\leq} 
C\Big(\int_{0}^{t}\mathbb{E}_{\mathbb{P}}\Big[\Big\| \tilde{\gamma}_{s,u} - \hat{\gamma}_{s,u}\Big\|_{2}^{2} \Big]\mathrm{d}s 
+ \int_{0}^{t}\int_{0}^{1}\Big|(\tilde{w}(u,v)-\hat{w}(u,v))\Big|^{2}\|a\|_{2}^{2}\|\mathbb{E}_{\mathbb{P}}[\hat{\gamma}_{s,v}]\|_{2}^{2}\mathrm{d}v\mathrm{d}s\Big),\\
&\leq C\Big(\int_{0}^{t}\mathbb{E}_{\mathbb{P}}\Big[\Big\| \tilde{\gamma}_{s,u} - \hat{\gamma}_{s,u}\Big\|_{2}^{2} \Big]\mathrm{d}s 
+ T\|a\|_{2}^{2}\int_{0}^{1}\Big|(\tilde{w}(u,v)-\hat{w}(u,v))\Big|^{2}\mathrm{d}v\Big).
\end{align*}
}

Integrating both sides over $I$ and using the boundeness of the graphon, we obtain
{{\small} 
\begin{align*}
\mathbb{E}_{\mathbb{P}}\Big[\int_{0}^{1}\Big\|\tilde{\gamma}_{t,u} - \hat{\gamma}_{t,u}\Big\|^{2}_{2}\mathrm{d}u \Big] \leq& C\Big( \int_{0}^{t}\mathbb{E}_{\mathbb{P}}\Big[\int_{0}^{1}\Big\| \tilde{\gamma}_{s,u} - \hat{\gamma}_{s,u}\Big\|_{2}^{2}\mathrm{d}u \Big]\mathrm{d}s \\
&+ \underbrace{\int_{0}^{1}\int_{0}^{1}\Big|\Big(\tilde{w}(u,v) - \hat{w}(u,v) \Big)\Big|^{}\mathrm{d}v\mathrm{d}u}_{= \|\tilde{w} - \hat{w} \|_{L^1}^{}}\Big).   
\end{align*}}
Applying Gr\"onwall lemma concludes the proof.
\end{enumerate}
\end{proof}

\begin{remark}[Strong Solution]
\iffalse
In finite dimensional case i.e $\dim(\mathbb{H}) < \infty $, thanks to Yamada-Watanabe theorem \cite{ikedawatanabe}[Theorem 1.1 page 148] we have strong solution of the graphon system.
\fi

In the case of perfect measurement where detector efficiency $\eta = 1$, homogeneous graphon $w \equiv 1 $ and bounded coupling operator $L$, the well-posedness in the strong sense was proved by a different method in \cite{kolokoltsov25quantumstoeq}. An extension to the unbounded coupling operator was recently provided in \cite{kolokoltsov25mathematical}.
For others results concerning well-posedness of stochastic Schrödinger equation, see \cite{mora08basic,fagnola12sto,holevo96} in infinite dimensional setting and \cite{pellegrini08diffusive} for the density operator equation in finite dimensional setting.  
\end{remark}

\section{Propagation of chaos: Block-wise case}\label{sec:Chaos}
In this section, we justify the limiting graphon system through a propagation of chaos in the case of complete observation ($\eta = 1$) and block-wise interaction, this means that the intensity of pairwise interaction between two particles converges to a function that is constant on blocks. This setup can be seen as a first step toward a more general justification of the mean-field limit for non-exchangeable quantum systems. 

For clarity, we slightly re-adjust the framework and notation. Specifically, we consider $c$-packets of $N$ quantum particles interacting through a graph $G_{cN} = (V_{cN}, E_{cN})$. The case where the blocks are not of equal size can be treated similarly. Indeed, we can consider $c$ packets, each containing $N_i$ particles, for $i \in [c]$, such that $\sum_{i=1}^{c}N_i = N,$ and $\lim_{N \to \infty} \frac{N_i}{N} = O(1)$. For simplicity of notation and exposition, we focus here on the case of equally-sized blocks, i.e., $N_i = N$ for all $i$.

The dynamics of the system is described by a $cN$-body Belavkin diffusive equation:

{{\small} \begin{align}
\mathrm{d}\boldsymbol{\rho}^{cN}_{t} &= -\mathrm{i}\Big[\sum_{i_l=1_1}^{c_N}\mathbf{\tilde{H}}_{i_l}\! +\! \tfrac{1}{2cN}\sum_{i_l=1_1}^{c_N}\sum_{j_{l'} = 1_1}^{c_N}\!\xi_{j_{l'}i_{l}}\mathbf{A}_{j_{l'}i_{l}}, \boldsymbol{\rho}^{cN}_{t}\Big]\!\mathrm{d}t\nonumber\\
&+\sum_{i_l=1_1}^{c_N}\big(\mathbf{L}_{i_l}\boldsymbol{\rho}^{cN}_{t}\mathbf{L}_{i_l}^{\dagger}-\tfrac{1}{2}\{\mathbf{L}_{i_l}^{\dagger}\mathbf{L}_{i_l},\boldsymbol{\rho}^{cN}_{t}\}\big)\mathrm{d}t\nonumber\\
&\qquad+\sum_{i_l=1_1}^{c_N}\Big(\mathbf{L}_{i_l}\boldsymbol{\rho}^{cN}_{t} + \boldsymbol{\rho}^{cN}_{t}\mathbf{L}_{i_l}^{\dagger} - \tr\big((\mathbf{L}_{i_l}+\mathbf{L}_{i_l}^{\dagger})\boldsymbol{\rho}^{cN}_{t}\big)\boldsymbol{\rho}^{cN}_{t}\Big)\mathrm{d}W_t^{i_l}.\label{cNbelavkin}
\end{align}}
\noindent
Here, the index $i_l$ refers to the particle $l$ in class $i$, and the shorthand notation is defined as:
{{\small}\begin{align*}
    \sum_{i_l = 1_1}^{c_N}f_{i_l} &:= \sum_{i=1}^{c}\sum_{l=1}^{N} f_{i_l},\quad
    \sum_{i_l = 1_1}^{c_N}\sum_{j_{l'} = 1_1}^{c_N} f_{i_lj_{l'}} := \sum_{i=1}^{c}\sum_{l=1}^{N}\sum_{j=1}^{c}\sum_{l'=1}^{N}f_{i_lj_{l'}}.
\end{align*}}
The lifted operators and marginals are given by
{{\small}\begin{align*}
    \mathbf{O}_{i_l} &:= \overbrace{\mathbf{1}\otimes\dots\otimes\mathbf{1}}^{(i-1)N \text{ times}}\otimes\overbrace{\mathbf{1}\otimes\dots\otimes\underbrace{O}_{l-\text{th}}\otimes\dots\otimes\mathbf{1}}^{N \text{ times}}\otimes\overbrace{\mathbf{1}\otimes\dots\otimes\mathbf{1}}^{(c-i)N \text{ times}},\\
    \boldsymbol{\gamma}_{i} &:= \boldsymbol{\gamma}_{i_1},\quad \boldsymbol{g}_{t,i_l} := \mathbf{1} - \boldsymbol{\gamma}_{t,i_l},\\
    \mathbf{O}_{i} &:= \mathbf{O}_{i_1},\quad \mathbf{O}_{i_li_l} := 0,\quad \mathbf{O}_{i_lj_{l'}} = \mathbf{O}_{j_{l'}i_l}, \\
    \rho_{t,i_l} &:=  \tr_{[N]\setminus\{l\}}\big(\tr_{[c]\setminus\{i\}}(\boldsymbol{\rho}_{t}^{cN})\big).
\end{align*}}
The interaction strength coefficient $\xi_{j_{l'}i_l} := \xi_{j_{l'}i_l}^{cN}$ is derived from a step  kernel function  ${w}^{G_{cN}}$,  and is constructed in the two following ways:
\begin{itemize}
    \item $\xi_{j_{l'}i_l} = {w}^{G_{cN}}(\frac{jN - N + l'}{cN},\frac{iN - N + l}{cN})$;
    \item $\xi_{j_{l'}i_l} = \xi_{i_lj_{l'}}= \text{Bernoulli}\Big({w}^{G_{cN}}\big(\frac{jN - N + l'}{cN},\frac{iN - N + l}{cN}\big)\Big)$.
\end{itemize}

The step graphon $w^{G_{cN}}$ converges with respect to the cut metric to the block-wise graphon function $w \in \mathcal{W}_{0}$, where $w$ is defined as follows.
Let the unit interval $I$ be partitioned into $c$ disjoint subintervals $\displaystyle I = \bigcup_{i = 1}^{c} I_i, \; \text{with } I_i = \Big[\frac{i-1}{c},\frac{i}{c}\Big), \; \text{for }  1\leq i < c, \; \text{and} \; I_{c} = \Big[\frac{c-1}{c},1\Big]$. Then the limiting graphon $w$ is defined by
\begin{align*}
    w(u,v) &:= \begin{cases}
        w_{ji}, \quad (u,v) \in I_{j}\times I_{i},\\
        0,\quad\quad\,\,\,\textrm{otherwise.}
    \end{cases}
\end{align*}

\begin{center}
\begin{tikzpicture}[scale=0.6]

% Premier graphon 4x4
\begin{scope}[shift={(0,0)}]
    \draw (0,0) rectangle (4,4);
    \draw (0,2) -- (4,2);
    \draw (2,0) -- (2,4);
    \fill[black] (2,2) rectangle (4,4); % Quadrant sup droit
    \fill[black] (0,0) rectangle (2,2); % Quadrant inf gauche
\end{scope}

% Deuxième graphon 4x4
\begin{scope}[shift={(6,0)}]
    \draw (0,0) rectangle (4,4);
    \draw (0,2) -- (4,2);
    \draw (2,0) -- (2,4);
    \fill[black] (0,2) rectangle (2,4); % Quadrant sup gauche
    \fill[black] (2,0) rectangle (4,2); % Quadrant inf droit
\end{scope}

% Troisième graphon (3x3 mis à l'échelle pour faire 4x4)
\begin{scope}[shift={(12,0)}, scale=4/3]
    \foreach \x in {0,1,2,3} {
        \draw (\x,0) -- (\x,3);
        \draw (0,\x) -- (3,\x);
    }
    \fill[white] (0,2) rectangle (1,3);
    \fill[gray] (1,2) rectangle (2,3);
    \fill[white] (2,2) rectangle (3,3);

    \fill[gray] (0,1) rectangle (1,2);
    \fill[white] (1,1) rectangle (2,2);
    \fill[gray!80] (2,1) rectangle (3,2);

    \fill[white] (0,0) rectangle (1,1);
    \fill[gray!80] (1,0) rectangle (2,1);
    \fill[white] (2,0) rectangle (3,1);
\end{scope}

% Quatrième graphon : 8x8
% On le place à droite, donc à x=18 (6 unités plus loin que le 3e)
\begin{scope}[shift={(18,0)}, scale = 1/2]
    % Dessin de la grille 8x8
    \foreach \x in {0,...,8} {
        \draw (\x,0) -- (\x,8);
        \draw (0,\x) -- (8,\x);
    } 

% First
    \fill[black!10] (0,0) rectangle (1,1);
    \fill[black!20] (0,1) rectangle (1,2);
    \fill[black!30] (0,2) rectangle (1,3);
    \fill[black!40] (0,3) rectangle (1,4);
    \fill[black!50] (0,4) rectangle (1,5);
    \fill[black!60] (0,5) rectangle (1,6);
    \fill[black!70] (0,6) rectangle (1,7);
    
% Second

    \fill[black!10] (1,1) rectangle (2,2);
    \fill[black!20] (1,2) rectangle (2,3);
    \fill[black!30] (1,3) rectangle (2,4);
    \fill[black!40] (1,4) rectangle (2,5);
    \fill[black!50] (1,5) rectangle (2,6);
    %Jump
    \fill[black!70] (1,7) rectangle (2,8);

% Third
    \fill[black!10] (2,2) rectangle (3,3);
    \fill[black!20] (2,3) rectangle (3,4);
    \fill[black!30] (2,4) rectangle (3,5);
    %Jump
    \fill[black!50] (2,6) rectangle (3,7);
    \fill[black!60] (2,7) rectangle (3,8);

% Fourth
    \fill[black!30] (3,5) rectangle (4,6);
    \fill[black!40] (3,6) rectangle (4,7);
    \fill[black!50] (3,7) rectangle (4,8);

    \fill[black!10] (3,3) rectangle (4,4);
    
    \fill[black!10] (4,4) rectangle (5,5);
    \fill[black!20] (4,5) rectangle (5,6);
    \fill[black!30] (4,6) rectangle (5,7);
    \fill[black!40] (4,7) rectangle (5,8);

    \fill[black!10] (5,5) rectangle (6,6);
    \fill[black!20] (5,6) rectangle (6,7);
    \fill[black!30] (5,7) rectangle (6,8);
    
    \fill[black!10] (6,6) rectangle (7,7);
    \fill[black!20] (6,7) rectangle (7,8);

    \fill[black!10] (7,7) rectangle (8,8);

\end{scope}
\end{tikzpicture}

\textbf{Schematic picture:} Differents scenarios of block-wise graphons.
\end{center}

If the initial state is pure, i.e., {$\boldsymbol{\rho}^{cN}_0 = \ket{\boldsymbol{\Psi}_{0}^{cN}}\bra{\boldsymbol{\Psi}_{0}^{cN}}$, then, by Proposition \ref{puritypreservation}, the state remains pure i.e $\boldsymbol{\rho}^{cN}_t = \ket{\boldsymbol{\Psi}_{t}^{cN}}\bra{\boldsymbol{\Psi}_{t}^{cN}}$} and can be written in the following pure state representation:
{{\small}
\begin{align*}
\mathrm{d}\boldsymbol{\Psi}_{t}^{cN} &= -\Big(\mathrm{i}\Big(\sum_{i_l=1_1}^{c_N}\mathbf{\tilde{H}}_{i_l} + \frac{1}{2cN}\sum_{j_{l'}=1_1}^{c_N}\xi_{j_{l'}i_l}\mathbf{A}_{j_{l'}i_l} -  \langle \frac{\mathbf{L}_{i_l} + \mathbf{L}_{i_l}^{\dagger}}{2}\rangle_{\boldsymbol{\Psi}_{t}^{cN}}\frac{\mathbf{L}_{i_l} - \mathbf{L}_{i_l}^{\dagger}}{2\mathrm{i}}\Big)\\ 
&\qquad + \frac{1}{2}\sum_{i_l=1_1}^{c_N}\Big((\mathbf{L}_{i_l} - \langle \frac{\mathbf{L}_{i_l} + \mathbf{L}_{i_l}^{\dagger}}{2}\rangle_{\boldsymbol{\Psi}_{t}^{cN}})^{\dagger}(\mathbf{L}_{i_l} - \langle \frac{\mathbf{L}_{i_l} + \mathbf{L}_{i_l}^{\dagger}}{2}\rangle_{\boldsymbol{\Psi}_{t}^{cN}})\Big) \Big){\boldsymbol{\Psi}_{t}^{cN}}\mathrm{d}t\\ &+ \sum_{i_l=1_1}^{c_N}\Big( \mathbf{L}_{i_l} -\langle \frac{\mathbf{L}_{i_l} + \mathbf{L}_{i_l}^{\dagger}}{2}\rangle_{\boldsymbol{\Psi}_{t}^{cN}} \Big)\boldsymbol{\Psi}_{t}\mathrm{d}W_t^{i_l}. 
\end{align*}
}

\noindent
A key element in deriving a mean-field limit within quantum systems revolves around the principle of indistinguishability. Unlike classical systems, quantum mechanics introduces two distinct categories of indistinguishable particles: fermions and bosons. For the reader’s convenience a detailed discussion of these categories is reviewed in Appendix~\ref{sec:appA}. 

In the subsequent analysis, we investigate a multi-class bosonic system, which implies symmetry by class for the wave-function state. 

\begin{proposition}[Class‐wise symmetry]
\label{prop:symmetry}
Suppose that the initial wave function \\
 $\boldsymbol{\Psi}_{0}^{cN}(x_{1_1},\dots,x_{c_N})$ is invariant in law under any permutation $\sigma$ applied to the $N$ particles of each class, i.e.,
{{\small} \begin{align}
\textrm{Law}\Big(\boldsymbol{\Psi}_{0}^{cN} \bigl(x_{1_1},\dots,x_{c_N}\bigr)\Big)
&= \textrm{Law}\Big({\Psi}_{0}^{cN} \Bigl(x_{1_{\sigma^{1}(1)}},\dots,
  x_{1_{\sigma^{1}(N)}}, \dots, 
  x_{c_{\sigma^{c}(1)}},\dots,x_{c_{\sigma^{c}(N)}}\Bigr)\Big),
\label{cNSymetry}
\end{align}}
for all permutations $\sigma^1,\dots,\sigma^c \in \mathfrak{S}_N$.  Then this symmetry is preserved in distribution by the \emph{$cN$-body Belavkin Equation} \eqref{cNbelavkin} \emph{if} the interaction coefficients $\{\xi_{j_{l'}i_{l}}\}$ are invariants in law:
$$
\forall i,j\in [c],\;\forall\,l,l'\in [N], 
\quad \; \; \textrm{Law}(\xi_{j_{\sigma^{j}({l'})}i_{\sigma^{i}({l})}}) =  \textrm{Law}(\xi_{j_{l'}i_{l}}). 
$$
\end{proposition}

\begin{proof}
The result follows directly from the class-wise symmetry of interactions and the i.i.d property of the Wiener processes. 
\end{proof}
We make the following assumptions to prove the propagation of chaos for the case of block-wise interaction. 
\begin{assumption}
The initial state $\boldsymbol{\rho}_{0}^{cN}$ is pure and the associated wave-function $\boldsymbol{\Psi}_{0}^{cN}$ is symmetric within each class in law. 
\label{assumption1}
\end{assumption}

\begin{assumption}
The step kernel graphon $w^{G_{cN}}$ is symmetric. 
\label{assumption2}
\end{assumption}

Under Assumption \ref{assumption1} and  Assumption \ref{assumption2}, Equation  \eqref{cNbelavkin} describes the dynamics of $c$-class of Bosons. 

\subsection{Limiting system}
In the mean-field limit, the system \eqref{graphonbelavkin} becomes 
{{\small} \begin{align*}
    \mathrm{d}\gamma_{t,u} &= -\mathrm{i}\Big[\tilde{H} + \int_{0}^{1}w(u,v)A^{\mathbb{E}_{\mathbb{P}}[\gamma_{t,v}]}\mathrm{d}v, \gamma_{t,u} \Big]\mathrm{d}t + \Big(L\gamma_{t,u}L^{\dagger} - \frac{1}{2}\{L^{\dagger}L, \gamma_{t,u}\}\Big)\mathrm{d}t\\
&\quad\quad\quad + \Big(L\gamma_{t,u} + \gamma_{t,u}L^{\dagger} - \tr\big((L+L^{\dagger})\gamma_{t,u}\big)\gamma_{t,u}\Big)\mathrm{d}W_t^{u}, \quad \quad \quad u \in [0,1].
\end{align*}}
Thanks to symmetry, for all  $v,v' \in I_{j},\; j \in [c] $ , 
$$\text{Law}(\gamma_{t,v}) = \text{Law}(\gamma_{t,v'}),$$ 
so we can select a representative in each class $j$, denoted by $\gamma_{t,j}$.  Moreover, 
{{\small}
\begin{align*}
  \int_{0}^{1}{w}\Big(\frac{j-1}{c},v\Big)A_{j}^{\mathbb{E}_{\mathbb{P}}[{\gamma_{t,v}}]}\mathrm{d}v 
  &= \sum_{i=1}^{c}\int_{0}^{1}\mathbb{I}_{ \Big\{ \frac{i-1}{c} \leq v < \frac{i}{c} \Big\}}{w}\Big(\frac{j-1}{c},v\Big)A_{j}^{\mathbb{E}_{\mathbb{P}}[{\gamma_{t,v}}]}\mathrm{d}v,\\
  &= \sum_{i=1}^{c} \int_{0}^{1}\mathbb{I}_{ \Big\{ \frac{i-1}{c} \leq v < \frac{i}{c} \Big\}}{w}\Big(\frac{j-1}{c},\frac{i-1}{c}\Big)A_{j}^{\mathbb{E}_{\mathbb{P}}[{\gamma_{t,v}}]}\mathrm{d}v,\\
  &= \sum_{i=1}^{c}\int_{\frac{i-1}{c}}^{\frac{i}{c}}{w}\Big(\frac{j-1}{c},\frac{i-1}{c}\Big)A_{j}^{\mathbb{E}_{\mathbb{P}}[{\gamma_{t,i}}]}\mathrm{d}v,\\
  &= \frac{1}{c}\sum_{i = 1}^{c}{w}\Big(\frac{j-1}{c},\frac{i-1}{c}\Big)A_{j}^{\mathbb{E}_{\mathbb{P}}[{\gamma_{t,i}}]},\\
  %\int_{0}^{1}{w}\Big(\frac{j-1}{c},v\Big)A_{j}^{\mathbb{E}_{\mathbb{P}}[{\gamma_{t,v}}]}\mathrm{d}v 
  &= 
  \frac{1}{c}\sum_{i=1}^{c}w_{ji}A_{j}^{\mathbb{E}_{\mathbb{P}}[{\gamma_{t,i}}]}.
\end{align*}
}

Thus, the graphon system can be written using $c$ representative particles $\gamma_{j}$, with $ j \in [c]$, as
{{\small} \begin{align}
\mathrm{d}\gamma_{t,j} &= -\mathrm{i}\Big[\tilde{H} + \frac{1}{c}\sum_{i=1}^{c}w_{ji}A^{\mathbb{E}_{\mathbb{P}}[\gamma_{t,i}]}, \gamma_{t,j} \Big]\mathrm{d}t + \big(L\gamma_{t,j}L^{\dagger} - \frac{1}{2}\{L^{\dagger}L,\gamma_{t,j}\}\big)\mathrm{d}t\nonumber\\
&\qquad\qquad +\Big(L\gamma_{t,j} + \gamma_{t,j}L^{\dagger} - \tr\big((L+L^{\dagger})\gamma_{t,j})\gamma_{t,j}\Big)\mathrm{d}W_t^{j}, \;\; j \in [c]. \label{classgraphonsys}
\end{align}}

\subsection{Propagation of chaos}
The propagation of chaos can now be stated as follows:
\begin{align*}
\boldsymbol{\rho}_{0}^{cN} = \rho_{0}^{\otimes (cN)} \xrightarrow[]{\text{Propagation of chaos}} \boldsymbol{\rho}_{t}^{cN}  \approx \bigotimes_{j=1}^{c}\bigotimes_{l=1}^{N}\gamma_{t,j_l},
\end{align*}
where for each class, we consider $N$ independent copies, each driven by its own Wiener process. Denote by $\gamma_{j_{l'}}$ the $l'$-th particle in class $j$, whose dynamics follow the SDE:
{{\small} \begin{align*}
\mathrm{d}\gamma_{t,j_{l'}} &= -\mathrm{i}\Big[\tilde{H} + \frac{1}{c}\sum_{i=1}^{c}w_{ji}A^{\mathbb{E}_{\mathbb{P}}[\gamma_{t,i}]}, \gamma_{t,j_{l'}} \Big]\mathrm{d}t + \big(L\gamma_{t,j_{l'}}L^{\dagger} - \frac{1}{2}\{L^{\dagger}L,\gamma_{t,j_{l'}}\}\big)\mathrm{d}t\nonumber\\
&\qquad\qquad +\Big(L\gamma_{t,j_{l'}} + \gamma_{t,j_{l'}}L^{\dagger} - \tr\big((L+L^{\dagger})\gamma_{t,j_{l'}}\big)\gamma_{t,j_{l'}}\Big)\mathrm{d}W_t^{j_{l'}}.
\end{align*}}

To address the mean-field limit, we extend the method introduced by Pickl \cite{pickl11simple} and later adapted to the stochastic setting in \cite{kolokoltsov21law,kolokoltsov22qmfg} to the case of a multi-class bosonic system.

The goal is to introduce a quantity that detects whether $\boldsymbol{\rho}^{cN}_t$ remains tensorized in a “good way’’ at each time. To this end, it suffices to show that the number of “bad’’ particles, namely particles deviating from the desired state $\boldsymbol{\gamma}_{i_l}$, is small.

To quantify the ratio of bad particles, we use a counting operator
that runs through all possible combinations of bad particles in each class. We introduce the operator counting the number of bad particles in class $i$:
\begin{align*}
\mathbf{P}_{i,n_i} &:= \sum_{{{s_k \in \{0,1\}^{N},\sum s_k = n_i}}}\prod_{l=1}^{N}\boldsymbol{g}_{i_l}^{s_k}\boldsymbol{\gamma}_{i_l}^{1-s_k},   
\end{align*}
where $n_{i} \in \{0,\dots,N\}$ and $i \in [c]$. We then define the total bad-particle counting operator as
\begin{align*}
\mathbf{P}_{\hat{n}} &:= \prod_{i=1}^{c}\mathbf{P}_{i,n_i}, \qquad \hat{n} = \begin{pmatrix}
    n_1\\
    \vdots\\
    n_c
\end{pmatrix}.
\end{align*}

We can now define a functional that quantifies the number of bad particles in the system by weighting all possible combinations of bad particles.  

For any sub-additive function $\varpi : \{0,\dots,cN\} \to \mathbb{R}_{+}$, we set
{{\small} \begin{align*}
\mathcal{E}_{N,\varpi} &:= \tr\Big(\Big(\sum_{n=0}^{cN}\varpi(n)\mathbf{P}_{\hat{n}}\Big)\boldsymbol{\rho}^{cN}\Big),
\end{align*}}
which can equivalently be written as
{{\small} \begin{align*}
\mathcal{E}_{N,\varpi} &:= \tr\Big(\Big(\sum_{n_1,\dots,n_c=0}^{N}\varpi(n_1+\dots+n_c)\mathbf{P}_{\hat{n}}\Big)\boldsymbol{\rho}^{cN}\Big).
\end{align*}}

\begin{lemma}
For any sub-additive function $\varpi : \{0,\dots,cN\} \to \mathbb{R}_{+}$ we have
$$ \mathcal{E}_{N,\varpi}(t) \leq \sum_{i=1}^{c}\mathcal{E}_{N,\varpi}^{i}(t), \quad \forall t \geq 0, $$
where $\mathcal{E}_{N,\varpi}^{i} := \tr\Big( \sum_{n_i=0}^{N}\varpi(n_i)\mathbf{P}_{i,n_i}\boldsymbol{\rho}^{cN}\Big)$.

Moreover, if we count in an unbiased way, i.e., $\varpi(k) = \tfrac{k}{cN}$, then
$$\mathbb{E}_{\mathbb{P}}\Big[\mathcal{E}_{N,\varpi}(t)\Big] \leq \frac{1}{c}\sum_{i=1}^{c}\mathbb{E}_{\mathbb{P}}\Big[\mathcal{D}^{i_l}_N(t)\Big], \quad \forall l \in [N],$$
where $\mathcal{D}^{i_l}_{N} := \tr\!\big(\boldsymbol{g}_{i_l}\boldsymbol{\rho}^{cN}\big)$.
\end{lemma}

%En gros la sous-additivité signifie qu'il n'est pas utile de compter plusieurs fois à quel point le système est " mauvais " lorsque le nombre total d'erreurs est réparti entre différentes parties. Il est possible de comptabiliser les erreurs par classe (ou par bloc) et de les combiner sans faire exploser le " coût " total.

\begin{proof}
We have
{{\small} \begin{align*}
    \mathcal{E}_{N,\varpi}(t) &= \tr\Big(\Big(\sum_{n_1,\dots,n_c=0}^{N}\varpi({n_1 + \dots + n_c})\mathbf{P}_{\hat{n}}\Big)\boldsymbol{\rho}_{t}^{cN}\Big)\\
    &= \tr\Big(\Big(\sum_{n_1,\dots,n_c=0}^{N}\varpi({n_1 + \dots + n_c})\prod_{j=1}^{c}\mathbf{P}_{{j,n_j}}\Big)\boldsymbol{\rho}_{t}^{cN}\Big)\\
    &= \langle\boldsymbol{\Psi}_{t}^{cN}, \sum_{n_1,\dots,n_c = 0}^{N}\varpi(n_1+\dots+n_c)\prod_{j=1}^{c}\mathbf{P}_{j,n_j}\boldsymbol{\Psi}_{t}^{cN} \rangle\\ 
    &\stackrel{\text{sub-additivity of }\varpi}\leq \langle\boldsymbol{\Psi}_{t}^{cN}, \sum_{n_1,\dots,n_c = 0}^{N}\sum_{i=1}^{c}\Big(\varpi(n_i)\prod_{j=1}^{c}\mathbf{P}_{j,n_j}\Big)\boldsymbol{\Psi}_{t}^{cN} \rangle\\
    &\stackrel{\text{Fubini}}= \langle\boldsymbol{\Psi}_{t}^{cN}, \sum_{i=1}^{c}\sum_{n_1,\dots,n_c = 0}^{N}\Big(\varpi(n_i)\mathbf{P}_{i,n_i}\prod_{j\neq i}^{c}\mathbf{P}_{j,n_j}\Big)\boldsymbol{\Psi}_{t}^{cN} \rangle\\
    &= \langle\boldsymbol{\Psi}_{t}^{cN}, \sum_{i=1}^{c}\Big(\sum_{n_i=0}^{N}\varpi(n_i)\mathbf{P}_{i,n_i}\Big)\Big(\prod_{j\neq i}^{c}\underbrace{\sum_{n_j=0}^{N}\mathbf{P}_{j,n_j}}_{\mathbf{1}}\Big)\boldsymbol{\Psi}_{t}^{cN} \rangle\\
    &= \sum_{i=1}^{c}\langle\boldsymbol{\Psi}_{t}^{cN}, \sum_{n_i=0}^{N}\varpi(n_i)\mathbf{P}_{i,n_i}\boldsymbol{\Psi}_{t}^{cN} \rangle\\
    &:= \sum_{i=1}^{c}\mathcal{E}_{N,\varpi}^{i}(t).
\end{align*}}

In order to prove the second statement, we notice the following identity:
$$ \sum_{l=1}^{N} \boldsymbol{g}_{i_l}\mathbf{P}_{i,n_i} = n_i\mathbf{P}_{i,n_i},$$
 since multiplying the product $\boldsymbol{g}_{i_l}$ 
 by $\prod_{l=1}^{N}\boldsymbol{g}_{i_l}^{s_k}\boldsymbol{\gamma}_{i_l}^{1-s_k}$ is zero unless  if $s_l = 1$, and summing over all $l$ counts the number of indices $l$ such that $s_l = 1$, which is $n_i$. 

Now if $\varpi(k) = \frac{k}{cN}$, we have:
{{\small} 
\begin{align*}
\mathbb{E}_{\mathbb{P}}\Big[\mathcal{E}_{N,\varpi}^{i}(t)\Big] 
&= \sum_{n_i=0}^{N}\mathbb{E}_{\mathbb{P}}\Big[\langle\boldsymbol{\Psi}_{t}^{cN}, \frac{n_i}{cN}\mathbf{P}_{i,n_i}\boldsymbol{\Psi}_{t}^{cN}\rangle\Big],\\
&= \frac{1}{c}\,\mathbb{E}_{\mathbb{P}}\Big[\langle\boldsymbol{\Psi}_{t}^{cN}, \frac{1}{N}\sum_{n_i=0}^{N}{n_i}\mathbf{P}_{i,n_i}\boldsymbol{\Psi}_{t}^{cN}\rangle\Big],\\
&= \frac{1}{c}\,\mathbb{E}_{\mathbb{P}}\Big[\langle\boldsymbol{\Psi}_{t}^{cN}, \frac{1}{N}\sum_{n_i=0}^{N}\sum_{l=1}^{N}{\boldsymbol{g}_{t,i_l}}\mathbf{P}_{i,n_i}\boldsymbol{\Psi}_{t}^{cN}\rangle\Big],\\
&\stackrel{\text{Invariance in law}}{=} \frac{1}{c}\,\mathbb{E}_{\mathbb{P}}\Big[\langle\boldsymbol{\Psi}_{t}^{cN}, \boldsymbol{g}_{t,i_l}\boldsymbol{\Psi}_{t}^{cN}\rangle\Big], \quad \forall l \in [N],\\
&= \frac{1}{c}\mathbb{E}_{\mathbb{P}}\Big[\mathcal{D}_{N}^{i_l}(t)\Big],
\end{align*}
}
which completes the proof.
\end{proof}

The above lemma is crucial because it allows us to compute the number of bad particles class by class, and to work directly with the trace norm. Indeed, thanks to \cite[Lemma~2.3]{knowles10mean}, the quantity $\mathcal{D}_{N}^{i_l}(t)$ provides a control of the trace-norm distance between the reduced state of particle $i_l$ in the $cN$-particle system and the corresponding limiting state in the graphon system. More precisely, one has the bound
$$ \mathcal{D}^{i_l}_{N}(t) \leq 
\big\|\rho_{t,i_l} - \gamma_{t,i_l}\big\|_{1} \leq C\sqrt{\mathcal{D}^{i_l}_{N}(t)},
$$
for some universal constant $C>0$.

We now state some technical lemmas that are useful to establish our result. The following lemmas results from a slight modification of \cite[Lemma A.1]{kolokoltsov22qmfg} and \cite[Lemma 3.10]{knowles10mean}. \begin{lemma}\label{technicallemma} Let $f : \mathbb{N} \to \mathbb{R}$ be a function and define the operator $ \hat{f}_{i} = \sum_{k}f(k)\mathbf{P}_{i,k}$, where $\mathbf{P}_{i,k}$ is an operator counting the number of bad particles in class $i$.

Consider two products $\boldsymbol{Q}_{r_i,r_j}^{(a)}$, $\boldsymbol{Q}_{r_i,r_j}^{(b)}$ defined in the following way: \begin{align*} \boldsymbol{Q}_{r_i,r_j}^{(a)} :=\Big(\prod_{l=1}^{r_i}\boldsymbol{p}_{i_l}^{(a)}\Big)\Big(\prod_{{l'}=1}^{r_j}\boldsymbol{p}_{j_{l'}}^{(a)}\Big), \ \boldsymbol{Q}_{r_i,r_j}^{(b)} :=\Big(\prod_{l=1}^{r_i}\boldsymbol{p}_{i_l}^{(b)}\Big)\Big(\prod_{{l'}=1}^{r_j}\boldsymbol{p}_{j_{l'}}^{(b)}\Big), \end{align*} where each $\boldsymbol{p}_{i_l}^{\bullet}$ (resp. $\boldsymbol{p}_{j_{l'}}^{\bullet}$) belongs to $\{\boldsymbol{\gamma}_{i_l},\boldsymbol{g}_{i_l}\}$ (resp. $\{\boldsymbol{\gamma}_{j_{l'}},\boldsymbol{g}_{j_{l'}}\}$). Define \begin{align*} n_{i}^{\bullet } := \text{Card}\Big\{ l \Big| \boldsymbol{p}_{i_l}^{\bullet} = \boldsymbol{g}_{i_l} \Big\}, \ n_{j}^{\bullet } := \text{Card}\Big\{ l \Big| \boldsymbol{p}_{j_l}^{\bullet} = \boldsymbol{g}_{j_l} \Big\}. \end{align*} For any operator $\mathbf{R}$ on $L^{2}_{\mathbb{C}}(\mathfrak{X}^{N}\times\mathfrak{X}^{N})$ acting only on the $r_i$ first variables of the first class and $r_j$ first variables of the second class, we have \begin{align*} \boldsymbol{Q}_{r_i,r_j}^{(a)}\mathbf{R}\hat{f}_{i}\hat{f}_{j}\boldsymbol{Q}_{r_i,r_j}^{(b)} &= \boldsymbol{Q}_{r_i,r_j}^{(a)}\widehat{\tau_{n_i}f_{i}}\widehat{\tau_{n_j}f_{j}}\mathbf{R}\boldsymbol{Q}_{r_i,r_j}^{(b)}, \end{align*} where $\tau_{n_i}$(resp. $\tau_{n_j}$) is shift function $\tau_{r}f(k) = f(k+r)$ and $n_i = n_i^{(b)} - n_{i}^{(a)}$(resp. $n_j = n_{j}^{(b)} - n_{j}^{(a)}$). \end{lemma}

\begin{lemma}[\cite{knowles10mean}]
Let  $\gamma_{1}, \gamma_{2}$ be operators on $L_{\mathbb{C}}^{2}(\mathfrak{X}^2,\mu^{\otimes 2})$ and let $A$ be a Hilbert-Schmidt operator on 
$L^2_{\mathbb{C}}(\mathfrak{X}^{2},\mu^{\otimes 2})$ with kernel $a$ satisfying the following proprieties:
\begin{align*}
\|a\|^{2}_{2} &= \int_{\mathfrak{X}^4}|a(x,y;x',y')|^{2}\mu(\mathrm{d}x)\mu(\mathrm{d}y)\mu(\mathrm{d}x')\mu(\mathrm{d}y') < \infty,\\
a(x,y;x',y') &= a(y,x;y',x'), \quad a(x,y;x',y') = \overline{a(x',y';x,y)}.
\end{align*}
Then, 
$$ \gamma_{2}A_{12}\gamma_{2} = \gamma_{2}A_{1}^{\gamma_{2}},$$
where $\gamma_{2}$ is viewed as an operator acting on the second variable.
\end{lemma}

\begin{proof}
For any function $f \in L_{\mathbb{C}}^{2}({\mathfrak{X}\times\mathfrak{X}})$,
{{\small} \begin{align*}  (\gamma_{2}&A_{12}\gamma_{2})f(x_1,x_2) =
 \int_{\mathfrak{X}}\gamma_{2}(x_2,y_2)(A_{12}\gamma_{2}f)(x_1,y_2)\mu(\mathrm{d}y_2)\\
&= \int_{\mathfrak{X}}\gamma_{2}(x_2,y_2)\int_{\mathfrak{X}^2}a(x_1,y_2;x_1',z_2)(\gamma_{2}f)(x_1',z_2)\mu(\mathrm{d}x_1')\mu(\mathrm{d}z_2)\mu(\mathrm{d}y_2)\\
&=\int_{\mathfrak{X}^3}\gamma_{2}(x_2,y_2)a(x_1,y_2,x_1',z_2)\int_{\mathfrak{X}}\gamma_{2}(z_{2},x_2')f(x_1',x_2')\mu(\mathrm{d}x_2')\mu(\mathrm{d}x_1')\mu(\mathrm{d}z_2)\mu(\mathrm{d}y_2)\\
&=\int_{\mathfrak{X}^4}\gamma(x_{2},y_{2})a(x_{1},y_{2};x'_{1},z_{2}){\gamma(z_{2},x'_{2})}f(x'_{1},x'_{2})\mu(\mathrm{d}x'_{1})\mu(\mathrm{d}x_{2}')\mu(\mathrm{d}y_{2})\mu(\mathrm{d}z_{2})\\
&\stackrel{\text{Fubini}}= \int_{\mathfrak{X}^2}\Big(\int_{\mathfrak{X}^2}\gamma(x_{2},y_{2})a(x_{1},y_{2};x'_{1},z_{2}){\gamma(z_{2},x'_{2})}\mu(\mathrm{d}y_{2})\mu(\mathrm{d}z_{2})\Big)f(x_{1}',x_{2}')\mu(\mathrm{d}x_{1}')\mu(\mathrm{d}x_{2}')\\
&= \int_{\mathfrak{X}^2}\Big(\int_{\mathfrak{X}^2}\gamma(x_{2},y_{2})a(x_{1},y_{2};x'_{1},z_{2})\overline{\gamma(x'_{2},z_{2})}\mu(\mathrm{d}y_{2})\mu(\mathrm{d}z_{2})\Big)f(x_{1}',x_{2}')\mu(\mathrm{d}x_{1}')\mu(\mathrm{d}x_{2}')\\
&= \int_{\mathfrak{X}^2}\gamma_{2}A_{1}^{\gamma_{2}}(x_{1},x_{2};x_{1}',x_{2}')f(x_1',x_2')\mu(\mathrm{d}x'_1)\mu(\mathrm{d}x_2')\\
&= (\gamma_{2}A_{1}^{\gamma_2})f(x_1,x_2),
\end{align*}}
and the proof follows.
\end{proof}

{The following theorem establishes the propagation of chaos.}

\begin{theorem}[Propagation of chaos estimate]
Under Assumptions \eqref{assumption1}-\eqref{assumption2} and with $\hat{f}(k) = \frac{k}{cN} $, we have
\begin{align*}
\mathbb{E}_{\mathbb{P}}[\mathcal{E}_{N,\hat{f}}(t)] &\leq e^{Ct}\Big(\mathcal{E}_{N,\hat{f}}(0) + \|{w}^{G_{cN}} - {w}\|_{\square} + \frac{1}{\sqrt{N}}\Big),
\end{align*}
where $C  = C(\|a\|_2,\|L\|,c)$.
\end{theorem}

\begin{proof}
We proceed “à la Gr\"onwall''. Recall that in unbiased counting, we have
$$ \mathbb{E}_{\mathbb{P}}[\mathcal{E}_{N,\hat{f}}(t)]\leq \frac{1}{c}\sum_{j=1}^{c}\mathbb{E}_{\mathbb{P}}[\mathcal{D}^{j_l}_{N}(t)], \quad \forall l \in [N].$$

Set $\mathcal{D}_{N}(t) := \sum_{i=1}^{c}\mathcal{D}_N^{j}(t)$. By It\^o formula, we obtain
\begin{align*}
\mathrm{d}\mathcal{D}_{N}^{j}(t) &= -\tr\Big(\boldsymbol{\rho}_{t}^{cN}\mathrm{d}\boldsymbol{\gamma}_{t,j}\Big) - \tr\Big(\boldsymbol{\gamma}_{t,j}\mathrm{d}\boldsymbol{\rho}_{t}^{cN}\Big) - \tr(\mathrm{d}\boldsymbol{\rho}_{t}^{cN}\mathrm{d}\boldsymbol{\gamma}_{t,j}),\\
&= (\mathcal{P}_{1}^{j} + \mathcal{P}_{2}^{j})\mathrm{d}t + \sum_{i=1}^{c}\sum_{l=1}^{N}\mathcal{P}_{3}^{i_l}\mathrm{d}W_{t}^{i_l},
\end{align*}
where $\mathcal{P}_{1}^{j}$ collects the drift terms that are independent of $L_{j} := L_{j_1}$ and $\mathcal{P}_{2}^{j}$ collects the drift terms that depend on $L_{j}$.

The estimates used to bound the quantity $|\mathcal{P}_{2}^{j}|$ are found in the lemmas of  \cite[Appendix A]{kolokoltsov21law}, and the extension for all density operators provided in \cite{chalal23mean}. We have
{{\small} \begin{align*}
\mathcal{P}_{2}^{j} \!&=\! \tr\Big(\Big(\tfrac{1}{2}L_{j}^{\dagger}L_{j}\boldsymbol{\rho}_{t}^{cN}\!+\!\tfrac{1}{2}\boldsymbol{\rho}_{t}^{cN}L_{j}^{\dagger}L_{j}\! - \!L_{j}\boldsymbol{\rho}_{t}^{cN}L_{j}^{\dagger}\Big)\boldsymbol{\gamma_{t,j}}\! +\! \boldsymbol{\rho}_{t}^{cN}\Big(\tfrac{1}{2}L_{j}^{\dagger}L_{j}\boldsymbol{\gamma}_{t,j}\!+\!\tfrac{1}{2}\boldsymbol{\gamma}_{t,j}L^{\dagger}_{j}L_{j} \!-\! L_{j}\boldsymbol{\gamma}_{t,j}L_{j}^{\dagger}\Big)  \\
    & - \Big(\boldsymbol{\rho}_{t}^{cN}L_{j}^{\dagger} + L_{j}\boldsymbol{\rho}_{t}^{cN} - \boldsymbol{\rho}_{t}^{cN}\tr\big((L_{j}+L_{j}^{\dagger})\boldsymbol{\rho}_{t}^{cN}\big)\Big)\Big(\boldsymbol{\gamma}_{t,j}L^{\dagger}_{j} + L_{j}\boldsymbol{\gamma}_{t,j} - \boldsymbol{\gamma}_{t,j}\tr\big((L_{j} + L_{j}^{\dagger})\boldsymbol{\gamma}_{t,j}\big)\Big)\Big).
\end{align*}}
By commutativity under the trace, the terms corresponding to $i_l \neq j_1$ vanish. Then, we have
{{\small} \begin{align*}
 |\mathcal{P}_{2}^{j}| \leq C(\|L\|)\big(1 - \tr(\boldsymbol{\gamma}_{t,j}\boldsymbol{\rho}_{t}^{cN})\big)= C\mathcal{D}_{N}^{j}(t).
\end{align*}}
For $\mathcal{P}_1^{j}$, we have 
\begin{align*}
\mathcal{P}_{1}^{j} &= \mathrm{i}\,\tr\Big(\boldsymbol{\rho}_{t}^{cN}[\tilde{\mathbf{H}}_{j},\boldsymbol{\gamma}_{t,j}] + \boldsymbol{\rho}_{t}^{cN}[\frac{1}{c}\sum_{i = 1}^{c}{w}_{ji}A_{j}^{\mathbb{E}_{\mathbb{P}}[\boldsymbol{\gamma_{t,i}}]},\boldsymbol{\gamma}_{t,j}]  \Big)\\
&\;\;\; + \mathrm{i}\, \tr\Big(\boldsymbol{\gamma}_{t,j}[\sum_{i_l = 1_1}^{c_N}\tilde{\mathbf{H}}_{i_l},\boldsymbol{\rho}_{t}^{cN}] + \boldsymbol{\gamma}_{t,j}[\frac{1}{2cN}\sum_{i_l=1_1}^{c_N}\sum_{h_k = i_l}^{c_N}\xi_{h_{k}i_{l}}A_{h_{k}i_{l}},\boldsymbol{\rho}_{t}^{cN}]  \Big).
\end{align*}
Using commutativity of trace, we can re-write the expression as 
\begin{align*}
\mathcal{P}_{1}^{j} &= \mathrm{i}\,\tr\Big(\big[\tilde{\mathbf{H}}_{j},\boldsymbol{\gamma}_{t,j}\big]\boldsymbol{\rho}_{t}^{cN} + [\frac{1}{c}\sum_{i = 1}^{c}{w}_{ji}A_{j}^{\mathbb{E}_{\mathbb{P}}[{\gamma_{t,i}}]},\boldsymbol{\gamma}_{t,j}]\boldsymbol{\rho}_{t}^{cN}  \Big)\\
&\; + \mathrm{i}\, \tr\Big([\boldsymbol{\gamma}_{t,j},\sum_{i_l = 1_1}^{c_N}\tilde{\mathbf{H}}_{i_l}]\boldsymbol{\rho}_{t}^{cN} + \big[\boldsymbol{\gamma}_{t,j},\frac{1}{2cN}\sum_{i_l=1_1}^{c_N}\sum_{j_{l'} = i_l}^{c_N}\xi_{j_{l'}i_{l}}A_{j_{l'}i_{l}}\big] \boldsymbol{\rho}_{t}^{cN} \Big).
\end{align*}
We now introduce $\boldsymbol{g}_{t,j} = \mathbf{1} - \boldsymbol{\gamma}_{t,j}$. Then,
{{\small} \begin{align*}
\mathcal{P}_{1}^{j} &= \mathrm{i} \,\tr\Big([\sum_{i_l = 1_1}^{c_N}\tilde{\mathbf{H}}_{i_l},\boldsymbol{g}_{t,j}]\boldsymbol{\rho}_{t}^{cN} + [\frac{1}{2cN}\sum_{i_l=1_1}^{c_N}\sum_{j_{l'} = i_l}^{c_N}\xi_{j_{l'}i_{l}}A_{j_{l'}i_{l}},\boldsymbol{g}_{t,j}] \boldsymbol{\rho}_{t}^{cN} \Big)\\ &\; - \mathrm{i}\,\tr\Big([\tilde{\mathbf{H}}_{j},\boldsymbol{g}_{t,j}]\boldsymbol{\rho}_{t}^{cN} + [\frac{1}{c}\sum_{i = 1}^{c}{w}_{ji}A_{j}^{\mathbb{E}[{\gamma_{t,i}}]},\boldsymbol{g}_{t,j}]\boldsymbol{\rho}_{t}^{cN}  \Big).
\end{align*}}
All terms with $\tilde{\mathbf{H}}_{i_l}$ vanish because $[\tilde{\mathbf{H}}_{i_l},\boldsymbol{g}_{t,j}] = 0$ when $i_l \neq j_{1}$~\footnote{This follows since the operators $\mathbf{\tilde{H}}_{i_l}, \boldsymbol{g}_{t,j}$ act on different variables.} and the two remaining $\tilde{\mathbf{H}}_{j}$ terms cancel each other,
{{\small} \begin{align*}
\mathcal{P}_{1}^{j} &= \mathrm{i} \,\tr\Big(\Big[\frac{1}{2cN}\sum_{i_l=1_1}^{c_N}\sum_{j_{l'} = i_l}^{c_N}\xi_{j_{l'}i_{l}}A_{j_{l'}i_{l}}, \boldsymbol{g}_{t,j}\Big]\boldsymbol{\rho}_{t}^{cN}\Big) -\mathrm{i}\,\tr\Big(\Big[\frac{1}{c}\sum_{i = 1}^{c}{w}_{ji}A_{j}^{\mathbb{E}_{\mathbb{P}}[\boldsymbol{\gamma_{t,i}}]}, \boldsymbol{g}_{t,j}\Big]\boldsymbol{\rho}_{t}^{cN}\Big)\\
 &= \mathrm{i} \,\tr\Big(\Big[\frac{1}{2cN}\sum_{i_l=1_1}^{c_N}\sum_{h_k = i_l}^{c_N}\xi_{j_{l'}i_{l}}A_{j_{l'}i_{l}} - \frac{1}{c}\sum_{i = 1}^{c}{w}_{ji}A_{j}^{\mathbb{E}_{\mathbb{P}}[\boldsymbol{\gamma_{t,i}}]},\boldsymbol{g}_{t,j}\Big]\boldsymbol{\rho}_{t}^{cN}\Big).
\end{align*}}
Similarly, if $j_{l'} \neq j_1 $ {and} $i_l \neq j_1$, then $[A_{j_{l'}i_l}, \boldsymbol{g}_{j}] = 0$. Furthermore, remind that $\xi_{ji_l} = \xi_{i_lj}, A_{ji_l} = A_{i_lj}$. Therefore, 
{{\small} \begin{align*}
\mathcal{P}_{1}^{j} &= \mathrm{i}\,\tr\Big(\Big[\frac{1}{cN}\sum_{i_{l} \neq j_{1}}^{}\xi_{ji_{l}}A_{ji_{l}} - \frac{1}{c}\sum_{i = 1}^{c}{w}_{ji}A_{j}^{\mathbb{E}[{\gamma_{t,i}}]},\boldsymbol{g}_{t,j}\Big]\boldsymbol{\rho}_{t}^{cN}\Big).
\end{align*}}
Adding and subtracting the graphon term $w$,
{{{\small} 
\begin{align*}
\mathcal{P}_{1}^{j} &=\mathrm{i}\, \tr\Big(\!\Big[\frac{1}{cN}\sum_{i_{l} \neq j_{1}}^{}\big(\xi_{ji_{l}} - {w}_{ji}\big)A_{ji_{l}} + \frac{1}{cN}\sum_{i_{l} \neq j_{1}}{w}_{ji}A_{ji_{l}}\! - \!\frac{1}{c}\sum_{i = 1}^{c}{w}_{ji}A_{j}^{\mathbb{E}_{\mathbb{P}}[{\gamma_{t,i}}]},\boldsymbol{g}_{t,j}\Big]\boldsymbol{\rho}_{t}^{cN}\Big)\\
%\mathcal{P}_{1}^{j} 
= \mathrm{i}\!&\underbrace{\tr\Big(\!\Big[\tfrac{1}{cN}\sum_{i_{l} \neq j_{1}}^{}\big(\xi_{ji_{l}} - {w}_{ji}\big)A_{ji_{l}},\boldsymbol{g}_{t,j}\Big]\boldsymbol{\rho}_{t}^{cN}\Big)}_{(\star)^{j_1}}\!+ \mathrm{i}\!\underbrace{\tr\Big(\!\Big[\tfrac{1}{cN}\sum_{i_{l} \neq j_{1}}^{}{w}_{ji}A_{ji_{l}}\!-\!\tfrac{1}{c}\sum_{i = 1}^{c}{w}_{ji}A_{j}^{\mathbb{E}_{\mathbb{P}}[{\gamma_{t,i}}]},\boldsymbol{g}_{t,j}\Big]\!\boldsymbol{\rho}_{t}^{cN}\Big)}_{(\star\star)^{j_1}}.
\end{align*}
}} 

We now proceed to estimate the two terms $(\star)^{j_1}$ and $(\star\star)^{j_1}$:
\paragraph*{Estimate $(\star)^{j_1}$.} We have
{{\small} 
\begin{align*}
    \Big|\mathbb{E}_{\mathbb{P}}[(\star)^{j_1}]\Big| &\leq \Big|\mathbb{E}_{\mathbb{P}}\Big[\tr\Big(\!\Big[\frac{1}{cN}\sum_{i_{l} \neq j_{1}}^{}\big(\xi_{ji_{l}} - {w}_{ji}\big)A_{ji_{l}},\boldsymbol{g}_{t,j}\Big]\boldsymbol{\rho}_{t}^{cN}\Big)\Big]\Big|\\
    &\leq C\Big|\mathbb{E}_{\mathbb{P}}\Big[\frac{1}{cN}\sum_{i_l \neq j_1}\Big( \xi_{ji_{l}} - {w}_{ji} \Big)\Big] \Big|
\end{align*}
}

    By adding and substracting a step graphon kernel $\mathrm{w}^{G_{cN}}$, we obtain 
   {{\small} 
\begin{align*}
    \Big|\mathbb{E}_{\mathbb{P}}[(\star)^{j_1}]\Big|
    &\leq C\Big(\Big|\frac{1}{cN}\sum_{i_l \neq j_1}\underbrace{\mathbb{E}_{\mathbb{P}}\Big[\Big( \xi_{ji_{l}} - {w}^{G_{cN}}\big(\frac{jN - N +1}{cN},\frac{iN - N + l}{cN}\big) \Big)\Big]}_{0}\Big|\\
    &\quad\quad  + \Big|\frac{1}{cN}\sum_{i_l \neq j_1}\mathbb{E}_{\mathbb{P}}\Big[\Big( {w}^{G_{cN}}\big(\frac{jN - N +1}{cN},\frac{iN - N + l}{cN}\big) - {w}_{ji} \Big)\Big]\Big|\Big).
\end{align*}
}

Taking the average over $j$,
{{\small} 
\begin{align*}
\frac{1}{c}\sum_{j=1}^{c}\big|\mathbb{E}_{\mathbb{P}}[(\star)^{j_1}]\big| &= \frac{1}{cN}\sum_{j_{l'}=1_1}^{c_N}\Big|\mathbb{E}_{\mathbb{P}}[(\star)^{j_{l'}}]\Big|\\
&\leq C\Big(\big(\frac{1}{cN}\big)^2\sum_{j_{l'}=1_1}^{c_N}\Big|\sum_{i_l=1_1}^{c_N}\mathbb{E}_{\mathbb{P}}\big[\Big(w^{G_{cN}}(\frac{jN - N + l'}{cN},\frac{iN - N + l}{cN}) - w_{ji}\big]\Big|\\
&\leq C\|\mathcal{T}_{w^{G_{cN}}-w}\|_{\text{op}} \stackrel{\eqref{lovasz}}\leq C\|w^{G_{cN}}- w \|_{\square}.
\end{align*}
}

\paragraph*{Estimate $(\star\star)^{j_1}$.} We have
{{\small} \begin{align*}
(\star\star)^{j_1} &= \tr\Big(\big[\frac{1}{cN}\sum_{i_{l} \neq j_{1}}{w}_{ji}A_{ji_{l}} - \frac{1}{c}\sum_{i=1}^{c}{w}_{ji}A_{j}^{\mathbb{E}_{\mathbb{P}}\big[\boldsymbol{\gamma}_{t,i}\big]},\boldsymbol{g}_{t,j} \big]\boldsymbol{\rho}_{t}^{cN}\Big) \\ 
&=\frac{1}{c}\tr\Big(\Big[\frac{1}{N}\Big(\sum_{i \neq j}\sum_{l=1}^{N}{w}_{ji}A_{ji_{l}} + \sum_{l=2}^{N}{w}_{jj}A_{jj_l}\Big) - \sum_{i=1}^{c}{w}_{ji}A_{j}^{\mathbb{E}_{\mathbb{P}}\big[\boldsymbol{\gamma}_{t,i}\big]},\boldsymbol{g}_{t,j}\Big]\boldsymbol{\rho}_{t}^{cN}\Big) \\
&=\frac{1}{c}\!\tr\Big(\Big[\frac{1}{N}\sum_{i \neq j}\sum_{l=1}^{N}{w}_{ji}A_{ji_{l}}  - \sum_{i \neq j}^{c}{w}_{ji}A_{j}^{\mathbb{E}_{\mathbb{P}}\big[\boldsymbol{\gamma}_{t,i}\big]} \!+\! \frac{1}{N}\sum_{l=2}^{N}{w}_{jj}A_{jj_l} - {w}_{jj}A_{j}^{\mathbb{E}_{\mathbb{P}}\big[\boldsymbol{\gamma}_{t,j}\big]},\boldsymbol{g}_{t,j}\Big]\boldsymbol{\rho}_{t}^{cN}\Big)\\
&=\frac{1}{c}\tr\Big(\Big[\sum_{i \neq j}^{c}{w}_{ji}\Big(\frac{1}{N}\sum_{l=1}^{N}A_{ji_{l}}  - A_{j}^{\mathbb{E}\big[\boldsymbol{\gamma}_{t,i}\big]}\Big) + {w}_{jj}\Big(\frac{1}{N}\sum_{l=2}^{N}A_{jj_l} - A_{j}^{\mathbb{E}_{\mathbb{P}}\big[\boldsymbol{\gamma}_{t,j}\big]}\Big),\boldsymbol{g}_{t,j}\Big]\boldsymbol{\rho}_{t}^{cN}\Big).
\end{align*}}
We split $(\star\star)^{j_1}$ into two parts, respectively containing interactions inside block $j$ and those outside block $j$:
\begin{align*}(\star\star)^{j_1} &=  (\star\star\star)_{1}^{j_1} + (\star\star\star)_{2}^{j_1},
\end{align*}
where:
\begin{enumerate}
\item Inter-block interaction:
{{\small} \begin{align*}
(\star\star\star)_{1}^{j_1} &:= \frac{1}{c}\tr\Big(\Big[\sum_{i \neq j}^{c}{w}_{ji}\Big(\frac{1}{N}\sum_{l=1}^{N}A_{ji_{l}}  - A_{j}^{\mathbb{E}_{\mathbb{P}}\big[\boldsymbol{\gamma}_{t,i}\big]}\Big),\boldsymbol{g}_{t,j}\Big]\boldsymbol{\rho}_{t}^{cN}\Big)
\end{align*}}
\item Intra-block interaction:
{{\small} \begin{align*}
(\star\star\star)_{2}^{j_1} &:= \frac{1}{c}\tr\Big(\Big[{w}_{jj}\Big(\frac{1}{N}\sum_{l=2}^{N}A_{jj_l} - A_{j}^{\mathbb{E}_{\mathbb{P}}\big[\boldsymbol{\gamma}_{t,j}\big]}\Big),\boldsymbol{g}_{t,j}\Big]\boldsymbol{\rho}_{t}^{cN}\Big).
\end{align*}}
\end{enumerate}

We focus below on the inter-block term $(\star\star\star)_1^{j_1}$. The treatment of the intra-block term $(\star\star\star)_2^{j_1}$ follows the same arguments used for the homogeneous case in \cite{kolokoltsov22qmfg}.

\paragraph*{Estimate $(\star\star\star)_{1}^{j_1}$.}
To eliminate the mean-field interaction term, we approximate expectation by empirical means:

{{\small} \begin{align*}
    \delta^{j_{l'},N}_{t} &:= \frac{1}{N}\sum_{l'' = 1 }^{N}\gamma_{t,j_{l''}} - \mathbb{E}_{\mathbb{P}}[\gamma_{t,j_{l'}}],\quad
    A_{j}^{\mathbb{E}_{\mathbb{P}}[\gamma_{t,j_{l'}}]}:= \frac{1}{N}\sum_{l''= 1}^{N}A_{j}^{\gamma_{t,j_{l''}}} - A_{j}^{\delta_{t}^{j_{l'}},N}.
\end{align*}}
Since $\mathbb{E}_{\mathbb{P}}\big[\delta_{t}^{j,l,N}\big]=0$, and by law of large numbers, we obtain
{{\small} 
\begin{align}    \mathbb{E}_{\mathbb{P}}\Big[\|\delta_{t}^{j,l,N}\|_{2}^{2}\Big] &\stackrel{}= \mathbb{V}\Big[\|\delta_{t}^{j,l,N}\|_{2}^{2}\Big]
    \leq \frac{1}{N}\mathbb{V}\Big[\|\gamma_{t,j_l}^{}\|_{2}^{2}\Big],
\label{lln}
\end{align}}
where $\mathbb{V}$ is the variance under $\mathbb{P}$, then, 
{{\small}\begin{align*}
    (\star\star\star)_{1}^{j_1} &= \frac{1}{c}\tr\Big(\Big[\sum_{i\neq j}^{c}{w}_{ji}\Big(\frac{1}{N}\sum_{l=1}^{N}A_{ji_l} - \frac{1}{N}\sum_{l=1}^{N}A_{j}^{\gamma_{t},i_l} + A_{j}^{\delta^{i_1,N}}\Big),\boldsymbol{g}_{t,j}\Big]\boldsymbol{\rho}_{t}^{cN}\Big),
\end{align*}}
and we decompose the term $(\star\star\star)_{1}^{j_1} = (\star\star\star)_{11}^{j_1} + (\star\star\star)_{12}^{j_1}$, where
{{\small}\begin{align*}
    (\star\star\star)_{11}^{j_1} &= \frac{1}{c}\tr\Big(\Big[\sum_{i\neq j}^{c}{w}_{ji}A_{j}^{\delta^{i_1,N}},\boldsymbol{g}_{t,j}\Big]\boldsymbol{\rho}_{t}^{cN}\Big),\\
    (\star\star\star)_{12}^{j_1} &= \frac{1}{c}\tr\Big(\Big[\sum_{i\neq j}^{c}{w}_{ji}\Big(\frac{1}{N}\sum_{l=1}^{N}A_{ji_l} - \frac{1}{N}\sum_{l=1}^{N}A_{j}^{\gamma_{t},i_l} \Big),\boldsymbol{g}_{t,j}\Big]\boldsymbol{\rho}_{t}^{cN}\Big).
\end{align*}}

\paragraph*{Estimate $(\star\star\star)_{11}^{j_1}$.}
  We have: 
    {{\small} \begin{align*}
        \mathbb{E}_{\mathbb{P}}\big[|(\star\star\star)_{11}^{j_1}|\big] &\leq C\mathbb{E}_{\mathbb{P}}\Big[\Big\|A_{j}^{\delta^{i_1,N}_t}\Big\|_{2}\Big]
        \stackrel{\text{Lemma}   \ref{mfoperatornorm}}\leq C\mathbb{E}_{\mathbb{P}}[\big\|a\big\|_{2}\big\|\delta^{i,N}_t\big\|_{2}]\stackrel{\text{C-S}}\leq C\sqrt{\mathbb{E}_{\mathbb{P}}\Big[\big\|\delta^{i,N}_t\big\|_{2}^{2}\Big]}
        \stackrel{\eqref{lln}}\leq \frac{C}{\sqrt{N}}.
    \end{align*}}
\paragraph*{Estimate $(\star\star\star)_{12}^{j_1}$.} We now estimate the term $(\star\star\star)_{12}^{j_1}$ by analyzing the difference of two trace expressions involving non-commuting operators. Specifically, 
{{\small} \begin{align*}
    \Big|(\star\star\star)_{12}^{j_1}\Big| &= \Big|\frac{1}{cN}\sum_{i\neq j}^{c}w_{ji}\tr\Big(\Big[\sum_{l=1}^{N}\big(A_{ji_l} - A^{\gamma_{t,i_l}}_{j}\big),\boldsymbol{g}_{t,j}\Big]\rho_{t}^{cN}\Big)\Big|\\
    &\leq \frac{1}{N}\sum_{l=1}^{N}\Big|{ \underbrace{\tr\Big(\big(A_{ji_l} - A^{\gamma_{t,i_l}}_{j}\big)\boldsymbol{g}_{t,j}\rho_{t}^{cN}\Big)}_{(\alpha)^{i_l}} - \underbrace{\tr\Big(\boldsymbol{g}_{t,j}\big(A_{ji_l} - A^{\gamma_{t,i_l}}_{j}\big)\rho_{t}^{cN}\Big)}_{(\beta)^{i_l}}}\Big|,
\end{align*}}
where,
{{\small} \begin{align*}
    (\alpha)^{i_l} &= \tr\Big(\big(A_{ji_l} - A^{\gamma_{t,i_l}}_{j}\big)\boldsymbol{g}_{t,j}\boldsymbol{\rho}_{t}^{cN}\Big) = \langle \Psi^{cN}_{t}, \big(A_{ji_l} - A^{\gamma_{t,i_l}}_{j}\big)\boldsymbol{g}_{t,j}\Psi^{cN}_{t}\rangle\\
    (\beta)^{i_l} &= \tr\Big(\boldsymbol{g}_{t,j}\big(A_{ji_l} - A^{\gamma_{t,i_l}}_{j}\big)\boldsymbol{\rho}_{t}^{cN}\Big) = \langle \Psi^{cN}_{t}, \boldsymbol{g}_{t,j}\big(A_{ji_l} - A^{\gamma_{t,i_l}}_{j}\big)\Psi^{cN}_{t}\rangle.
\end{align*}}
By using the fact that $\boldsymbol{g}_{t,h_k} + \boldsymbol{\gamma}_{t,h_k} = \mathbf{1}$,
{{\small}\begin{align*}
    (\alpha)^{i_l} &= \langle \Psi^{cN}_{t}, (\boldsymbol{g}_{t,j} + \boldsymbol{\gamma}_{t,j})(\boldsymbol{g}_{t,i_l} + \boldsymbol{\gamma}_{t,i_l})\big(A_{ji_l} - A^{\gamma_{t,i_l}}_{j}\big)\boldsymbol{g}_{t,j}(\boldsymbol{g}_{t,j} + \boldsymbol{\gamma}_{t,j})(\boldsymbol{g}_{t,i_l} + \boldsymbol{\gamma}_{t,i_l})\Psi^{cN}_{t}\rangle\\
    %(\alpha)^{i_l} 
    &= \langle \Psi^{cN}_{t}, (\boldsymbol{g }_{t,j} + \boldsymbol{\gamma}_{t,j})(\boldsymbol{g}_{t,i_l} + \boldsymbol{\gamma}_{t,i_l})\big(A_{ji_l} - A^{\gamma_{t,i_l}}_{j}\big)\boldsymbol{g}_{t,j}(\boldsymbol{g}_{t,i_l} + \boldsymbol{\gamma}_{t,i_l})\Psi^{cN}_{t}\rangle,\\
    (\beta)^{i_l} &= \langle \Psi^{cN}_{t}, \boldsymbol{g}_{t,j}(\boldsymbol{g}_{t,j} + \boldsymbol{\gamma}_{t,j})(\boldsymbol{g}_{t,i_l} + \boldsymbol{\gamma}_{t,i_l})\big(A_{ji_l} - A^{\gamma_{t,i_l}}_{j}\big)(\boldsymbol{g}_{t,j} + \boldsymbol{\gamma}_{t,j})(\boldsymbol{g}_{t,i_l} + \boldsymbol{\gamma}_{t,i_l})\Psi^{cN}_{t}\rangle\\
    %(\beta)^{i_l} 
    &= \langle \Psi^{cN}_{t},\boldsymbol{g}_{t,j}(\boldsymbol{g}_{t,i_l} + \boldsymbol{\gamma}_{t,i_l})\big(A_{ji_l} - A^{\gamma_{t,i_l}}_{j}\big)(\boldsymbol{g}_{t,j} + \boldsymbol{\gamma}_{t,j})(\boldsymbol{g}_{t,i_l} + \boldsymbol{\gamma}_{t,i_l})\Psi^{cN}_{t}\rangle.
\end{align*}}

We decompose $(\alpha)^{i_l}$ and $(\beta)^{i_l}$ into eight components each such that
{{\small} \begin{align*}
(\alpha)^{i_l} &= \sum_{k=1}^{8}(\alpha_{k})^{i_l} \;\; \text{and } \; (\beta)^{i_l} = \sum_{k=1}^{8}(\beta_{k})^{i_l}, 
\end{align*}}
where,
{\footnotesize
\begin{align*}
    (\alpha_{1})^{i_l} = \langle \Psi_{t}^{cN}, \boldsymbol{g}_{t,j}\boldsymbol{g}_{t,i_l} \big( A_{ji_l} - A^{\gamma_{t,i_l}}_{j} \big) \boldsymbol{g}_{t,j}\boldsymbol{g}_{t,i_l} \Psi_{t}^{cN} \rangle, 
    & \ 
    (\beta_{1})^{i_l} = \langle \Psi_{t}^{cN}, \boldsymbol{g}_{t,j}\boldsymbol{\gamma}_{t,i_l} \big( A_{ji_l} - A^{\gamma_{t,i_l}}_{j} \big) \boldsymbol{\gamma}_{t,j}\boldsymbol{\gamma}_{t,i_l} \Psi_{t}^{cN} \rangle, \\
    (\alpha_{2})^{i_l} = \langle \Psi_{t}^{cN}, \boldsymbol{g}_{t,j}\boldsymbol{g}_{t,i_l} \big( A_{ji_l} - A^{\gamma_{t,i_l}}_{j} \big) \boldsymbol{g}_{t,j}\boldsymbol{\gamma}_{t,i_l} \Psi_{t}^{cN} \rangle, 
    & \
    (\beta_{2})^{i_l} = \langle \Psi_{t}^{cN}, \boldsymbol{g}_{t,j}\boldsymbol{\gamma}_{t,i_l} \big( A_{ji_l} - A^{\gamma_{t,i_l}}_{j} \big) \boldsymbol{\gamma}_{t,j}\boldsymbol{g}_{t,i_l} \Psi_{t}^{cN} \rangle, \\
    (\alpha_{3})^{i_l} = \langle \Psi_{t}^{cN}, \boldsymbol{g}_{t,j}\boldsymbol{\gamma}_{t,i_l} \big( A_{ji_l} - A^{\gamma_{t,i_l}}_{j} \big) \boldsymbol{g}_{t,j}\boldsymbol{g}_{t,i_l} \Psi_{t}^{cN} \rangle, 
    & \ 
    (\beta_{3})^{i_l} = \langle \Psi_{t}^{cN}, \boldsymbol{g}_{t,j}\boldsymbol{\gamma}_{t,i_l} \big( A_{ji_l} - A^{\gamma_{t,i_l}}_{j} \big) \boldsymbol{g}_{t,j}\boldsymbol{\gamma}_{t,i_l} \Psi_{t}^{cN} \rangle, \\
    (\alpha_{4})^{i_l} = \langle \Psi_{t}^{cN}, \boldsymbol{g}_{t,j}\boldsymbol{\gamma}_{t,i_l} \big( A_{ji_l} - A^{\gamma_{t,i_l}}_{j} \big) \boldsymbol{g}_{t,j}\boldsymbol{\gamma}_{t,i_l} \Psi_{t}^{cN} \rangle, 
    & \ 
    (\beta_{4})^{i_l} = \langle \Psi_{t}^{cN}, \boldsymbol{g}_{t,j}\boldsymbol{\gamma}_{t,i_l} \big( A_{ji_l} - A^{\gamma_{t,i_l}}_{j} \big) \boldsymbol{g}_{t,j}\boldsymbol{g}_{t,i_l} \Psi_{t}^{cN} \rangle, \\
    (\alpha_{5})^{i_l} = \langle \Psi_{t}^{cN}, \boldsymbol{\gamma}_{t,j}\boldsymbol{g}_{t,i_l} \big( A_{ji_l} - A^{\gamma_{t,i_l}}_{j} \big) \boldsymbol{g}_{t,j}\boldsymbol{g}_{t,i_l} \Psi_{t}^{cN} \rangle, 
    & \ 
    (\beta_{5})^{i_l} = \langle \Psi_{t}^{cN}, \boldsymbol{g}_{t,j}\boldsymbol{g}_{t,i_l} \big( A_{ji_l} - A^{\gamma_{t,i_l}}_{j} \big) \boldsymbol{\gamma}_{t,j}\boldsymbol{\gamma}_{t,i_l} \Psi_{t}^{cN} \rangle, \\
    (\alpha_{6})^{i_l} = \langle \Psi_{t}^{cN}, \boldsymbol{\gamma}_{t,j}\boldsymbol{g}_{t,i_l} \big( A_{ji_l} - A^{\gamma_{t,i_l}}_{j} \big) \boldsymbol{g}_{t,j}\boldsymbol{\gamma}_{t,i_l} \Psi_{t}^{cN} \rangle, 
    & \ 
    (\beta_{6})^{i_l} = \langle \Psi_{t}^{cN}, \boldsymbol{g}_{t,j}\boldsymbol{g}_{t,i_l} \big( A_{ji_l} - A^{\gamma_{t,i_l}}_{j} \big) \boldsymbol{\gamma}_{t,j}\boldsymbol{g}_{t,i_l} \Psi_{t}^{cN} \rangle, \\
    (\alpha_{7})^{i_l} = \langle \Psi_{t}^{cN}, \boldsymbol{\gamma}_{t,j}\boldsymbol{\gamma}_{t,i_l} \big( A_{ji_l} - A^{\gamma_{t,i_l}}_{j} \big) \boldsymbol{g}_{t,j}\boldsymbol{g}_{t,i_l} \Psi_{t}^{cN} \rangle, 
    & \ 
    (\beta_{7})^{i_l} = \langle \Psi_{t}^{cN}, \boldsymbol{g}_{t,j}\boldsymbol{g}_{t,i_l} \big( A_{ji_l} - A^{\gamma_{t,i_l}}_{j} \big) \boldsymbol{g}_{t,j}\boldsymbol{\gamma}_{t,i_l} \Psi_{t}^{cN} \rangle, \\
    (\alpha_{8})^{i_l} = \langle \Psi_{t}^{cN}, \boldsymbol{\gamma}_{t,j}\boldsymbol{\gamma}_{t,i_l} \big( A_{ji_l} - A^{\gamma_{t,i_l}}_{j} \big) \boldsymbol{g}_{t,j}\boldsymbol{\gamma}_{t,i_l} \Psi_{t}^{cN} \rangle, 
    & \ 
    (\beta_{8})^{i_l} = \langle \Psi_{t}^{cN}, \boldsymbol{g}_{t,j}\boldsymbol{g}_{t,i_l} \big( A_{ji_l} - A^{\gamma_{t,i_l}}_{j} \big) \boldsymbol{g}_{t,j}\boldsymbol{g}_{t,i_l} \Psi_{t}^{cN} \rangle.
\end{align*}}

Several of these terms cancel pairwise:
{{\small} \begin{align*}(\alpha_{6})^{i_l} - (\beta_{2})^{i_l} = (\alpha_{4})^{i_l} - (\beta_{3})^{i_l} = (\alpha_{3})^{i_l} - (\beta_{4})^{i_l} = (\alpha_{2})^{i_l} - (\beta_{7})^{i_l} = (\alpha_{1})^{i_l} - (\beta_{8})^{i_l} = 0.
\end{align*}} 
We thus have
{{\small} \begin{align*}
    \Big|(\star\star\star)_{12}^{j_1}\Big| 
    &\leq \frac{2}{N}\sum_{l=1}^{N}\Big({(\triangle_{1}^{j,i_l})} + {(\triangle_{2}^{j,i_l})} +{(\triangle_{3}^{j,i_l})}\Big),
\end{align*}}
where
{{\small} \begin{align*}
{(\triangle_{1}^{j,i_l})} &:= \Big|\langle \Psi_{t}^{cN}, \boldsymbol{\gamma}_{t,j}\boldsymbol{\gamma}_{t,i_l} \big( A_{ji_l} - A^{\gamma_{t,i_l}}_{j} \big) \boldsymbol{\gamma}_{t,i_l}\boldsymbol{g}_{t,j} \Psi_{t}^{cN} \rangle\Big|,\\
{(\triangle_{2}^{j,i_l})} &:= \Big|\langle \Psi_{t}^{cN}, \boldsymbol{\gamma}_{t,j}\boldsymbol{\gamma}_{t,i_l} \big( A_{ji_l} - A^{\gamma_{t,i_l}}_{j} \big) \boldsymbol{g}_{t,j}\boldsymbol{g}_{t,i_l} \Psi_{t}^{cN} \rangle\Big|,\\
{(\triangle_{3}^{j,i_l})} &:= \Big|\langle \Psi_{t}^{cN}, \boldsymbol{\gamma}_{t,j}\boldsymbol{g}_{t,i_l} \big( A_{ji_l} - A^{\gamma_{t,i_l}}_{j} \big) \boldsymbol{g}_{t,j}\boldsymbol{g}_{t,i_l} \Psi_{t}^{cN} \rangle\Big|.
\end{align*}}

\paragraph*{Estimate $(\triangle_{1}^{j,i_l})$.} We first show that $(\triangle_{1}^{j,i_l})=0$: 
{{\small} \begin{align*}
    (\triangle_{1}^{j,i_l}) &= \Big|\langle \Psi_{t}^{cN}, \boldsymbol{\gamma}_{t,j}\boldsymbol{\gamma}_{t,i_l} \big( A_{ji_l} - A^{\gamma_{t,i_l}}_{j} \big) \boldsymbol{\gamma}_{t,i_l}\boldsymbol{g}_{t,j} \Psi_{t}^{cN} \rangle\Big|\\
    &= \Big|\langle \Psi_{t}^{cN}, \boldsymbol{\gamma}_{t,j}\big( \boldsymbol{\gamma}_{t,i_l}A_{ji_l}\boldsymbol{\gamma}_{t,i_l} - \boldsymbol{\gamma}_{t,i_l}A^{\gamma_{t,i_l}}_{j}\boldsymbol{\gamma}_{t,i_l} \big)\boldsymbol{g}_{t,j} \Psi_{t}^{cN} \rangle\Big|\\
    &= \Big|\langle \Psi_{t}^{cN}, \boldsymbol{\gamma}_{t,j}\big( \boldsymbol{\gamma}_{t,i_l}A_{ji_l}\boldsymbol{\gamma}_{t,i_l} - \boldsymbol{\gamma}_{t,i_l}A_{ji_l}\boldsymbol{\gamma}_{t,i_l} \big)\boldsymbol{g}_{t,j} \Psi_{t}^{cN} \rangle\Big|= 0,
\end{align*}}
where we used the fact that $\gamma_{i_l}A_{ji_l}\gamma_{i_l} = \gamma_{i_l}A_{j}^{\gamma_{i_l}}.$

\paragraph*{Estimate $(\triangle_{2}^{j,i_l})$.}
%{{\small} \begin{align*}
%    (\triangle_{2}^{j,i_l}) &= \Big|\langle \Psi_{t}^{cN}, \boldsymbol{\gamma}_{t,j}\boldsymbol{\gamma}_{t,i_l} \big( A_{ji_l} - A^{\gamma_{t,i_l}}_{j} \big) \boldsymbol{g}_{t,j}\boldsymbol{g}_{t,i_l} \Psi_{t}^{cN} \rangle\Big|,
%\end{align*}}
We define the weight functions
\begin{align*}
    f(k) = \frac{k}{N}, \quad  h(k) = \sqrt{\frac{k}{N}},
\end{align*}
and the associated operators
{{\small} \begin{align*}
    \hat{f}_i = \sum_{k=0}^{N}f(k)\mathbf{P}_{k,i}, \quad  \hat{h}_i = \sum_{k=0}^{N}h(k)\mathbf{P}_{k,i}, \quad i \in [c],
\end{align*}}
along with their inverses $\hat{f}^{-1}$ and $\hat{h}^{-1}$. Using again $\gamma_{i_l}A_{ji_l}\gamma_{i_l} = \gamma_{i_l}A_{j}^{\gamma_{i_l}} $, we have
{{\small} \begin{align*}
(\triangle_{2}^{j,i_l})   &= \Big|\langle \Psi_{t}^{cN}, \boldsymbol{\gamma}_{t,j}\boldsymbol{\gamma}_{t,i_l} A_{ji_l}\boldsymbol{g}_{t,j}\boldsymbol{g}_{t,i_l} \Psi_{t}^{cN} \rangle\Big|\\
&= \Big|\langle \Psi_{t}^{cN}, \boldsymbol{\gamma}_{t,j}\boldsymbol{\gamma}_{t,i_l} A_{ji_l}\hat{h}_{j}\hat{h}_{j}^{-1}\hat{h}_{i}\hat{h}_{i}^{-1}\boldsymbol{g}_{t,j}\boldsymbol{g}_{t,i_l} \Psi_{t}^{cN} \rangle\Big|.
\end{align*}}
Applying the technical Lemma \ref{technicallemma},
{{\small} \begin{align*}
(\triangle_{2}^{j,i_l})
&= \Big|\langle \Psi_{t}^{cN}, \boldsymbol{\gamma}_{t,j}\boldsymbol{\gamma}_{t,i_l} \hat{\tau_{1} h}_{i}\hat{\tau_{1}h}_{j}A_{ji_l}\hat{h}_{i}^{-1}\hat{h}_{j}^{-1}\boldsymbol{g}_{t,j}\boldsymbol{g}_{t,i_l} \Psi_{t}^{cN} \rangle\Big|.
\end{align*}}
Taking the expectation and applying the Cauchy–Schwarz inequality, 
{{\small} 
\begin{align*}
\mathbb{E}_{\mathbb{P}}\Big[(\triangle_{2}^{j,i_l}) \Big] 
&\stackrel{\text{C-S}}\leq \sqrt{\mathbb{E}_{\mathbb{P}}\Big[\Big\|A_{ji_l}\hat{\tau_{1}h_i}\hat{\tau_{1}h_j}\boldsymbol{\gamma}_{t,j}\boldsymbol{\gamma}_{t,i_l}\Psi_{t}^{cN}\Big\|_{2}^{2} \Big]}
\sqrt{\mathbb{E}_{\mathbb{P}}\Big[\Big\|\hat{h}^{-1}_{i}\hat{h}^{-1}_{j}\boldsymbol{g}_{t,j}\boldsymbol{g}_{t,i_l}\Psi_{t}^{cN}\Big\|_{2}^{2}\Big]}.
\end{align*}}

We estimate each term by

{{\small} 
\begin{align*}
&\mathbb{E}_{\mathbb{P}}\Big[\Big\|\hat{h}^{-1}_{i}\hat{h}^{-1}_{j}\boldsymbol{g}_{t,j}\boldsymbol{g}_{t,i_l}\Psi_{t}^{cN}\Big\|_{2}^{2}\Big] 
= \mathbb{E}_{\mathbb{P}}\Big[\langle \Psi_{t}^{cN},\hat{h}_{i}^{-2}\hat{h}^{-2}_{j}\boldsymbol{g}_{t,i}\boldsymbol{g}_{t,i_l}\Psi^{cN}\rangle \Big]\\
& \quad \stackrel{\text{C-S}}\leq \mathbb{E}_{\mathbb{P}}\Big[\|\hat{f}_i^{-1}\hat{f}_{j}^{-1}\|_{\text{op}}\Big\|\Psi_t^{cN}\boldsymbol{g}_{t,i} \Big\|_{2}^{2}\Big\|\Psi_t^{cN}\boldsymbol{g}_{t,j_l} \Big\|_{2}^{2} \Big]
\leq C\,\mathbb{E}_{\mathbb{P}}\Big[\mathcal{D}_{N}^{i}(t)\mathcal{D}_{N}^{j}(t)\Big] 
\leq  C\underbrace{\mathbb{E}_{\mathbb{P}}\Big[\mathcal{D}_N(t)\Big]}_{{x}},
\end{align*}}

{{\small} 
\begin{align*}
\mathbb{E}_{\mathbb{P}}\Big[\Big\|A_{ji_l}\hat{\tau_{1}h_i}\hat{\tau_{1}h_j}\boldsymbol{\gamma}_{t,j}\boldsymbol{\gamma}_{t,i_l}\Psi_{t}^{cN}\Big\|_{2}^{2} \Big] 
&\leq\! \mathbb{E}_{\mathbb{P}}\!\Big[\langle\Psi_{t}^{cN}, \hat{\tau_{1}h_{i}}\hat{\tau_{1}h_j}\boldsymbol{\gamma}_{t,i_l}\boldsymbol{\gamma}_{t,j}A_{ji_l}^{2}\hat{\tau_{1}h_{i}}\hat{\tau_{1}h_j}\boldsymbol{\gamma}_{t,i_l}\boldsymbol{\gamma}_{t,j}\Psi_{t}^{cN}\rangle \Big]\\
&\leq \|a\|_{2} \,\mathbb{E}_{\mathbb{P}}\Big[\langle \Psi_{t}^{cN},\hat{\tau_{1}h_{i}^2}\Psi_{t}^{cN}\rangle \langle \Psi_{t}^{cN},\hat{\tau_{1}h_{j}^2}\Psi_{t}^{cN}\rangle\Big]\\
&\leq \|a\|_{2}\,\mathbb{E}_{\mathbb{P}}\Big[\Big(\langle \Psi_{t}^{cN},\hat{f_{i}}\Psi_{t}^{cN}\rangle + \frac{1}{N} \Big)\Big(\langle \Psi_{t}^{cN},\hat{f_{j}}\Psi_{t}^{cN}\rangle + \frac{1}{N} \Big) \Big]\\
&\leq \|a\|_{2}\,\mathbb{E}_{\mathbb{P}}\Big[\Big(\mathcal{D}^{i}_{N}(t)+\frac{1}{N}\Big)\Big(\mathcal{D}^{j}_{N}(t)+\frac{1}{N}\Big)\Big]\\
&\leq \|a\|_2\,\mathbb{E}_{\mathbb{P}}\Big[\Big(\mathcal{D}_{N}(t)+\frac{1}{N}\Big)^2\Big] 
\leq \|a\|_{2}\underbrace{\mathbb{E}_{\mathbb{P}}\Big[\Big(\mathcal{D}_{N}(t) + \frac{3}{N} \Big)\Big]}_{{y}}.
\end{align*}}

Hence, 

{{\small} 
\begin{align*}
\mathbb{E}_{\mathbb{P}}\Big[(\triangle_{2}^{j,i_l}) \Big] 
&\stackrel{\sqrt{xy} \leq \frac{x+y}{2}}{\leq} 
C\Big(\mathbb{E}_{\mathbb{P}}\Big[\mathcal{D}_N(t)\Big] + \frac{3}{2N}\Big).
\end{align*}}

\paragraph*{Estimate $(\triangle_{3}^{j,i_l})$.} We have
{{\small} 
\begin{align*}
\mathbb{E}_{\mathbb{P}}\big[{(\triangle_{3}^{j,i_l})}\big] 
&= \mathbb{E}_{\mathbb{P}}\Big[\Big|\langle \Psi_{t}^{cN}, \boldsymbol{\gamma}_{t,j}\boldsymbol{g}_{t,i_l} \big( A_{ji_l} - A^{\gamma_{t,i_l}}_{j} \big) \boldsymbol{g}_{t,j}\boldsymbol{g}_{t,i_l} \Psi_{t}^{cN} \rangle\Big|\Big]\\
&\stackrel{\text{C-S}}\leq \mathbb{E}_{\mathbb{P}}\Big[\Big\|\boldsymbol{g}_{t,j}\boldsymbol{g}_{t,i_l}\Psi_t^{cN}\Big\|_{2}
\Big\|\big( A_{ji_l} - A^{\gamma_{t,i_l}}_{j}\big)\boldsymbol{\gamma}_{t,j}\boldsymbol{g}_{t,i_l}\Psi_{t}^{cN} \Big\|_{2}\Big]\\
&\leq \mathbb{E}_{\mathbb{P}}\Big[\Big\|\boldsymbol{g}_{t,j}\boldsymbol{g}_{t,i_l}\Psi_t^{cN}\Big\|_{2}
\Big\|\boldsymbol{\gamma}_{t,j}\big( A_{ji_l} - A^{\gamma_{t,i_l}}_{j}\big)^{2}\boldsymbol{\gamma}_{t,j}\Psi_{t}^{cN}\Big\|_{\text{op}}
\Big\|\boldsymbol{g}_{t,i_l}\Psi_{t}^{cN}\Big\|_{2} \Big]\\
&\leq C\,\mathbb{E}_{\mathbb{P}}\Big[\mathcal{D}_{N}^{i_l}(t) + \mathcal{D}_{N}^{j_1}(t)\Big] 
\leq C\,\mathbb{E}_{\mathbb{P}}\Big[\mathcal{D}_{N}(t)\Big],
\end{align*}
}

where we used the bound 
{{\small}$\Big\|\boldsymbol{\gamma}_{t,j}\big( A_{ji_l} - A^{\gamma_{t,i_l}}_{j}\big)^{2}\boldsymbol{\gamma}_{t,j}\Psi_{t}^{cN}\Big\|_{\text{op}} \leq C.$}

We thus have

{{\small} 
\begin{align*}
    \mathbb{E}_{\mathbb{P}}\Big[\Big|(\star\star\star)_{12}^{j_1} \Big| \Big] 
    &\leq C\Big(\mathbb{E}_{\mathbb{P}}[\mathcal{D}_N(t)] + \frac{1}{N}\Big), 
\end{align*}}

and,

{{\small} 
\begin{align*}
    \mathbb{E}_{\mathbb{P}}\Big[\Big|(\star\star\star)_{1}^{j_1} \Big| \Big] 
    &\leq \mathbb{E}_{\mathbb{P}}\Big[\Big|(\star\star\star)_{11}^{j_1} \Big| \Big] 
    + \mathbb{E}_{\mathbb{P}}\Big[\Big|(\star\star\star)_{12}^{j_1} \Big| \Big] 
    \leq C\Big(\mathbb{E}_{\mathbb{P}}[\mathcal{D}_N(t)] + \frac{1}{\sqrt{N}}\Big).
\end{align*}}

\subparagraph*{Estimate $(\star\star\star)_{2}^{j_1}$.}
By similar arguments, we obtain

{{\small} 
\begin{align*}
    \mathbb{E}_{\mathbb{P}}\Big[\Big|(\star\star\star)_{2}^{j_1} \Big| \Big] 
    &\leq C\Big(\mathbb{E}_{\mathbb{P}}[\mathcal{D}_N(t)] + \frac{1}{\sqrt{N}}\Big).
\end{align*}}

Putting all estimates together,

{{\small} 
\begin{align*}
    \mathrm{d}\mathbb{E}_{\mathbb{P}}[\mathcal{D}_{N}^{j}(t)] 
    &\leq C\Big(\|w^{cN} - w\|_{\square} + \frac{1}{\sqrt{N}} + \mathbb{E}_{\mathbb{P}}\big[\mathcal{D}_{N}(t)\big]\Big).
\end{align*}}

By summing over all classes, we have
{{\small} \begin{align*}
    \mathrm{d}\mathbb{E}_{\mathbb{P}}[\mathcal{D}_{N}^{}(t)] &\leq C\Big(\|w^{cN} - w\|_{\square} + \frac{1}{\sqrt{N}} + \mathbb{E}_{\mathbb{P}}\big[\mathcal{D}_{N}(t)\big]\Big),
\end{align*}}
and applying Gr\"onwall's inequality completes the proof. 

\end{proof}

\section{Application}\label{sec:Application}
In this section, we explore potential applications of the graphon quantum filtering system. 
To keep the presentation simple and avoid unnecessary technicalities, we restrict attention to a finite-dimensional Hilbert space $\mathbb{H}=\mathbb{C}^{d}$. 
In this case, $\mathcal{B}(\mathbb{H})=\mathcal{M}_d(\mathbb{C})$, the set of complex $d\times d$ matrices, which is isomorphic to $\mathbb{R}^{2d^2}$. Hence, the spaces can be identified as the standard Euclidean spaces. 
%As a result, we can invoke known results in Euclidean space without elaboration.

\subsection{Quantum state preparation}

An important objective in quantum control theory is quantum state preparation, which has been studied extensively, see e.g., \cite{stockton04deter,mirrahimiHandel07,sayrin11real,benoist14large,liang19exponential,cardona19continuous,amini24exponential}. This task is challenging due to the high dimensionality of the system, and various methods have been developed to manage this complexity (see e.g., \cite{handel05project,tezak17low,grigoletto25reduction}). Here we use the formalism of the mean-field limiting graphon as an approximation to the $N$-body system. 

As an illustration, we consider a concrete example of a graphon–qubit system with Ising-type interactions.

We begin by recalling the Pauli matrices and the ground and excited states of a two-level system:
\begin{center}
{{\small} \begin{align*}
    \rho_g := \begin{pmatrix}0 & 0 \\ 0 & 1\end{pmatrix}, \; 
    \rho_e := \begin{pmatrix}1 & 0 \\ 0 & 0\end{pmatrix}, \; \sigma_{x} := \begin{pmatrix}0 & -1 \\ 1 & 0\end{pmatrix}, \; \sigma_{z} := \begin{pmatrix}
        1 & 0 \\ 0 & -1 
    \end{pmatrix}. \end{align*}}
\end{center}

We fix the parameters of a system as follows:
{{\small} \begin{align*}
\tilde{H}&=0,\;\; \hat{H}=\sigma_x,\;\; 
L=\sigma_z,\;\; 
A^\rho=\tr(\rho\sigma_z)\sigma_z,\\
\eta &= 1, \;\; w(u,v) = uv, \;\; \gamma_{0,u} = \begin{pmatrix}
    \tfrac{1}{2} & 0 \\ 0 & \tfrac{1}{2}
\end{pmatrix},\\
\mathcal{J}_-&=[0,\frac{1}{2}),\;\; 
\mathcal{J}_+=[\frac{1}{2},1].
\end{align*}}

%\medskip

\begin{center}
\begin{tikzpicture}[scale=0.18]  % « pixel » = carré de 1 unit.
\begin{scope}[rotate=270,shift={(0,-16)}]

  % Grille externe (optionnelle)
  \draw (0,0) rectangle (16,16);

  % Boucles imbriquées sur les lignes (v) et colonnes (u)
  \foreach \i in {0,..., 15}  % colonne = u
  {%
    \foreach \j in {0,..., 15}  % ligne = v
    {%
      % coord. centre du bloc -> approx. (u,v)
      \pgfmathsetmacro{\u}{(\i+0.5)/16}
      \pgfmathsetmacro{\v}{(\j+0.5)/16}
      \pgfmathsetmacro{\shade}{100*\u*\v} % 0 = blanc ; 100 = noir

      % bloc (i,j) (attention : j = v = axe vertical)
      \fill[black!\shade] (\i,\j) rectangle ++(1,1);
    }%
  }%

  % Grille fine (optionnelle) :
  %\foreach \x in {1,...,\numexpr\N-1} \draw (\x,0) -- (\x,\N);
  %\foreach \y in {1,...,\numexpr\N-1} \draw (0,\y) -- (\N,\y);
\end{scope}
\end{tikzpicture}

\textbf{Schematic picture:} Graphon $w(u,v) = uv.$
\end{center}

%\medskip

Under an admissible control 
$\underline{\alpha}(t) = \{ \alpha_u(t), u \in I \}$, 
the graphon Belavkin system evolves, for all $u \in I$ and a.s. for $t \geq 0$, according to
\begin{equation}\label{eq:graphon-filter}
\begin{aligned}
\mathrm{d}\gamma_{t,u} &=
-\mathrm{i}\bigl[\alpha_u(t)\sigma_x
+u\Bigl(\int_{0}^{1}v\mathbb{E}_{\mathbb{P}}[z_{t,v}]\mathrm{d}v\Bigr)\sigma_z,\gamma_{t,u}\bigr]\mathrm{d}t  +\bigl(\sigma_z\gamma_{t,u}\sigma_z-\gamma_{t,u}\bigr)\mathrm{d}t \\
&\quad+\bigl(\sigma_z\gamma_{t,u}+\gamma_{t,u}\sigma_z-2z_{t,u}\gamma_{t,u}\bigr)\mathrm{d}W^{u}_t,
\end{aligned}
\end{equation}
where $(z_{t,u})_{t \geq 0}$ is the $z$-component of the qubit $\gamma_{t,u}$, defined as
$$ z_{t,u} := \tr(\gamma_{t,u}\sigma_{z}).$$

\subsubsection{Asymptotic behavior without control}
Setting $\underline{\alpha} \equiv 0$ turns into an
uncontrolled dynamics. We define the stochastic Lyapunov function $(V_{t,u})_{t \geq 0}$ by
\begin{align*}V_{t,u} := 1-z_{t,u}^2, \;\; \mathcal{V}_t(u,\omega) := V_{t,u}(\omega). 
\end{align*}
Applying It\^o's formula, we obtain
\begin{align*}
    \mathrm{d}z_{t,u}^2 &= 2z_{t,u}\mathrm{d}z_{t,u} + 4(1-z_{t,u}^2)^2\mathrm{d}t\\
    &= 4(1-z_{t,u}^2)^2\mathrm{d}t + 2z_{t,u}(1-z_{t,u}^2)\mathrm{d}W_t^{u},
\end{align*}
so that
\begin{align*}
    \mathrm{d}V_{t,u} &= -4(1-z_{t,u}^2)^2\mathrm{d}t - 2z_{t,u}\bigl(1-z_{t,u}^2\bigr)\mathrm{d}W^u_t.
\end{align*}
Taking the expectation, we have 
\begin{align*}
\mathbb{E}_{\mathbb{P}}[V_{t,u}] &= V_0 - 4 \int_{0}^{t}\mathbb{E}_{\mathbb{P}}[V_{s,u}^2]\mathrm{d}s.
\end{align*}
Then,
\begin{align*}
\mathbb{E}_{\mathcal{P}}[\mathcal{V}_{t}] &\leq V_0 - 4 \int_{0}^{t}\mathbb{E}_{\mathcal{P}}[\mathcal{V}_{s}^2]\mathrm{d}s.
\end{align*}

Since $0 \leq \mathcal{V}_{t} \leq 1 $, it follows that a.s. $\int_{0}^{\infty}\mathcal{V}_{s}^2\mathrm{d}s < \infty $. As $\mathcal{V}_{t}$ is a non-negative supermartingale, bounded below and with decreasing expectation, we conclude that a.s. $\mathcal{V}_{t} \to 0$. Thus $V_{t,u} \to 0$ for all $u \in I$, which implies that $\gamma_{t,u} \to \{\rho_e,\rho_g\}$.
Hence, the state of each particle asymptotically collapses to one of the projectors
$\rho_g$ or $\rho_e$.

\iffalse
\begin{remark}
We remark that the operators $L$ and $A^{\rho}$ commute, then they fulfills the non-demolition condition formulate in \cite{benoist14large}.  
\end{remark}
\fi 

\subsubsection{Deterministic state preparation}
For $u \in \mathcal{J}_{\pm}$, we adapt the feedback control strategy from \cite{tsumura07} and set
{{\small} \begin{equation}\label{eq:feedback-law}
\alpha_u(t)=
\begin{cases}
-8\mathrm{i}\tr\bigl([\sigma_x,\gamma_{t,u}]\rho_g\bigr)
+ 5\bigl(1-\tr(\gamma_{t,u}\rho_g)\bigr), & u\in\mathcal{J}_-,\\
-8\mathrm{i}\tr\bigl([\sigma_x,\gamma_{t,u}]\rho_e\bigr)
+ 5\bigl(1-\tr(\gamma_{t,u}\rho_e)\bigr), & u\in\mathcal{J}_+.
\end{cases}
\end{equation}
}
Under this feedback control, quantum state preparation becomes deterministic in the sense that particles with $u \in \mathcal{J}_-$ converge to the ground state $\rho_g$, while particles with $u \in \mathcal{J}_+$ converge to the excited state $\rho_e$.

\subsubsection{Simulations}
We simulate the dynamics of system \eqref{eq:graphon-filter} under two scenarios: the uncontrolled case with $\underline{\alpha} \equiv 0$ and the controlled case using the feedback strategy defined in \eqref{eq:feedback-law}. In both cases, we evaluate the fidelity $\mathcal{F}$ between the evolving quantum state $\gamma_{t,u}$ and the ground state $\rho_g$, defined by
 $$\mathcal{F}(\gamma,\rho):= \big(\tr(\sqrt{\sqrt{\rho}\gamma\sqrt{\rho}})\big)^{2}.$$ 
For particles with $u \in \mathcal{J}_{-}$, the fidelity with respect to $\rho_g$ is plotted in red, and for those with $u \in \mathcal{J}_{+}$ in blue.

The fidelity satisfies:
$$
\mathcal{F}(\gamma,\rho_g) = 
\begin{cases}
1, & \text{if } \gamma = \rho_g,\\
0, & \text{if } \gamma = \rho_e.
\end{cases}
$$

Figure~\ref{fig:reduction} illustrates the case of uncontrolled quantum state reduction, while Figure~\ref{fig:stabilization} shows the deterministic stabilization achieved under the control law~\eqref{eq:feedback-law}.

\begin{figure}[ht]
  \centering
  \includegraphics[width=10.0cm]{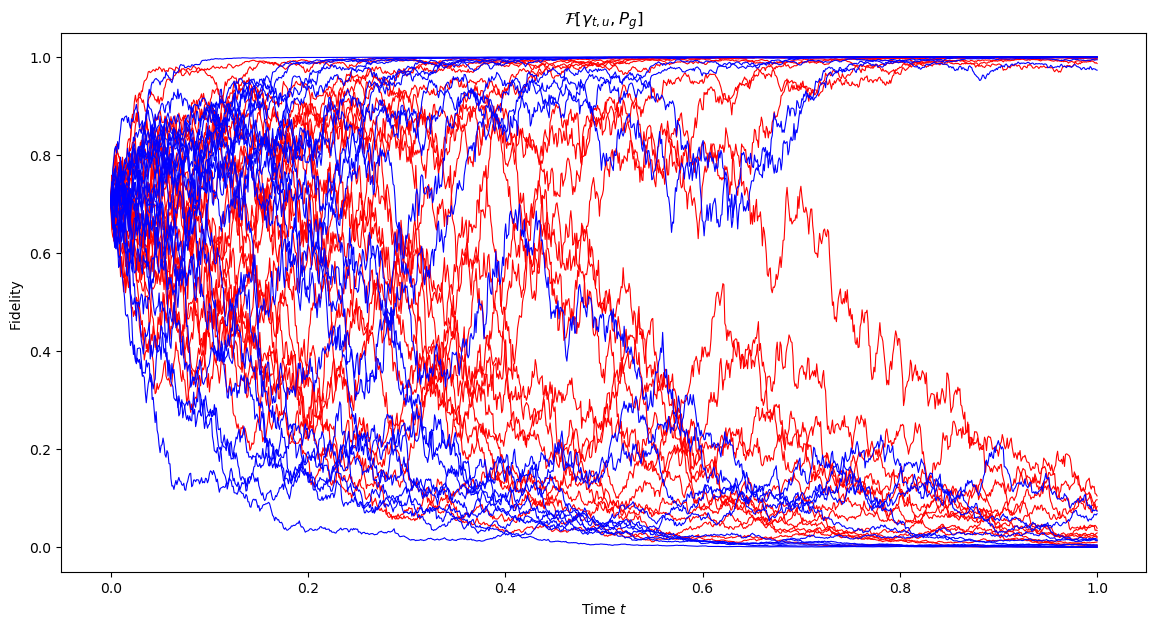}
  \caption{Asymptotic behavior under $\underline{\alpha} \equiv 0$.}
  \label{fig:reduction}
\end{figure}

\begin{figure}[ht]
  \centering
  \includegraphics[width=10.0cm]{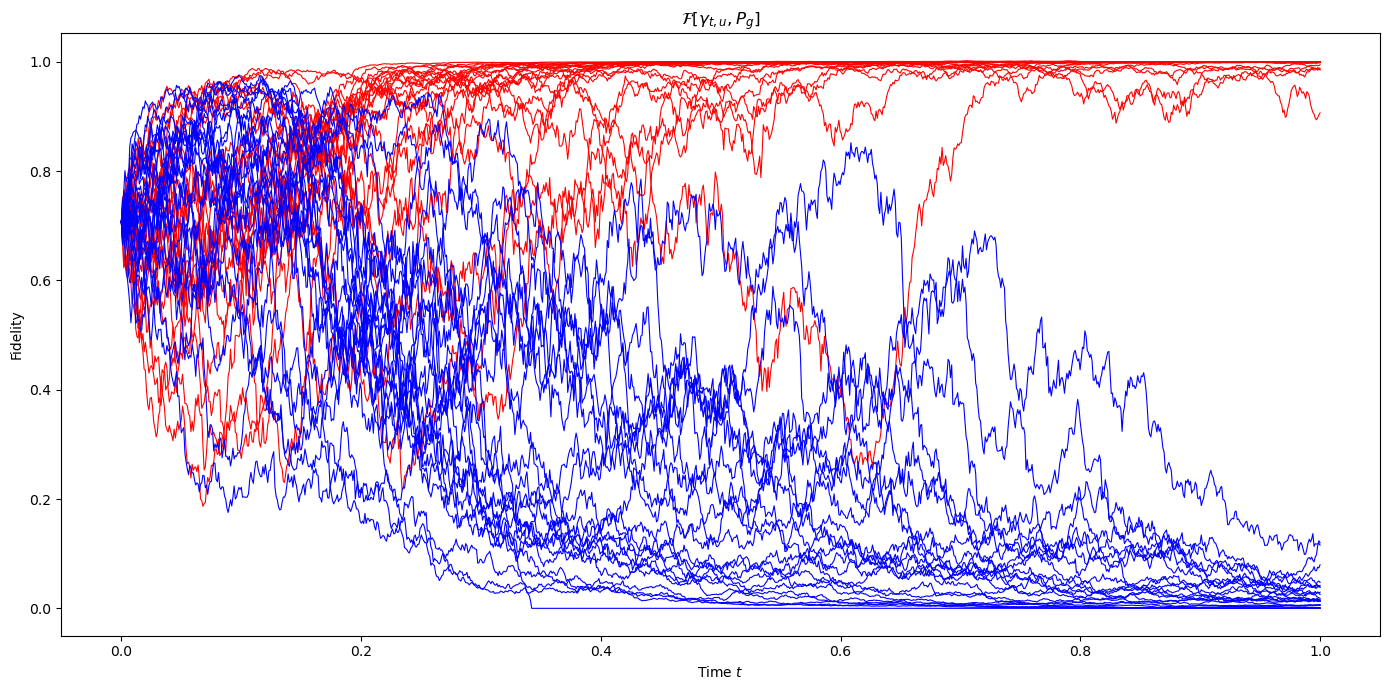}
  \caption{Quantum state preparation using the control law~\eqref{eq:feedback-law}.}
  \label{fig:stabilization}
\end{figure}

\iffalse
\begin{algorithm}[H]
  \SetAlgoLined
  \KwIn{$(T,w,(\gamma_{0,u})_{u \in I},H,L,A)$}
  \KwOut{$((\gamma_{t,u})_{u \in I})_{0 \leq t \leq T} $}

  \textbf{Initialization:} $k \gets 0$
  \textbf{Time step:} $\Delta t = \frac{T}{N}$\;

  \For{$k=1,\dots,\text{steps}$}{
    \For {$u=1,\dots,c$}{
    \For{$l=1,\dots,N$}{
    $ \tilde{\gamma}_{k+1,u} \gets \frac{1}{N}\sum_{l=1}^{N}{\gamma}_{k+1,u}^{l}$ \;
      $\gamma_{k+1,u}^{n} \gets \gamma_{k,u}^{n} + F(\dots)\Delta t + G(\dots)\sqrt{\Delta t}\mathcal{N}(0,1)$\;   
    }}
    $k \gets k+1$
  }
  \caption{Euler Scheme for the Graphon Quantum filter system}
\end{algorithm}
\fi

\subsection{Graphon quantum dynamic game}
Here we introduce a graphon quantum dynamic game with a continuum of particles indexed by $I$, in the same manner as in \cite{parise23graphon,aurell22sto}, where each particle goal is to choose the best admissible strategy. 

Let $\bar{\mathcal{U}}$ be the set of bounded measurable functions $\underline{\alpha}$, 
{{\small} 
\begin{align*}
    \underline{\alpha} &: I\times [0,T] \times \mathcal{S}(\mathbb{H}) \to [-U,U]\\
    &\quad (u,s,\rho) \longrightarrow \underline{\alpha}(u,s,\rho).
\end{align*}}
For each $u \in I$, we define $\mathcal{U}(u)$ as the set of bounded measurable functions,
{{\small} 
\begin{align*}
    {\alpha} &:  [0,T] \times \mathcal{S}(\mathbb{H}) \to [-U,U]\\
    &\quad (s,\rho) \longrightarrow {\alpha}(s,\rho).
\end{align*}}
A family $(\alpha_u)_{u \in I}$ of progressively measurable real-valued processes is called a control strategy, and the strategy is said to be admissible if there exists $\underline{\alpha} \in \bar{\mathcal{U}}$ such that $\alpha_{u} = \underline{\alpha}(u,.,.)$
for all $u \in I$.

Given an admissible control strategy $\underline{\alpha}$, the state of the controlled graphon quantum filtering system evolves as
{{\small} \begin{align*}
\mathrm{d}\gamma_{t,u}^{\underline{\alpha}} &= -\mathrm{i}\Big[\tilde{H} + \alpha_{u}(t)\hat{H} + \int_{0}^{1}w(u,v)A^{\mathbb{E}_{\mathbb{P}}[\gamma_{t,v}^{\underline{\alpha}}]}\mathrm{d}v, \gamma_{t,u}^{\underline{\alpha}} \Big]\mathrm{d}t + \Big(L\gamma_{t,u}^{\underline{\alpha}}L^{\dagger} - \frac{1}{2}\{L^{\dagger}L, \gamma_{t,u}^{\underline{\alpha}}\}\Big)\mathrm{d}t\\
&\quad\quad + \sqrt{\eta}\Big(L\gamma_{t,u}^{\underline{\alpha}} + \gamma_{t,u}^{\underline{\alpha}}L^{\dagger} - \tr\big((L+L^{\dagger})\gamma_{t,u}^{\underline{\alpha}}\big)\gamma_{t,u}^{\underline{\alpha}}\Big)\mathrm{d}W_t^{u}, \quad \quad \quad u \in [0,1].
\end{align*}} 
The expected cost for particle $u$ using admissible control $\beta \in \mathcal{U}(u)$, while other particles uses the  control strategy $\underline{\alpha}$ is denoted $\mathcal{J}_{u}$ and defined by:
{{\small}
\begin{align*}
    \mathcal{J}_{u}(\beta,\mathcal{A}_u^{\underline{\alpha}}) := \mathbb{E}_{\mathbb{P}}\Big[\int_{0}^{T}C\big(\beta,\gamma_{s,u}^{(\beta,\underline{\alpha}_{-u})},\mathcal{A}_{s,u}^{\underline{\alpha}}\big)\mathrm{d}s  +  F(\gamma_{T}^{\underline{\alpha}},\mathcal{A}_{T,u}^{\underline{\alpha}})\Big],
\end{align*}}
where we adopt the notation $(\beta,\underline{\alpha}_{-u})$ for a control strategy 
{{\small}
\begin{align*}
    (\beta,\underline{\alpha}_{-u})^{v} = \begin{cases}
        \underline{\alpha}(v), \quad &\text{if } v \neq u\\
        \beta, \quad &\text{if } v = u.
    \end{cases}
\end{align*}}
The measurable functions $C : \mathbb{R}\times\mathcal{M}_d(\mathbb{C})\times\mathcal{M}_d(\mathbb{C}) \to \mathbb{R}, \; \text{and \;} F : \mathcal{M}_d(\mathbb{C})\times\mathcal{M}_d(\mathbb{C}) \to \mathbb{R}$ are the current and terminal cost, respectively. The graphon-weighted aggregate $\mathcal{A}_{u}^{\underline{\alpha}} := (\mathcal{A}_{t,u}^{\underline{\alpha}})_{0 \leq t \leq T}$ is  defined by
\begin{align*}
    \mathcal{A}_{t,u}^{\underline{\alpha}} := \int_{0}^{1}w(u,v)\gamma_{t,v}^{\underline{\alpha}}\lambda(\mathrm{d}v), \quad \forall t \in [0,T].
\end{align*}

Thanks to exact law of large numbers \eqref{ELLN}, we can write the aggregate in terms of $\mathbb{E}_{\mathbb{P}}[\gamma_{t,v}^{\underline{\alpha}}$],
\begin{align*}
    \mathcal{A}_{t,u}^{\underline{\alpha}} &= \int_{0}^{1}w(u,v)\mathbb{E}_{\mathbb{P}}[{\gamma_{t,v}^{\underline{\alpha}}}]\lambda(\mathrm{d}v), \quad \forall t \in [0,T].
\end{align*}

We next adopt the following notion of  Nash equilibrium.
\begin{definition}
An admissible strategy control $\underline{\hat{\alpha}} $ is a graphon quantum game Nash equilibrium if
\begin{align*}
\mathcal{J}_u(\hat{\alpha}_u,\mathcal{A}_u^{\underline{\hat{\alpha}}}) &\leq \mathcal{J}_u(\beta,\mathcal{A}_u^{\underline{\hat{\alpha}}}), \quad \beta \in \mathcal{U}(u), \; \forall u \in I.
\end{align*}
\end{definition}

If a strategy control $\underline{\hat{\alpha}} $ achieves a Nash equilibirum, this means that no particle  can do better minimization by unilaterally changing control.

The question of the existence of equilibria can be approached using classical arguments from stochastic optimal control, such as the stochastic maximum principle or the dynamic programming principle; see e.g., \cite{yong99stochastic,bensoussan92stochastic}. We leave this and related issues to future work.

\iffalse
Under standard assumptions—$U$ compact; $C$ continuous and convex in the control; $F$ continuous; Lipschitz/linear-growth conditions ensuring well-posedness of \eqref{eq:GQFG}; and continuity of the best-response map in the aggregate—one can fix an aggregate flow $(\mathcal{A}_{t,\cdot})_{t\in[0,T]}$, solve each player’s control problem via HJB/SMP to obtain a measurable selector $\alpha^\star_u$, and close the loop by defining a map $\Phi$ sending aggregates to the induced aggregates. Compactness and continuity of $\Phi$ yield a fixed point by Schauder, hence a graphon Nash equilibrium. Approximate Nash equilibria for large finite graphs then follow by standard discretization/graphon-consistency arguments.
\fi

%CONCLUSION 
{\section{Conclusion and perspectives}\label{sec:concludes} 
 We have introduced a framework for studying continuously observed systems composed of a large number of interacting quantum particles with heterogeneous interactions, and showed that the resulting graphon-based system of quantum filtering equations is well-posed.
 We established a convergence result in the case of block-wise interactions with bosonic-by-class structure. A natural direction for future research is to further investigate the mean-field limit for general non-exchangeable systems. 
 
 In addition, while our analysis assumes a sequence of dense graphs, this hypothesis can be relaxed to include moderately sparse graphs, using similar arguments as in \cite{bayraktar23graphon,crucianelli24interacting}. 
{This can be seen as a bond percolation problem, where the $N$-body Hamiltonian is modified as   {{\small} \begin{align*}
    \mathbf{H}^{N} &:= \sum_{q \in V_N}\mathbf{\tilde{H}}_{q} + \frac{1}{N\mathfrak{p}_N}\sum_{q \in V_N}\sum_{p > q} \xi^{N}_{pq}\mathbf{A}_{pq},
\end{align*}}
with percolation parameter $\mathfrak{p}_N$ a sequence of positive real numbers that controls the sparsity of the graph, and such that $\mathfrak{p}_N \to 0$ while $N\mathfrak{p}_N \to \infty$ as $N \to \infty$.}
 
 Moreover, all results in this paper can be directly adapted to the counting measurement case (see the corresponding equation in the introduction) when the coupling operator $L$ is unitary, i.e., $L^{\dagger} = L^{-1}.$ For more general coupling operator, however, the mean-field limit for $N$-particle system remains an open problem, see the  discussion in \cite[Remark 5]{kolokoltsov21law}. We leave this and related questions for future work.}

\appendix

\section{Exchangeable quantum particles}\label{sec:appA}

The question of indistinguishability of classical or quantum particles is addressed in many textbooks, see for example \cite{gough18quantumfield}. We briefly review here the mains ingredients in the quantum setting.

Intuitively,  two particles are said to be indistinguishable if exchanging them does not change the probability of outcomes. The idea can be formalized in the following way:
{{\small} \begin{align*} \underline{P}_{N}(E_{1}\times\dots\times E_{N}) = \underline{P}_{N}(E_{\sigma(1)}\times\dots\times E_{\sigma(N)}), \quad \forall (E_1,\dots,E_N) \in \mathfrak{E}^{N},
\end{align*}}
where $\underline{P}_{N}$ is a probability measure on $\mathfrak{E}^{N}$, with  $\mathfrak{E}^{N}$ being the $N-$fold Cartesian product of a measurable space $\mathfrak{E}$.

Let $\boldsymbol{\Psi}^{N}$ denote the wave-function of $N$ particles. This describes the state of the system. The modulus square $|\boldsymbol{\Psi}^{N}|^2$ is interpreted as the probability density of finding the particles at $(x_1, \dots, x_N)$:
{{\small} \begin{align*}\underline{P}_{N}(x_1,\dots,x_N) := |\boldsymbol{\Psi}^{N}(x_1,\dots,x_N)|^{2}.
\end{align*}}
 If the particles are indistinguishable, the probability density must be invariant under any permutation $\sigma$, 
{{\small} \begin{align*} \underline{P}_{N}(x_{\sigma(1)},\dots,x_{\sigma(N)}) &= \underline{P}_{N}(x_{1},\dots,x_{N}),\\
\Big|\boldsymbol{\Psi}^{N}(x_{\sigma(1)},\dots,x_{\sigma(N)})\Big|^2 &= \Big|\boldsymbol{\Psi}^{N}(x_{1},\dots,x_{N}) \Big|^{2}. 
\end{align*}}
From this constraint two types of wave-function (and particles) exist: 
\begin{itemize}
    \item Bosons: The wave-function is  symmetric:
    {{\small} \begin{align*}\boldsymbol{\Psi}^{N}(x_{\sigma(1)},\dots,x_{\sigma(N)}) = \boldsymbol{\Psi}^{N}(x_{1},\dots,x_{N}).\end{align*}}
    \item Fermions: The wave-function is anti-symmetric:
    {{\small} \begin{align*}\boldsymbol{\Psi}^{N}(x_{\sigma(1)},\dots,x_{\sigma(N)}) = (-1)^{\sigma}\boldsymbol{\Psi}^{N}(x_{1},\dots,x_{N}),\end{align*}}
\end{itemize}
where $(-1)^{\sigma}$ denotes the sign of permutation. 
\begin{remark}
In both cases of Bosons and Fermions, symmetry and anti-symmetry of the wave-function imply the symmetry of the density operator, i.e., 
{{\small} \begin{align*}
\boldsymbol{\rho}^{N}(x_1,\dots,x_N;y_1,\dots,y_N) &= \boldsymbol{\Psi}^{N}(x_1,\dots,x_N)\overline{\boldsymbol{\Psi}^{N}(y_1,\dots,y_N)},\\
\boldsymbol{\rho}^{N}(x_{\sigma(1)},\dots,x_{\sigma(N)};y_{\sigma(1)},\dots,y_{\sigma(N)}) &=   \boldsymbol{\Psi}^{N}(x_{\sigma(1)},\dots,x_{\sigma(N)})\overline{\boldsymbol{\Psi}^{N}(y_{\sigma(1)},\dots,y_{\sigma(N)})}\\
&= \boldsymbol{\Psi}^{N}(x_1,\dots,x_N)\overline{\boldsymbol{\Psi}^{N}(y_1,\dots,y_N)}\\
&= \boldsymbol{\rho}^{N}(x_1,\dots,x_N;y_1,\dots,y_N).
\end{align*}}
\end{remark}

\section{Link between the linear and non-linear forms}\label{sec:appB}
In this section, we provide a strong solution to the linear system \eqref{lineargraphonpara} and use it to construct a weak solution of the non-linear system \eqref{graphonpara}.

We start by stating  an estimate lemma, first derived in \cite{kolokoltsov25quantumstoeq}, which is useful for ensuring well-posedness in Liouville space. For completeness, we provide here a shorter proof that avoids the approximation arguments used in the original reference. 

\begin{lemma}[Estimate lemma]\label{lem:estim}
Let $A$ be a bounded self-adjoint operator and $B$ a bounded operator. Then:
\begin{itemize}
    \item[(i)] $2|\tr(ABAB^{\dagger})| \leq \tr(A^2(BB^\dagger + B^\dagger B))$,
    \item[(ii)] $|\tr(AB^{\dagger}AB^{\dagger} + ABAB)| \leq \tr(A^2(BB^\dagger + B^\dagger B))$.
\end{itemize}
\end{lemma}

\begin{proof}
For the first estimate, applying the Cauchy-Schwarz inequality, we get
{{\small} $$
\big| \tr(B^{\dagger}ABA) \big| \leq \sqrt{\tr(B^{\dagger}A^{2}B)} \sqrt{\tr(BA^{2}B^{\dagger})} = \sqrt{\tr(A^{2}BB^{\dagger})} \sqrt{\tr(A^{2}B^{\dagger}B)}.
$$}
Using the arithmetic-geometric mean inequality
$ab \leq \frac{a^2 + b^2}{2}, \; \text{for all real } a, b,$ we obtain 
{{\small} $$
\big|\tr(ABAB^\dagger)\big| \leq \tfrac{1}{2}\Bigg(\tr(A^2 B B^\dagger) + \tr(A^2 B^\dagger B)\bigg).
$$}

For the second estimate, use the triangular inequality:
{{\small}\begin{align*}
\Big|\tr(AB^{\dagger}AB^{\dagger} + ABAB)\Big| &\leq \Big|\tr(AB^{\dagger}AB^{\dagger})\Big| + \Big|\tr(ABAB)\Big|.
\end{align*}}
By the Cauchy-Schwarz inequality,
{{\small}\begin{align*}
\Big|\tr(ABAB)\Big| = \Big|\tr(AB(B^{\dagger}A)^{\dagger})\Big| \leq \sqrt{\tr(ABB^{\dagger}A)}\sqrt{\tr(B^{\dagger}A^{2}B)} = \tr(A^{2}BB^{\dagger}),
\end{align*}}
and similarly,
{{\small} $$\Big|\tr(AB^{\dagger}AB^{\dagger})\Big| \leq \tr(A^{2}B^{\dagger}B).
$$}Hence, 
{{\small} $$
\Big|\tr(AB^{\dagger}AB^{\dagger} + ABAB)\Big| \leq \tr(A^{2}(B^{\dagger}B + BB^{\dagger})).
$$}

\end{proof}

{Throughout this part we work on a Fubini extended probability space 
$
(I\times\Omega,\mathcal{V},\mathcal{Q})$ of the usual product space 
$(I\times\Omega,\mathcal{I}\otimes\mathcal{F},\lambda\otimes\mathbb{Q})$. 
This space carries a family of e.p.i. Wiener processes $
\mathbf{Y} = \{Y^{u} : u \in I\},
$
and we fix a bounded operator $L \in \mathcal{B}(\mathbb{H})$.} 

%\medskip
Under this standing framework, we state the following results.

\begin{proposition}[Linear system SDE on Liouville space]
\label{linearwellposdness}
Let $\mathbf{H} = \{ H_{u} : u \in I \}$ be a measurable family such that 
$$
t \mapsto H_{t,u} \in \mathcal{B}(\mathbb{H}), \qquad 
\sup_{0 \leq s \leq T, \, u \in I} \| H_{s,u} \| < \infty,
$$
and assume $H_{t,u} = H_{t,u}^{\dagger}$ for all $(t,u) \in [0,T]\times I$.

%$$
%   \sup_{0 \leq s \leq T, \; u \in I}\|H_{s,u}\| < \infty,
%$$
%and assume %$H_{t,u}=H_{t,u}^{\dagger}$. 
Fix  $\eta\in (0,1]$ and an initial  density operator $\rho_0 \in \mathcal{S}(\mathbb{H})$. Consider  the  Liouville–space valued stochastic differential equation
{{\small}\begin{align}\label{linear}
\mathrm{d}\mathrm{R}_t &= \mathcal{L}_t(\mathrm{R}_t)\mathrm{d}t + \sqrt{\eta}\mathcal{R}(\mathrm{R}_t)\mathrm{d}\mathbf{Y}_t,\\
    \mathrm{R}_{0} &= \rho_{0} \in \mathcal{B}_2(\mathbb{H}).\nonumber
\end{align}
}
The drift $\mathcal{L}_t$ is defined by 
{{\small} \begin{align*}
\mathcal{L}_{t}(R) = -\mathrm{i}[\mathbf{H}_{t},R] + \big(LR L^{\dagger} - \frac{1}{2}\{L^{\dagger}L,R\}\big),
\end{align*}}
and the diffusion $\mathcal{R}$ by 
\begin{align*}
\mathcal{R}_{}(R) = LR  + R{L^{\dagger}}.
\end{align*}
Coordinatewise, the equation reads
\begin{equation*}
\begin{aligned}
   \mathrm{d}\varrho_{t,u}
   &= -\mathrm i[H_{t,u},\varrho_{t,u}]\mathrm{d}t
      + \Bigl(L\varrho_{t,u}L^{\dagger}-\frac{1}{2}\{L^{\dagger}L,\varrho_{t,u}\}\Bigr)\mathrm{d}t
      + \sqrt\eta\bigl(L\varrho_{t,u}+\varrho_{t,u}L^{\dagger}\bigr)\mathrm{d}Y_t^{u},\\
   \varrho_{0,u}&=\rho_0 \in \mathcal{B}_2(\mathbb{H}). 
\end{aligned}
\end{equation*}
Then the equation has a unique global solution. If $\rho_{0} \in \mathcal{S}(\mathbb{H})$, then positivity is preserved and
\begin{align*}
    \tr(\varrho_{t,u}) &= 1 + \sqrt{\eta}\int_{0}^{t}\tr\big((L+L^{\dagger})\varrho_{s,u}\big)\mathrm{d}Y_s^{u}, \quad \forall u \in I.
\end{align*}
\end{proposition}

\begin{proof}
First, the existence and uniqueness are addressed directly via  standard arguments on stochastic differential equations in infinite-dimensional Hilbert spaces \cite{SDE14inf}.

Positivity is more technical. The factor $\eta \in (0,1]$ represents measurement efficiency and when $\eta < 1$ the observation is incomplete. From a mathematical point of view, one may enlarge the probability space $(I\times\Omega, \mathcal{V}, \mathcal{Q})$ to incorporate a complete observation. Then we can consider the following auxillary system:

\begin{align*} 
\mathrm{d}\bar{\varrho}_{t,u} &= -\mathrm{i}[H_t,\bar{\varrho}_{t,u}]\mathrm{d}t + \big(L\bar{\varrho}_{t,u}L^{\dagger} - \frac{1}{2}\{L^{\dagger}L,\bar{\varrho}_{t,u}\}\big)\mathrm{d}t + \big(L\bar{\varrho}_{t,u} + \bar{\varrho}_{t,u}L^{\dagger}\big)\mathrm{d}X_t^{u},
\end{align*}
where $X_t^{u} = \sqrt{\eta}Y_t^{u} + \sqrt{1 -\eta}Z_t^{u} $ and $(Z^{u}_t)_{t \geq 0}$ is an independent Wiener process.

The process $(\bar{\varrho}_{t,u})_{t \geq 0}$ restricted to $\sigma({Y^{u}_t})$ is equivalent to $({\varrho}_{t,u})_{t \geq 0}$, i.e., 
$$ \mathbb{E}_{\mathbb{Q}}[\bar{\varrho}_{t,u}|\sigma({Y}^u_t)] \sim {\varrho}_{t,u}, \; \forall t \geq 0.$$

Since $\bar{\varrho}_{0} = {\varrho}_{0} = \rho_{0}\in \mathcal{S}(\mathbb{H})$, there exists a sequence of elements in $\mathbb{H}$, $\{(\varphi_{k})_{k \geq 1}\}$ such that~\footnote{Every quantum state can be written as a convex combination of pure states.} 
\begin{align*}\rho_{0} &= \sum_{k=1}^{\infty}\lambda_{k}\ket{\varphi_{0,k}}\bra{\varphi_{0,k}},
\end{align*}
with $\{(\lambda_{k})_{k \geq 1}\} \in l^{2}(\mathbb{N},\mathbb{R}_{+})$ and $
\sum_{k=1}^{\infty}\lambda_{k} = 1$.

Consider a finite-dimensional truncation of the Hilbert space $\mathbb{H}_{d} := P_d\mathbb{H}$ of dimension $d$, and  a finite dimensional truncation  $\rho_{0}^{d}$ of  $\rho_{0}$, given by 

{{\small} \begin{align*} \rho_{0}^{d} = \sum_{k=1}^{d}\lambda_{k}^{d}\ket{\varphi_{0,k}^d}\bra{\varphi_{0,k}^d} \longrightarrow \rho_{0} = \sum_{k=1}^{\infty}\lambda_{k}\ket{\varphi_{0,k}}\bra{\varphi_{0,k}}.
\end{align*}}
Then $({\bar{\varrho}_{t,u}^{d}})_{t \geq 0}$ solves the following {matrix} valued stochastic differential equation~\footnote{We denote by the same symbols $H_{t,u}, L \in \mathcal{B}(\mathbb{H})$ and their projector in $\mathcal{B}(\mathbb{H}_d)$.}: 

\begin{align} \mathrm{d}{\bar{\varrho}_{t,u}^{d}} &= -\mathrm{i}[H_{t}, {\bar{\varrho}_{t,u}^{d}}]\mathrm{d}t + \Big(L{\bar{\varrho}_{t,u}^{d}}L^{\dagger} - \frac{1}{2}\{L^{\dagger}L, {\bar{\varrho}_{t,u}^{d}}\}\Big)\mathrm{d}t + \Big(L{\bar{\varrho}_{t,u}^{d}} + {\bar{\varrho}_{t,u}^{d}}L^{\dagger}\Big)\mathrm{d}X_t^{u},\nonumber\\
{\bar{\varrho}_{0,u}^{d}} &= \rho_{0}^{d} =  \sum_{k=1}^{d}\lambda_{k}^d\ket{\varphi_{0,k}^{d}}\bra{\varphi_{0,k}^{d}}\in \mathcal{S}(\mathbb{H}_d).
\label{finitecompletelinearbelavkin}
\end{align}
Consider the following {vector} valued (which represents pure states)  SDE: 
\begin{align}
\mathrm{d}\varphi_{t,u,k}^{d} &= -\mathrm{i} H_{t,u} \varphi_{t,u,k}^{d}\mathrm{d}t - \frac{1}{2}L^{\dagger}L\varphi_{t,u,k}^{d}\mathrm{d}t + L\varphi_{t,u,k}^{d}\mathrm{d}X_t^{u}, \; 1 \leq k \leq d.
\end{align}

By Itô formula, $\ket{\varphi_{t,u,k}^d}\bra{\varphi_{t,u,k}^d} $ is a  solution of \eqref{finitecompletelinearbelavkin}. By a classical result \cite[Proposition 2.8]{bar09}, there exists an {invertible} propagator $(\mathcal{U}_{t,u,k}^{d})_{t \geq 0}$ such that
$$ \varphi_{t,u,k}^{d} =  \mathcal{U}_{t,u,k}^{d}[\varphi_{0}^{d}].$$

The invertibility of $\mathcal{U}_{t,u,k}^{d} $ and $\rho_{0}^{d} \neq 0$ guarantee the positivity of $\bar{\varrho}^{d}_{t,u}$ for all $t \in [0,T]$. In order to show the positivity of $(\bar{\varrho}_t)_{t\geq 0}$, it is sufficient to extract  a sequence $(\bar{\varrho}_t^{d})_{t \geq 0}$ that converges in the weak-sense, i.e., for some sequence  $(d_j)_{j \geq 0}$,

$$ \forall \chi \in \mathbb{H}, \forall t \in [0,T], \langle \chi, (\bar{\varrho}_{t,u} - \bar{\varrho}_{t,u}^{d_j} )\chi\rangle \to 0.$$
To prove this, we show that
\begin{align}\label{keyestimate} \sup_{d} \mathbb{E}\Big[ \sup_{0 \leq t \leq T} |\bar{\varrho}_{t,u}^{d}| \Big] <  \infty.
\end{align}

To show \eqref{keyestimate}, let us denote
$\tr(\bar{\varrho}_{t,u}^{d}) = \sum_{k=1}^{d}\lambda_{k}^d\|\varphi_{t,u,k}^d\|^2$ and $\quad M_{t,k}^{d,u}:=\|\varphi_{t,u,k}^{d}\|^{2}$.  
From It\^o's formula,
$$
   \mathrm{d}M_{t,k}^{d,u}=2\Re\langle L\varphi_{t,u,k}^{(d)},\varphi_{t,u,k}^{(d)}\rangle\mathrm{d}X_t^{u},\qquad \langle M^{d,u}\rangle_t\leq 4\|L\|^{2}\int_0^{t}\|\varphi_{s,u,k}^{d}\|^{4}\mathrm{d}s.
$$
Using the Burkholder–Davis–Gundy inequality followed by Gr\"onwall's lemma, we obtain
$$
   \mathbb{E}_{\mathbb{Q}}\Big[\sup_{0 \leq s\leq T}\|\varphi_{s,u,k}^{d}\|^{2}\Big] \leq C<\infty,
$$
 and therefore
$$
   \sup_{d}\mathbb{E}_{\mathbb{Q}}\bigl[\sup_{0 \leq t \leq T}\tr(\bar\varrho_{t,u}^{d})\bigr]\le C.
$$
Then $\bar\varrho^{d_j}_{\cdot,u}$ converges weakly to a positive operator valued process $\bar\varrho_{\cdot,u}$.  Hence, $\bar\varrho_{t,u}$ is positive, and consequently so is $\varrho_{t,u}$.

In order to conclude, we have to show that $R_t \in L^2_{\mathcal{L}_2(\mathbb{H})}(I\times\Omega,\mathcal{Q})$. We proceed “à la Grönwall”:
{{\small} \begin{align*}
\mathrm{d}\tr(\varrho_{t,u}^{2}) &= \tr\big(-\mathrm{i}[H_{t,u},\varrho_{t,u}]\varrho_{t,u} - \mathrm{i}\varrho_{t,u}[H_{t,u},\varrho]\big)\mathrm{d}t\\
+ \tr\big((L\varrho_{t,u}L^{\dagger}& - \frac{1}{2}L^{\dagger}L\varrho_{t,u} - \frac{1}{2}\varrho_{t,u}L^{\dagger}L)\varrho_{t,u} 
+ \varrho_{t,u}(L\varrho_{t,u}L^{\dagger} - \frac{1}{2}L^{\dagger}L\varrho_{t,u} 
- \frac{1}{2}\varrho_{t,u}L^{\dagger}L)\big)\mathrm{d}t,\\
+ \eta\tr\big(&(L\varrho_{t,u} + \varrho_{t,u}L^{\dagger})(L\varrho_{t,u} + \varrho_{t,u}L^{\dagger})\big)\mathrm{d}t 
+ (\dots)\mathrm{d}Y_t^{u}\\[6pt]
\frac{\mathrm{d}\mathbb{E}_{\mathbb{Q}}[\tr(\varrho_{t,u}^{2})]}{\mathrm{d}t} 
&=\! \mathbb{E}_{\mathbb{Q}}\Biggl[2\tr\bigl(\varrho_{t,u}L\varrho_{t,u}L^{\dagger}\! - \!\varrho_{t,u}^{2}L^{\dagger}L\bigr) 
+ \eta\tr\bigl(L\varrho_{t,u}L\varrho_{t,u} + L^{\dagger}\varrho_{t,u}L^{\dagger}\varrho_{t,u} 
+ 2L^{\dagger}L\varrho_{t,u}^{2}\bigr)\Biggr],\\[6pt]
\frac{\mathrm{d}\mathbb{E}_{\mathbb{Q}}[\tr(\varrho_{t,u}^{2})]}{\mathrm{d}t} 
&=\! \mathbb{E}_{\mathbb{Q}}\!\Biggl[\tr\Bigl(2\varrho_{t,u}L\varrho_{t,u}L^{\dagger}
+ \eta\bigl(L\varrho_{t,u}L\varrho_{t,u} 
+ L^{\dagger}\varrho_{t,u}L^{\dagger}\varrho_{t,u}\bigr)\Bigr) 
- 2(1-\eta)\tr\bigl(L^{\dagger}L\varrho_{t,u}^{2}\bigr)\Biggr],
\end{align*}}
and using Lemma~\ref{lem:estim}, 
{{\small} \begin{align*}
\mathrm{d}\mathbb{E}_{\mathbb{Q}}[\tr(\varrho_{t,u}^{2})] 
&\leq \mathbb{E}_{\mathbb{Q}}\Biggl[(1 + \eta)\tr\bigl(\varrho_{t,u}^{2}\bigl(L^{\dagger}{L} + {L}{L}^{\dagger}\bigr)\bigr) 
+ 2(1 - \eta)\tr\bigl(L^{\dagger}L\varrho_{t,u}^{2}\bigr)\Biggr]\mathrm{d}t\\
&\leq 2(1+\eta)\|L\|^{2}\mathbb{E}_{\mathbb{Q}}\bigl[\tr(\varrho_{t,u}^{2})\bigr]\mathrm{d}t
+ 4(1-\eta)\|L\|^{2}\mathbb{E}_{\mathbb{Q}}\bigl[\tr(\varrho_{t,u}^{2})\bigr]\mathrm{d}t\\
&\leq (6 - 2\eta)\|L\|^{2}\mathbb{E}_{\mathbb{Q}}\bigl[\tr(\varrho_{t,u}^{2})\bigr]\mathrm{d}t.
\end{align*}}
Thanks to Grönwall's inequality,
{{\small}
\begin{align*}
    \mathbb{E}_{\mathbb{Q}}[\tr(\rho_{t,u}^2)] &\leq e^{\big((6-2\eta)\|L\|^2\big)t}.
\end{align*}
}
Then,
{{\small}
\begin{align*}
    \mathbb{E}_{\mathcal{Q}}[\|\mathrm{R}_t\|^2_2] &= \int_{I\times\Omega}\|\mathrm{R}_t(u,\omega)\|^2_2\mathcal{Q}(\mathrm{d}u,\mathrm{d}\omega),\\
    &= \int_{I}\mathbb{E}_{\mathbb{Q}}[\tr(\varrho_{t,u}^2)]\lambda(\mathrm{d}u),\\
    &\leq e^{\big((6-2\eta)\|L\|^2\big)T} < \infty.
\end{align*}
}

This completes the proof. 
\end{proof}

By the Girsanov theorem stated in Proposition \ref{girsanovlemma}, we establish a link between the linear and nonlinear equations.

\begin{proposition}\label{lineartononlinear}
Consider $\mathrm{R}_t = \{\varrho_{t,u} : u \in I\}$ a stochastic process valued on $\mathcal{L}_{2}(\mathbb{H})$, issued from the linear quantum stochastic master equation \eqref{linear}. Then the   process $\bar{\mathrm{R}}_t = \{\rho_{t,u} : u \in I\}$, with $ \rho_{t,u} := \frac{\varrho_{t,u}}{\tr(\varrho_{t,u})}$, is a solution to the non-linear quantum stochastic master equation
{{\small} \begin{align}
\mathrm{d}\rho_{t,u} =& -\mathrm{i}[H_{t,u},\rho_{t,u}]\mathrm{d}t + (L\rho_{t,u}L^{\dagger} - \frac{1}{2}\{L^{\dagger}L,\rho_{t,u}\})\mathrm{d}t \nonumber \\
&+ \sqrt{\eta}\Big(L\rho_{t,u} + \rho_{t,u}L^{\dagger}-\tr((L+L^{\dagger})\rho_{t,u})\rho_{t,u}\Big)\mathrm{d}W_t^{u},
\label{normalizebelavkin}
\end{align}}
where the  process $\mathbf{W} = \big\{W^{u} : u \in I \big\}$ is a   Wiener process under the new probability measure
$\tilde{\mathcal{P}}$, defined by the Girsanov transform
$$ \mathbf{W}_t := \mathbf{Y}_t - \sqrt{\eta}\int_{0}^{t}\tr\big((L+L^{\dagger})\bar{\mathrm{R}}_{s}\big)\mathrm{d}s.$$
The measure $\tilde{\mathcal{P}}$ has Radon-Nikodym derivative $\mathcal{E}_{t} : (u,\omega) \to \mathcal{E}_{t,u}(\omega)$  with respect to $\mathcal{Q}$, given by
{{\small} \begin{align*}
{\mathcal{E}}_{t,{u}} :=   \exp\Big\{ \int_{0}^{t}\sqrt{\eta}\tr\big((L+L^{\dagger})\rho_{s,u}\big)\mathrm{d}Y_s^{u} - \frac{\eta}{2}\int_{0}^{t}\tr^{2}\big((L+L^{\dagger})\rho_{s,u}\big)\mathrm{d}s\Big\}.
\end{align*}}
\end{proposition}

\begin{proof}
We proceed “\`a la It\^o”  to derive the nonlinear Equation \eqref{normalizebelavkin} from the linear dynamics:  
{{\small} \begin{align*}
&\mathrm{d}\Big(\frac{\varrho_{t,u}}{\tr(\varrho_{t,u})}\Big) = \frac{\mathrm{d}\varrho_{t,u}}{\tr(\varrho_{t,u})} - \frac{\varrho_{t,u}\mathrm{d}\tr(\varrho_{t,u})}{\tr^{2}(\varrho_{t,u})} - \Big(\mathrm{d}\varrho_{t,u}\Big)\Big(\frac{\mathrm{d}\tr(\varrho_{t,u})}{\tr^{2}(\varrho_{t,u})}\Big) + \frac{\varrho_{t,u}\tr^{2}((L+L^{\dagger})\varrho_{t,u})\mathrm{d}t}{\tr^{3}(\varrho_{t,u}^{})}\\
=& -\mathrm{i}[H_{t,u},\rho_{t,u}]\mathrm{d}t + (L\rho_{t,u}L^{\dagger} - \frac{1}{2}\{L^{\dagger}L,\rho_{t,u}\})\mathrm{d}t+ \sqrt{\eta}(L\rho_{t,u} + \rho_tL^{\dagger})\mathrm{d}Y_t^{u} \\
-& \sqrt{\eta}\tr((L+L^{\dagger})\rho_{t,u})\rho_{t,u}\mathrm{d}Y_t^{u} - {\eta}\tr((L+L^{\dagger})\rho_{t,u})(L\rho_{t,u} + \rho_{t,u}L^{\dagger})\mathrm{d}t + \eta\tr^{2}((L+L^{\dagger})\rho_{t,u})\rho_{t,u}\mathrm{d}t\\
=& \mathcal{L}_t(\rho_{t,u})\mathrm{d}t + \sqrt{\eta}\Big(L\rho_{t,u} + \rho_{t,u}L^{\dagger}-\tr((L+L^{\dagger})\rho_t)\rho_{t,u}\Big)\mathrm{d}Y_t^{u}\\
&\quad -  \sqrt{\eta}\Big(L\rho_{t,u} + \rho_{t,u}L^{\dagger}-\tr((L+L^{\dagger})\rho_{t,u})\rho_{t,u}\Big)\sqrt{\eta}\tr((L+L^{\dagger})\rho_{t,u})\mathrm{d}t\\
=& \mathcal{L}_{t,u}(\rho_{t,u})\mathrm{d}t + \sqrt{\eta}\Big(L\rho_{t,u} + \rho_tL^{\dagger}-\tr((L+L^{\dagger})\rho_{t,u})\rho_{t,u}\Big)\underbrace{(\mathrm{d}Y_t^{u} 
- \sqrt{\eta}\tr((L+L^{\dagger})\rho_{t,u})\mathrm{d}t)}_{\mathrm{d}W_t^{u}},
\end{align*}}
where $\mathcal{L}_{t,u}(\rho) = -\mathrm{i}[H_{t,u},\rho]+ (L\rho L^{\dagger} - \frac{1}{2}\{L^{\dagger}L,\rho\}).$

To apply Proposition~\ref{girsanovlemma} and show that the process
{{\small} \begin{align*}
\mathbf{W}_t^{} := \mathbf{Y}_t^{} - \sqrt{\eta}\int_{0}^{t}\tr\big((L+L^{\dagger})\bar{\mathrm{R}}_{s}\big)\mathrm{d}s
\end{align*}}
is a Wiener process under the new probability measure $\tilde{\mathcal{P}}$, it is sufficient to verify Novikov's condition.
Since $\Big|\tr\big((L+L^{\dagger}){\rho}_{s,u}\big)\Big| \leq 2\|L\|$ and $\lambda(I)=1$,
then
\begin{align*}
\mathbb{E}_{\mathcal{Q}}\big[\exp\{\frac{\eta}{2}\int_{0}^{T}\tr^2\big((L+L^{\dagger})\bar{\mathrm{R}}_{s}\big)\mathrm{d}s \} \big] &= \int_{I}\mathbb{E}_{\mathbb{Q}}\big[\exp\{\frac{\eta}{2}\int_{0}^{T}\tr^2\big((L+L^{\dagger})\rho_{s,u}\big)\mathrm{d}s \} \big]\lambda(\mathrm{d} u)\\ 
&\leq \exp\Big\{2{\eta}{}\|L\|_{}^{2}T \Big\} < \infty.
\end{align*}
Thus, Novikov’s condition is satisfied, and hence the Girsanov theorem applies.
\end{proof}

\section*{Acknowledgments}
This work was supported by the ANR projects Q-COAST (ANR-19-CE48-0003) and IGNITION (ANR-21-CE47-0015). N-H.A. and S.C. thank the \emph{Institute for Mathematical and Statistical Innovation} (IMSI), supported by the National Science Foundation (Grant No. DMS-1929348), where part of this research was conducted. H. A. acknowledges support from NSF award No. 2534856. G.G. acknowledges financial support by \emph{Institut Europlace de Finance}. 

\bibliographystyle{plain}
\bibliography{bibliography}

\providecommand{\noopsort}[1]{}
\begin{thebibliography}{10}

\bibitem{alicki83nonlinear}
R.~Alicki and J.~Messer.
\newblock Nonlinear quantum dynamical semigroups for many-body open systems.
\newblock {\em Journal of Statistical Physics}, 32:299--312, 1983.

\bibitem{amini25brownian}
H.~Amini, N.H. Amini, S.~Chalal, and G.~Guo.
\newblock Brownian motion on the fubini extension space and applications.
\newblock {\em preprint arXiv:2509.12096}, 2025.

\bibitem{amini23graphon}
H.~Amini, Z.~Cao, and A.~Sulem.
\newblock Stochastic graphon games with jumps and approximate nash equilibria.
\newblock {\em SSRN}, 2023.

\bibitem{amini24exponential}
N.H. Amini, M.~Bompais, and C.~Pellegrini.
\newblock Exponential selection and feedback stabilization of invariant
  subspaces of quantum trajectories.
\newblock {\em SIAM Journal on Control and Optimization}, 62(5):2834--2857,
  2024.

\bibitem{aurell22sto}
A.~Aurell, R.~Carmona, and M.~Lauriere.
\newblock Stochastic graphon games: Ii. the linear-quadratic case.
\newblock {\em Applied Mathematics \& Optimization}, 85(3):39, 2022.

\bibitem{ayi24large}
N.~Ayi and N.P. Duteil.
\newblock Large-population limits of non-exchangeable particle systems.
\newblock {\em Active Particles, Volume 4}, pages 79--133, 2024.

\bibitem{BB91}
A.~Barchielli and V.P. Belavkin.
\newblock Measurements continuous in time and a posteriori states in quantum
  mechanics.
\newblock {\em Journal of Physics A: Mathematical and General}, 24(7):1495,
  1991.

\bibitem{bar09}
A.~Barchielli and M.~Gregoratti.
\newblock {\em Quantum trajectories and measurements in continuous time: the
  diffusive case}, volume 782.
\newblock Springer, 2009.

\bibitem{bar95const}
A.~Barchielli and A.S. Holevo.
\newblock Constructing quantum measurement processes via classical stochastic
  calculus.
\newblock {\em Stochastic Processes and their applications}, 58(2):293--317,
  1995.

\bibitem{bardos00weak}
C.~Bardos, F.~Golse, and N.J. Mauser.
\newblock Weak coupling limit of the $ n $-particle schr{\"o}dinger equation.
\newblock {\em Methods and Applications of Analysis}, 7(2):275--294, 2000.

\bibitem{bayraktar23graphon}
E.~Bayraktar, S.~Chakraborty, and R.~Wu.
\newblock Graphon mean field systems.
\newblock {\em The Annals of Applied Probability}, 33(5):3587--3619, 2023.

\bibitem{belavkin83theory}
V.P. Belavkin.
\newblock On the theory of controlling observable quantum systems.
\newblock {\em Avtomatika i Telemekhanika}, (2):50--63, 1983.

\bibitem{belavkin87non}
V.P. Belavkin.
\newblock Non-demolition measurement and control in quantum dynamical systems.
\newblock {\em Information complexity and control in quantum physics}, pages
  311--329, 1987.

\bibitem{belavkin92quantum}
V.P. Belavkin.
\newblock Quantum stochastic calculus and quantum nonlinear filtering.
\newblock {\em Journal of Multivariate analysis}, 42(2):171--201, 1992.

\bibitem{belavkin03eventum}
V.P. Belavkin.
\newblock Quantum trajectories, state diffusion, and time-asymmetric eventum
  mechanics.
\newblock {\em International Journal of Theoretical Physics}, 42:2461--2485,
  2003.

\bibitem{belavkin09cybernetics}
V.P. Belavkin.
\newblock Towards quantum cybernetics : Optimal feedback control in quantum
  bio-informatics.
\newblock {\em Quantum Bio-Informatics II: From Quantum Information to
  Bio-Informatics}, pages 30--41, 2009.

\bibitem{belavkin09qdpp}
V.P. Belavkin, A.~Negretti, and K.~M{\o}lmer.
\newblock Dynamical programming of continuously observed quantum systems.
\newblock {\em Physical Review A}, 79(2):022123, 2009.

\bibitem{benoist14large}
T.~Benoist and C.~Pellegrini.
\newblock Large time behavior and convergence rate for quantum filters under
  standard non demolition conditions.
\newblock {\em Communications in Mathematical Physics}, 331(2):703--723, 2014.

\bibitem{bensoussan92stochastic}
A.~Bensoussan.
\newblock {\em Stochastic control of partially observable systems}.
\newblock Cambridge University Press, 1992.

\bibitem{bet24weakly}
G.~Bet, F.~Coppini, and F.R. Nardi.
\newblock Weakly interacting oscillators on dense random graphs.
\newblock {\em Journal of Applied Probability}, 61(1):255--278, 2024.

\bibitem{boutenhandel07}
L.~Bouten, R.~Van~Handel, and M.R. James.
\newblock An introduction to quantum filtering.
\newblock {\em SIAM J. Control Optim.}, 46(6):2199–2241, dec 2007.

\bibitem{boutenhandel09}
L.~Bouten, R.~Van~Handel, and M.R. James.
\newblock A discrete invitation to quantum filtering and feedback control.
\newblock {\em SIAM Review}, 51(2):239--316, 2009.

\bibitem{caines19graphon}
P.E. Caines and M.~Huang.
\newblock Graphon mean field games and the gmfg equations: $\varepsilon$-nash
  equilibria.
\newblock {\em 2019 IEEE 58th conference on decision and control (CDC)}, pages
  286--292, 2019.

\bibitem{cardona19continuous}
G.~Cardona, A.~Sarlette, and P.~Rouchon.
\newblock Continuous-time quantum error correction with noise-assisted quantum
  feedback.
\newblock {\em IFAC-PapersOnLine}, 52(16):198--203, 2019.

\bibitem{carmonadelarue18I}
R.~Carmona and F.~Delarue.
\newblock {\em Probabilistic theory of mean field games with applications I}.
\newblock Springer, 2018.

\bibitem{chaintron22I}
L-P. Chaintron and A.~Diez.
\newblock Propagation of chaos: A review of models, methods and applications.
  i. models and methods.
\newblock {\em Kinetic and Related Models}, 15(6):895--1015, 2022.

\bibitem{chaintron22II}
L-P. Chaintron and A.~Diez.
\newblock Propagation of chaos: A review of models, methods and applications.
  ii. applications.
\newblock {\em Kinetic and Related Models}, 15(6):1017--1173, 2022.

\bibitem{chalal23mean}
S.~Chalal, N.H. Amini, and G.~Guo.
\newblock On the mean-field belavkin filtering equation.
\newblock {\em IEEE Control Systems Letters}, 2023.

\bibitem{coppini22note}
F.~Coppini.
\newblock A note on fokker--planck equations and graphons.
\newblock {\em Journal of Statistical Physics}, 187(2):15, 2022.

\bibitem{coppini25nonlinear}
F.~Coppini, A.~De~Crescenzo, and H.~Pham.
\newblock Nonlinear graphon mean-field systems.
\newblock {\em Stochastic Processes and their Applications}, page 104728, 2025.

\bibitem{crucianelli24interacting}
C.~Crucianelli and L.~Tangpi.
\newblock Interacting particle systems on sparse $ w $-random graphs.
\newblock {\em preprint arXiv:2410.11240}, 2024.

\bibitem{SDE14inf}
G.~Da~Prato and J.~Zabczyk.
\newblock {\em Stochastic equations in infinite dimensions}, volume 152.
\newblock Cambridge university press, 2014.

\bibitem{doherty00quantum}
A.C. Doherty, S.~Habib, K.~Jacobs, H.~Mabuchi, and S.M. Tan.
\newblock Quantum feedback control and classical control theory.
\newblock {\em Physical Review A}, 62(1):012105, 2000.

\bibitem{erdHos2010derivation}
L.~Erd{\H{o}}s, B.~Schlein, and H-T. Yau.
\newblock Derivation of the {Gross-Pitaevskii} equation for the dynamics of
  {Bose-Einstein} condensate.
\newblock {\em Annals of mathematics}, pages 291--370, 2010.

\bibitem{fagnola12sto}
F.~Fagnola and C.M. Mora.
\newblock Stochastic schr{\"o}dinger equations and applications to
  ehrenfest-type theorems.
\newblock {\em Latin American Journal of Probability and Mathematical
  Statistics}, 2012.

\bibitem{gough18introduction}
J.E. Gough.
\newblock An introduction to quantum filtering.
\newblock {\em preprint arXiv:1804.09086}, 2018.

\bibitem{gough05hamilton}
J.E. Gough, V.P. Belavkin, and O.G. Smolyanov.
\newblock Hamilton--jacobi--bellman equations for quantum optimal feedback
  control.
\newblock {\em Journal of Optics B: Quantum and Semiclassical Optics},
  7(10):S237, 2005.

\bibitem{gough18quantumfield}
J.E. Gough and J.~Kupsch.
\newblock {\em Quantum fields and processes: a combinatorial approach}.
\newblock Cambridge University Press, 2018.

\bibitem{grigoletto25reduction}
T.~Grigoletto, C.~Pellegrini, and F.~Ticozzi.
\newblock Quantum model reduction for continuous-time quantum filters.
\newblock {\em Annales Henri Poincar{\'e}}, pages 1--53, 2025.

\bibitem{holevo96}
A.S. Holevo.
\newblock On dissipative stochastic equations in a hilbert space.
\newblock {\em Probability theory and related fields}, 104:483--500, 1996.

\bibitem{jabin25mean}
P-E. Jabin, D.~Poyato, and J.~Soler.
\newblock Mean-field limit of non-exchangeable systems.
\newblock {\em Communications on Pure and Applied Mathematics}, 78(4):651--741,
  2025.

\bibitem{judd85law}
K.L. Judd.
\newblock The law of large numbers with a continuum of iid random variables.
\newblock {\em Journal of Economic theory}, 35(1):19--25, 1985.

\bibitem{kac56}
M.~Kac.
\newblock Foundations of kinetic theory.
\newblock {\em Proceedings of The third Berkeley symposium on mathematical
  statistics and probability}, 3:171--197, 1956.

\bibitem{knowles10mean}
A.~Knowles and P.~Pickl.
\newblock Mean-field dynamics: singular potentials and rate of convergence.
\newblock {\em Communications in Mathematical Physics}, 298:101--138, 2010.

\bibitem{kolokoltsov21law}
V.N. Kolokoltsov.
\newblock The law of large numbers for quantum stochastic filtering and control
  of many-particle systems.
\newblock {\em Theoretical and Mathematical Physics}, 208(1):937--957, 2021.

\bibitem{kolokoltsov22dynamic}
V.N. Kolokoltsov.
\newblock Dynamic quantum games.
\newblock {\em Dynamic Games and Applications}, 12(2):552--573, 2022.

\bibitem{kolokoltsov22qmfg}
V.N. Kolokoltsov.
\newblock Quantum mean-field games.
\newblock {\em The Annals of Applied Probability}, 32(3):2254--2288, 2022.

\bibitem{kolokoltsov25quantumstoeq}
V.N. Kolokoltsov.
\newblock On quantum stochastic master equations.
\newblock {\em Electronic Journal of Probability}, 30:1--21, 2025.

\bibitem{kolokoltsov25mathematical}
V.N. Kolokoltsov.
\newblock On the mathematical theory of quantum stochastic filtering equations
  for mixed states.
\newblock {\em preprint arXiv:2505.14605}, 2025.

\bibitem{lacker22label}
D.~Lacker and A.~Soret.
\newblock A label-state formulation of stochastic graphon games and approximate
  equilibria on large networks.
\newblock {\em Mathematics of Operations Research}, 2022.

\bibitem{lewin14derivation}
M.~Lewin, P.~T. Nam, and N.~Rougerie.
\newblock Derivation of {H}artree's theory for generic mean-field bose systems.
\newblock {\em Advances in Mathematics}, 254:570--621, 2014.

\bibitem{liang19exponential}
W.~Liang, N.H. Amini, and P.~Mason.
\newblock On exponential stabilization of n-level quantum angular momentum
  systems.
\newblock {\em SIAM Journal on Control and Optimization}, 57(6):3939--3960,
  2019.

\bibitem{liurockner15spde}
W.~Liu and M.~R{\"o}ckner.
\newblock {\em Stochastic partial differential equations: an introduction}.
\newblock Springer, 2015.

\bibitem{lonigro24liouville}
D.~Lonigro, A.~Hahn, and D.~Burgarth.
\newblock On the liouville--von neumann equation for unbounded hamiltonians.
\newblock {\em Open Systems \& Information Dynamics}, 31(04):2450018, 2024.

\bibitem{lovasz12large}
L.~Lov{\'a}sz.
\newblock {\em Large networks and graph limits}, volume~60.
\newblock American Mathematical Soc., 2012.

\bibitem{mirrahimiHandel07}
M.~Mirrahimi and R.~Van~Handel.
\newblock Stabilizing feedback controls for quantum systems.
\newblock {\em {SIAM} J. Control. Optim.}, 46(2):445--467, 2007.

\bibitem{mischler13kac}
S.~Mischler and C.~Mouhot.
\newblock Kac’s program in kinetic theory.
\newblock {\em Inventiones mathematicae}, 193:1--147, 2013.

\bibitem{misra77zeno}
B.~Misra and EC.G Sudarshan.
\newblock The zeno’s paradox in quantum theory.
\newblock {\em Journal of Mathematical Physics}, 18(4):756--763, 1977.

\bibitem{mora08basic}
C.M. Mora and R.~Rebolledo.
\newblock Basic properties of nonlinear stochastic schr{\"o}dinger equations
  driven by brownian motions.
\newblock {\em The Annals of Applied Probability}, 2008.

\bibitem{parise23graphon}
F.~Parise and A.~Ozdaglar.
\newblock Graphon games: A statistical framework for network games and
  interventions.
\newblock {\em Econometrica}, 91(1):191--225, 2023.

\bibitem{parthasarathy12}
K.R. Parthasarathy.
\newblock {\em An introduction to quantum stochastic calculus}, volume~85.
\newblock Birkh{\"a}user, 2012.

\bibitem{pearce1981mean}
P.~A. Pearce.
\newblock Mean-field bounds on the magnetization for ferromagnetic spin models.
\newblock {\em Journal of Statistical Physics}, 25(2):309--320, 1981.

\bibitem{pellegrini08diffusive}
C.~Pellegrini.
\newblock Existence, uniqueness and approximation of a stochastic
  schr{\"o}dinger equation: the diffusive case.
\newblock {\em The Annals of Probability}, 36(6):2332--2353, 2008.

\bibitem{pellegrini10jump}
C.~Pellegrini.
\newblock Existence, uniqueness and approximation of the jump-type stochastic
  schr{\"o}dinger equation for two-level systems.
\newblock {\em Stochastic Processes and their Applications}, 120(9):1722--1747,
  2010.

\bibitem{pellegrini10markov}
C.~Pellegrini.
\newblock Markov chains approximation of jump-diffusion stochastic master
  equations.
\newblock {\em Annales de l'IHP Probabilit{\'e}s et statistiques},
  46(4):924--948, 2010.

\bibitem{pickl11simple}
P.~Pickl.
\newblock A simple derivation of mean field limits for quantum systems.
\newblock {\em Letters in Mathematical Physics}, 97(2):151--164, 2011.

\bibitem{porat24pickl}
I.B. Porat and F.~Golse.
\newblock Pickl’s proof of the quantum mean-field limit and quantum
  klimontovich solutions.
\newblock {\em Letters in Mathematical Physics}, 114(2):51, 2024.

\bibitem{sayrin11real}
C.~Sayrin, I.~Dotsenko, X.~Zhou, B.~Peaudecerf, T.~Rybarczyk, S.~Gleyzes,
  P.~Rouchon, M.~Mirrahimi, H.~Amini, M.~Brune, J-M Raimond, and S.~Haroche.
\newblock Real-time quantum feedback prepares and stabilizes photon number
  states.
\newblock {\em Nature}, 477(7362):73--77, 2011.

\bibitem{searle24thermo}
A.~Searle and J.~Tindall.
\newblock Thermodynamic limit of spin systems on random graphs.
\newblock {\em Physical Review Research}, 6(1):013011, 2024.

\bibitem{spohn80kinetic}
H.~Spohn.
\newblock Kinetic equations from hamiltonian dynamics: Markovian limits.
\newblock {\em Reviews of Modern Physics}, 52(3):569, 1980.

\bibitem{stockton04deter}
J.K. Stockton, R.~Van~Handel, and H.~Mabuchi.
\newblock Deterministic dicke-state preparation with continuous measurement and
  control.
\newblock {\em Physical Review A}, 70(2):022106, 2004.

\bibitem{sun98almost}
Y.~Sun.
\newblock The almost equivalence of pairwise and mutual independence and the
  duality with exchangeability.
\newblock {\em Probability Theory and Related Fields}, 112(3):425--456, 1998.

\bibitem{sun06exact}
Y.~Sun.
\newblock The exact law of large numbers via fubini extension and
  characterization of insurable risks.
\newblock {\em Journal of Economic Theory}, 126(1):31--69, 2006.

\bibitem{sun09individual}
Y.~Sun and Y.~Zhang.
\newblock Individual risk and lebesgue extension without aggregate uncertainty.
\newblock {\em Journal of Economic Theory}, 144(1):432--443, 2009.

\bibitem{tezak17low}
N.~Tezak, N.H Amini, and H.~Mabuchi.
\newblock Low-dimensional manifolds for exact representation of open quantum
  systems.
\newblock {\em Physical Review A}, 96(6):062113, 2017.

\bibitem{tindall22}
J.~Tindall, A.~Searle, A.~Alhajri, and D.~Jaksch.
\newblock Quantum physics in connected worlds.
\newblock {\em Nature Communications}, 13(1):7445, 2022.

\bibitem{tsumura07}
K.~Tsumura.
\newblock Global stabilization of n-dimensional quantum spin systems via
  continuous feedback.
\newblock {\em 2007 IEEE 26th American Control Conference (ACC)}, pages
  2129--2134, 2007.

\bibitem{handel05project}
R.~Van~Handel and H.~Mabuchi.
\newblock Quantum projection filter for a highly nonlinear model in cavity qed.
\newblock {\em Journal of Optics B: Quantum and Semiclassical Optics},
  7(10):S226, 2005.

\bibitem{handel05}
R.~Van~Handel, J.K. Stockton, and H.~Mabuchi.
\newblock Modelling and feedback control design for quantum state preparation.
\newblock {\em Journal of Optics B: Quantum and Semiclassical Optics},
  7(10):S179, September 2005.

\bibitem{UMD15}
J.M.A.M Van~Neerven, M.C. Veraar, and L.~Weis.
\newblock Stochastic integration in banach spaces--a survey.
\newblock {\em Stochastic analysis: a series of lectures}, 68:297--332, 2015.

\bibitem{wiseman09quantum}
H.M. Wiseman and G.J. Milburn.
\newblock {\em Quantum measurement and control}.
\newblock Cambridge university press, 2009.

\bibitem{yong99stochastic}
J.~Yong and X.Y. Zhou.
\newblock {\em Stochastic controls: Hamiltonian systems and HJB equations},
  volume~43.
\newblock Springer Science \& Business Media, 1999.

\end{thebibliography}

\end{document}